\documentclass[11pt]{amsart}
\usepackage{amssymb}
\usepackage{amsmath} 
\usepackage{amsthm} 
\usepackage{epsfig}
\usepackage{graphicx}
\usepackage{color}
\usepackage{transparent}
\usepackage{soul}
\usepackage{subfig}
\usepackage{wrapfig}
\usepackage{blkarray}
\usepackage{tikz}
\usetikzlibrary{shapes.geometric, arrows}
\usepackage{forest}
\usetikzlibrary{trees}

\usepackage[colorinlistoftodos]{todonotes}
 
\numberwithin{equation}{section}
\numberwithin{table}{section}

\newtheorem*{thma}{Theorem A}

\newtheorem*{thmb}{Theorem B}

\newtheorem*{thmc}{Theorem C}

\newtheorem{thm}{Theorem}[section]
\newtheorem{lem}[thm]{Lemma}
\newtheorem{defn}[thm]{Definition}
\newtheorem{cor}[thm]{Corollary}
\newtheorem{rem}[thm]{Remark}

\newtheorem{prop}[thm]{Proposition}

\newcommand{\pcc}{\mathbf{P}^2(\mathbb{C})}
\newcommand{\prr}{\mathbf{P}^2(\mathbb{R})}

\newcommand{\R}{\mathbb{R}}
\newcommand{\C}{\mathbb{C}}

\newcommand{\GL}{\text{GL}}
\newcommand{\E}{\mathcal{E}}

\title{Real Rational Surface Automorphisms:\\ Positivity and Linearity} 

\author{Kyounghee Kim}
\address{Department of Mathematics\\
         Florida State University\\
         Tallahassee, FL, 32308}
\email{kkim6@fsu.edu}

\author{Insung Park}
\address{Department of Mathematics\\
Stony Brook University\\Stony Brook, NY,11794}
\email{insung.park@stonybrook.edu}


\begin{document}
\maketitle

\begin{abstract}
We study the real dynamics of a family of rational surface automorphisms obtained from quadratic birational maps of $\pcc$ that preserve a cuspidal cubic and whose critical orbits have lengths $(1,m,n)$ with $1+m+n\ge 10$. Passing to the real locus and cutting along the invariant cubic, we obtain a diffeomorphism of an orientable surface whose fundamental group is free.

Our key device is a finitely generated invariant, positive semigroup $S_{m,n}$ in the fundamental group on which an iterate of induced action acts by concatenation without cancellation. This positivity yields a nonnegative primitive transition matrix, so Perron-Frobenius theory supplies an explicit exponential growth rate $\lambda>1$ for the induced action on the fundamental group. Consequently, the real map has positive topological entropy. We package the combinatorics of the generators in a ``Core-Tail Induction Principle," which allows us to treat simultaneously seven orbit-data families with only finite base checks. Finally, using Bestvina-Handel and the Dehn-Nielsen-Baer correspondence, we show that the induced outer automorphism with $m+n$ odd is realized by a pseudo-Anosov homeomorphism of the cut surface.

\end{abstract}

\section{Introduction}\label{S:intro}
A compact complex surface admitting an automorphism with positive entropy is rare. By Cantat \cite{Cantat:1999}, there are only three kinds of such surfaces: complex $2$-tori, $\text{K}3$ surfaces and their unramified quotients, and rational surfaces.  Because the dimension of cohomology group of a rational surface can be any positive integer, the dynamics of rational surface automorphisms have a much wider palette. In fact, unlike $2$-tori or $\text{K}3$'s--whose Hodge structures and Picard lattices are highly constrained---rational surfaces can carry automorphisms with arbitrarily large cohomological action, infinitely many distinct entropy values. This abundance makes their dynamics especially rich: one sees a full spectrum of growth rates, intricate invariant curves or webs, and a delicate interplay between algebraic geometry and complex dynamics.

\medskip

Constructing rational surfaces with automorphisms exhibiting interesting dynamics goes back to Coble \cite{Coble:1939}. Since then, many new constructions and examples have appeared (see, e.g., \cite{Bedford-Kim:2006, Blanc:2008, McMullen:2007, Uehara:2010, Diller:2011, Kim:2023}). The set $\Lambda$ of all possible topological entropies of rational surface automorphisms is countably infinite. Using a cubic of arithmetic genus $1$ as an invariant curve, McMullen \cite{McMullen:2007} and Uehara \cite{Uehara:2010} showed that for every $\log\lambda\in\Lambda$ one can construct a rational surface automorphism realizing $h_{\mathrm{top}}=\log\lambda$ by blowing up $\mathbf{P}^2(\C)$ at a carefully chosen finite set of points on the smooth locus of the cubic. Moreover, their construction can be carried out over $\R$: the blow-up centers may be chosen real, so the resulting automorphism preserves a real two-dimensional invariant submanifold lying over $\mathbf{P}^2(\R)$.

\medskip
The aim of this article is to analyze the restriction of such automorphisms to the invariant real surface and thereby clarify the relationship between the complex and real dynamics. Let \(F\colon X\to X\) be a rational surface automorphism preserving a real two-dimensional submanifold \(X(\R)\subset X\) lying over \(\mathbf{P}^2(\R)\), and let \(F_\R:=F|_{X(\R)}\). Then
\[
  h_{\mathrm{top}}(F_\R)\;\le\;h_{\mathrm{top}}(F).
\]
By results of Gromov \cite{Gromov}, Yomdin \cite{Yomdin}, and Manning \cite{Manning:1975}, we have the chain of inequalities
\[
  \log \rho\,\bigl(F_{\R*}\big|_{H_1(X(\R))}\bigr)
  \;\le\;
  \log \rho\,\bigl(F_{\R*}\big|_{\pi_1(X(\R))}\bigr)
  \;\le\;
  h_{\mathrm{top}}(F_\R)
  \;\le\;
  h_{\mathrm{top}}(F)
  \;=\;
  \log \rho\,\bigl(F_{*}\big|_{H_2(X)}\bigr),
\]
where \(\rho\) denotes the spectral radius. The first inequality holds because \(H_1(X(\R))\) is the abelianization of \(\pi_1(X(\R))\); the second is due to Manning's lower bound on entropy via word growth; and the last equality is the Gromov-Yomdin formula for complex surface automorphisms.

\medskip

Computing \(\rho(F_{*}|_{H_2(X)})\) is a straightforward linear-algebra problem, so \(h_{\mathrm{top}}(F)\) is readily accessible. In contrast, the lack of equalities on the real side makes \(h_{\mathrm{top}}(F_\R)\) much harder to pin down. For quadratic automorphisms, Diller-Kim \cite{Diller-Kim} obtained a (very coarse) lower bound using homology growth, which did not even decide whether \(F_\R\) has positive entropy. Kim-Klassen \cite{Kim-Klassen} pursued a homotopy-based lower bound, but the computations were intricate and yielded only an estimate of homotopy growth.

\medskip
The main difficulty of computing the growth rate of the induced action on the fundamental group is because $\pi_1(X(\R))$ is nonabelian and there is no a priori linear model for the induced action $F_{\R*}$. We resolve this by cutting along the invariant cubic. The resulting ``cut surface" $X_{1,m,n}$ is orientable with one or two boundary components (Nicholls-Scherich-Shneidman \cite{Nicholls-Scherich-Shneidman:2023}), and $\pi_1(X_{1,m,n})$ is free of rank $1+m+n$. We then construct, inside $\pi_1(X_{1,m,n})$, a finitely generated semigroup $S_{m,n}$ that is invariant under an iterate $F^k_*$ and on which $F^k_*$ acts positively (no inverses appear). The action on $S_{m,n}$ is encoded by a nonnegative integer matrix $M_{m,n}$ that we show to be primitive. Perron-Frobenius theory supplies a simple eigenvalue $\lambda>1$ governing the exponential growth of word length, hence $h_\mathrm{top}(F_\R)>0$. A major technical point is organizing the proliferating generators as the orbit lengths $(1,m,n)$ vary. We introduce a Core-Tail Induction Principle: for each orbit-data family we fix a finite ``core" of generators and then adjoin ``tail" elements in a controlled way. This gives a uniform proof of invariance, positivity, and matrix primitivity for all large $m,n$, after checking finitely many base cases. We apply this to seven families distinguished by the relative sizes of the orbit lengths.

Finally, appealing to Bestvina-Handel's train-track theorem and the Dehn-Nielsen-Baer correspondence, we show that the outer automorphism $F_*^k$  of the free group is realized by a pseudo-Anosov homeomorphism of the cut surface. Thus, the restriction on the cut surface is pseudo-Anosov.

\medskip



\medskip
Let $F:X\to X$ be a rational surface automorphism obtained by lifting a quadratic birational map $f:\pcc \dasharrow \pcc$ that fixes a cuspidal cubic $c$ and preserves $\prr$. Assume the orbit data (defined in Section~\ref{S:rational}) are $(1,m\ge1,n\ge1)$ with cyclic permutation, and that all blow-up centers lie on the smooth locus of $c$ (we say $f$ properly fixes the cubic). Let $F_\R:X(\R) \to X(\R)$ be the restriction to the real surface and let $\hat F_\R:\hat X(\R) \to \hat  X(\R)$ be the further restriction to the cut surface obtained by removing (cutting along) the invariant cubic $\mathcal{C}$ (the strict transform of $c$). We compactify $\pi_1(\hat X(\mathbb{R}))$ by adjoining its free group boundary $\pi_1(\hat X(\mathbb{R}))$ (a Cantor set), so that every cyclically reduced element can be viewed as a periodic bi-infinite word.  For notational economy, we continue to write $\pi_1(\hat X(\mathbb{R}))$ for this compactification and $\hat F_{\mathbb{R}*}$ for the induced action on it.

\begin{thma}
There exists $k\in \{5,6,8\}$ and a subsemigroup $S \subset \pi_1(\hat X(\R))$ such that
\begin{enumerate}
\item $S$ is finitely generated;
\item the generating set of $S$ generates the whole group $ \pi_1(\hat X(\R))$;
\item $S$ is invariant under $\hat F_{\R*}^k$ (and $\hat F_{\R*}^k$ acts positively on $S$).
\end{enumerate}
\end{thma}

Using this semigroup, we obtain:

\begin{thmb}
The induced action $\hat F_{\R*}^k:  \pi_1(\hat X(\R)) \to  \pi_1(\hat X(\R))$ has exponential growth rate given by the spectral radius of a primitive non-negative integer transition matrix. In particular, this growth rate is $>1$; hence the real diffeomorphism $F_\R: X(\R) \to X(\R)$ has positive topological entropy.
\end{thmb}

Algebraic ``partial linearity" of the action on the fundamental group for surface homeomorphisms was studied by Birman-Series \cite{Birman-Series} and by Hamidi-Tehrani-Chen \cite{Hamidi-Chen:1996} using Thurston's $\pi_1$-train-track machinery. Their approach establishes the existence of a linear part but does not give an explicit, canonical choice of generators on which the action is linear. In contrast, our method produces a concrete $F_*^k$ -invariant seThis comparison is a bit awkward. Their theorem is not about specific outer automorphism, so that explicit canonical choice of generators for general actions wouldn't be possible.migroup $S\subset \pi_1$, yielding an explicit generating set and a straightforward transition matrix description.

\begin{rem}
In Theorem~B the spectral radius of the transition matrix is the $k$th power of the homotopy growth rate of $\hat F_{\R*}$ (here $k$ is the semigroup period).  Although one can, in principle, write a closed form for the characteristic polynomial of this matrix, comparing it to the homology growth rate of the complex map $F:X\to X$ (i.e.\ the maximal possible rate) is not straightforward.

Computer experiments indicate that, for the orbit data considered here, the homotopy growth rate of $\hat F_{\R*}$ actually attains this maximal value.  Because of the combinatorial explosion, we do not develop the full analysis for the general orbit data $(\ell,m,n)$ with all three lengths $>1$.  Nonetheless, computations suggest a different phenomenon there: when  $\ell, m ,n>1$ and $(\ell,m,n) \notin \{(2,n,n), n\ge 4\}$, the homotopy growth rate of $\hat F_{\R*}$ is strictly smaller than the complex maximal rate.  In fact, Diller-Kim \cite{Diller-Kim} proved that the family with orbit data $(3,3,n \ge 4)$ has non-maximal topological entropy.
\end{rem}

\medskip
For simplicity, we say that a surface homeomorphism is {\it pseudo-Anosov} if it is isotopic to a pseudo-Anosov diffeomorphism. By the Nielsen–Thurston classification, a surface homeomorphism is pseudo-Anosov if and only if it is neither periodic up to homotopy nor reducible, i.e., having a multicurve invariant up to homotopy \cite{Thurston:1988}. Although $F_\R$ is reducible on $X(\R)$ due to the invariant cubic, the cut map $\hat F_\R$ with $m+n$ odd is genuinely pseudo-Anosov:

\begin{thmc}
Suppose $m+n$ is odd. Then the diffeomorphism $\hat F_\R:\hat X(\R) \to \hat X(\R)$ is pseudo-Anosov.
\end{thmc}

\begin{rem}
When \(m+n\) is odd, the cut surface \(\hat X(\R)\) is orientable with a single boundary component. 
In Sections~\ref{S:linearity}--\ref{S:pseudo} we show that, for every \(k\ge 1\), the induced automorphism \(\hat F_{\R*}^k:\pi_1(\hat X(\R))\to\pi_1(\hat X(\R))\) is fully irreducible (irreducible for all powers). 
Theorem~C then follows directly from Bestvina-Handel \cite[Theorem~4.1]{Bestvina-Handel:1992} (yielding a train-track representative, and in our geometric setting a pseudo-Anosov homeomorphism of a compact surface with one boundary component).

When \(m+n\) is even, the cut surface has two boundary components. The induced map on \(\pi_1\) is still fully irreducible, so Bestvina-Handel \cite[Theorem~4.1]{Bestvina-Handel:1992} provides a train-track representative (hence, in the geometric case, a pseudo-Anosov representative on a surface with one boundary component). 
However, realizing the dynamics on the actual two-boundary cut surface requires additional control of the peripheral conjugacy classes and possible periodic Nielsen paths; this extra bookkeeping is absent in the one-boundary case and calls for a more delicate analysis. 
We will treat this case in future work.
\end{rem}

\medskip

In this article, we introduce a uniform \emph{core-tail} mechanism that certifies generation and dynamics for the families \(S_{m,n}\) on the cut surfaces \(X_{1,m,n}\).
On the group side, we isolate a finite generating set  \(G_{m*,n*}\) for $S_{m*,n*}$ at a base pair \((m_*,n_*)\) and verify \(\langle G_{m*,n*}\rangle=\pi_1(X_{1,m_*,n_*})\) by a finite computation in \textsf{GAP}/\textsf{FGA}.
For the successor step \((m,n)\mapsto (m',n')\), the ambient group splits as a free product
\[
\pi_1\!\big(X_{1,m',n'}\big)\;\cong\;\pi_1\!\big(X_{1,m,n}\big)\;*\;F_r,\qquad r\in\{1,2\},
\]
where \(r\) is the number of new basis letters introduced at that step.
By construction of the tail list, there is a cyclically reduced tail word \(w\) that involves \emph{exactly one} of the new letters, say \(x\), and whose abelianization has coefficient \(\pm1\) on \(x\).
This reduces the proof to a \emph{finite} verification per orbit-data family: the base case and one representative successor, which we certify in \textsf{GAP}/\textsf{FGA}.
On the dynamical side, we show that the induced automorphisms on \(\pi_1(\hat X(\R))\) are fully irreducible; in the one-boundary case (\(m+n\) odd) Bestvina-Handel then yields a pseudo-Anosov representative.
Conceptually, the paper couples a free-product/abelianization projection with Perron-Frobenius control of transition matrices, giving a reusable scheme for families in \(\mathrm{Out}(F_n)\) and surface dynamics, rather than a case-by-case computation.

\medskip

\noindent\textbf{Organization of the paper.}
Section~\ref{S:rational} recalls basic facts about quadratic birational maps of \(\pcc\) and their lifts to rational surface automorphisms.
Section~\ref{S:diffeo} introduces the cut surface, defines the cut map, and explains how to compute the induced action on \(\pi_1\) via reading curves.
In Section~\ref{S:seven} we partition the orbit data into seven families and record explicit formulas for the induced actions.
Section~\ref{S:homotopy} develops structural properties of the homotopy action.
Section~\ref{S:positivity} proves the existence of an invariant positive semigroup (part of Theorem~A).
Section~\ref{S:linearity} proves that the cyclically reduced generators produce the whole fundamental group and establishes linearity and exponential growth (Theorems~A and~B).
Finally, Section~\ref{S:pseudo} proves the pseudo-Anosov property (Theorem~C).

\subsection*{Acknowledgement} We would like to thank Jim Belk, Kevin Pilgrim, and Yvon Verberne for their valuable discussions.

\section{From Rational surface automorphisms}\label{S:rational}
Let
\[
  f \;=\; L_- \circ \sigma \circ (L_+)^{-1}
  \;:\;
  \pcc\dasharrow\pcc
\]
be a quadratic birational map, where \(L_\pm\in\text{GL}(3,\mathbb{C})\) are projective automorphisms and 
\[
  \sigma\colon [x_1:x_2:x_3]\;\mapsto\;[x_2x_3 : x_1x_3 : x_1x_2]
\]
is the standard Cremona involution.  The map \(\sigma\) is undefined precisely at the three coordinate points
\(\;e_1=[1:0:0],\,e_2=[0:1:0],\,e_3=[0:0:1]\), and its Jacobian vanishes along the three coordinate lines 
\(\ell_k=\overline{e_i,e_j}\) for \(\{i,j,k\}=\{1,2,3\}\).  Moreover, \(\sigma\) preserves each pencil of lines through \(e_k\), and is a diffeomorphism on the complement of \(\bigcup_{k=1}^3\ell_k\).

Accordingly, the indeterminacy locus $\mathcal{I}(f)$ (resp.\ $\mathcal{I}(f^{-1})$) of \(f\) (resp.\ \(f^{-1}\)) consists of the three points $p_i^+ = L_+(e_i)$ ($p_i^- = L_-(e_i)$) for $i=1,2,3.$
The exceptional locus of $f^\pm$ is given by the union of three lines $E_i^\pm = L_\pm (\ell_i)$.
\[ f: E_i^+ \setminus\mathcal{I}(f) \mapsto p_i^-, \qquad f^{-1}: E_i^- \setminus\mathcal{I}(f) \mapsto p_i^+,\quad i=1,2,3.\]
To simplify the notation, for a divisor $V$ of $\pcc$ we use $f(V)$ for the strict transform $f(V\setminus \mathcal{I}(f))$.

\medskip
\noindent\textbf{Orbit Data.}  
For each \(i=1,2,3\), set
\[
  n_i
  \;:=\;
  \min\bigl\{\,n\in\mathbb{Z}_{>0} : \dim\bigl(f^{n+1}(E_i^+)\bigr)=1 \bigr\}
  \;\in\;\mathbb{Z}_{>0}\cup\{\infty\},
\]
with the convention \(\min\emptyset=\infty\).  Since \(f(E_i^+)\) lies in the indeterminacy locus \(\mathcal I(f^{-1})\), one always has \(f^2(E_i^+)\neq f(f(E_i^+))\) (for example, $\sigma^2 \{ x_1=0\} = \{x_1 =0\}$ but $\sigma(\sigma \{x_1 =0\}) = \sigma [1:0:0]$ is not defined) 
 so any finite \(n_i\) satisfies \(n_i\ge1\).  We call the triple \((n_1,n_2,n_3)\) the \textit{orbit lengths} of \(f\).

\medskip

Whenever \(n_i<\infty\), the divisor \(E_i^+\) is carried after \(n_i\) further iterates to one of the points \(p_j^+\).  Hence, there is a well-defined map
\[
  \tau \;:\;\{\,i: n_i<\infty\}
  \;\longrightarrow\;\{1,2,3\},
  \quad
  f^{\,n_i+1}(E_i^+)=p_{\tau(i)}^+.
\]
We call \(\tau\) the \textit{orbit permutation}.  The pair consisting of the orbit-lengths \((n_1,n_2,n_3)\) together with the permutation \(\tau\) is the \textit{orbit data} of the birational map \(f\).  

\begin{thm}[Bedford-Kim {\cite{Bedford-Kim:2004}}]\label{T:orbitD}
Let \((n_1,n_2,n_3)\) and \(\tau\) be the orbit data of a quadratic birational map 
\[
  f\;=\;L_-\circ\sigma\circ L_+^{-1}\colon\pcc\dasharrow\pcc.
\]
If \(n_i<\infty\) for all \(i\), then \(f\) lifts to an automorphism 
\[
  F\colon X_{(n_1,n_2,n_3),\,\tau}\;\longrightarrow\;X_{(n_1,n_2,n_3),\,\tau}
\]
on the rational surface \(X_{(n_1,n_2,n_3),\,\tau}\) obtained by successively blowing up the finite orbit
\[
  O_f
  \;=\;
  \bigl\{f^j(E_i^+)\,\big|\;i=1,2,3,\;j=1,\dots,n_i\bigr\}.
\]
Moreover, the topological entropy of \(F\) is
\[
  h_{\mathrm{top}}(F)
  \;=\;
  \log\lambda,
\]
where \(\lambda>1\) is the largest real root of the polynomial \(\chi_{(n_1,n_2,n_3),\tau}(t)\) given by
\begin{equation}\label{E:DdegreeChar}
\begin{aligned}
&\chi_{(n_1,n_2,n_3),\tau}(t)\ =\ \\
&\begin{cases}
  (t-1)\bigl((t^{n_1}+1)(t^{n_2}+1)(t^{n_3}+1)+1\bigr)-(t^{n_1+n_2+n_3}-1),
  & \tau=(1\,2\,3),\\[6pt]
  (t-1)\bigl(t^{n_3}(t^{n_1}+1)(t^{n_2}+1)-t^{n_1}-t^{n_2}-2\bigr)
  -(t^{n_1+n_2}-1)(t^{n_3}-1),
  & \tau=(1\,2),\\[6pt]
  t\bigl(t^{n_1+n_2+n_3}-t^{n_1}-t^{n_2}-t^{n_3}+2\bigr)
  -\bigl(2t^{n_1+n_2+n_3}-t^{n_1+n_2}-t^{n_1+n_3}-t^{n_2+n_3}+1\bigr),
  & \tau=\mathrm{Id}.
\end{cases}
\end{aligned}
\end{equation}
\end{thm}

\medskip

The rational surface in Theorem \ref{T:orbitD} is obtained by blowing up \(\pcc\) at the finite orbit
\[
  O_f \;=\;\bigl\{\,f^j(E_i^+)\;\big|\;i=1,2,3,\;j=1,\dots,n_i\,\bigr\}.
\]
In general, it can be quite delicate to choose \(L_\pm\in\GL(3,\C)\) so that all three orbit lengths \(n_i\) are finite.  A convenient strategy is to arrange the points of \(O_f\) on a fixed invariant curve. It is known \cite{Nagata, Nagata2, Diller-Favre:2001,blancdynamical} that the set of all possible entropies of rational surface automorphisms is countably infinite. For every possible entropy one can even realize the construction over the reals: in our setting, there exist real automorphisms \(L_\pm\in\GL(3,\R)\) whose associated quadratic birational map preserves \(\prr\) and attains that entropy value (see McMullen \cite{McMullen:2007}, Diller \cite{Diller:2011}, Uehara\cite{Uehara:2010}, Kim\cite{Kim:2023}).

\medskip
In this article we restrict to the orbit data
\[
  (n_1,n_2,n_3)\in\mathbb{N}^3
  \quad\text{with}\quad
  n_1=1,\;\tau=(1\,2\,3),
  \quad\text{and}\quad
  1+n_2+n_3\ge10.
\]

\medskip
The first two conditions simplify the combinatorics, and the last guarantees positive entropy.  Since we have fixed \(\tau\) to be the 3-cycle \((1\,2\,3)\), we will henceforth omit it and write \(X_{n_1,n_2,n_3}\) in place of \(X_{(n_1,n_2,n_3),\tau}\), etc.

\medskip

Below we extract the consequences of Diller \cite{Diller:2011} that we use.
We pass to the inverse map so that the cusp fixed point is repelling and, to avoid an accumulation of minus signs, we adopt a mild reparameterization.
Accordingly, the formulas are stated for \(f^{-1}\) under this parameter convention.

\begin{thm}[Diller {\cite[Thm.~1.3]{Diller:2011}}]\label{T:Diller}
Let \((n_1,n_2,n_3)\) be orbit lengths with \(n_1=1\), \(n_2\neq n_3\), and set \(\tau=(1\,2\,3)\).  Then there exist real automorphisms \(L_\pm\in\GL(3,\R)\) such that the quadratic birational map
\[
  f \;=\; L_-\circ\sigma\circ L_+^{-1}
\]
on \(\pcc\) has orbit data \((1,n_2,n_3)\) with permutation \(\tau\), preserves the set of real points $\prr$, and preserves a cuspidal cubic curve \(\mathcal{C}\subset\pcc\).  Moreover:

\begin{enumerate}
  \item \(\mathcal{C}\) can be parametrized by
    \(\gamma\colon\C\cup\{\infty\}\to\mathcal{C}\), 
    with \(\gamma(0)\) the unique smooth fixed point of \(f\) on \(\mathcal{C}\) other than the cusp.

  \item There is a real multiplier \(\delta\neq\pm1\), namely a real root of the characteristic polynomial \(\chi_{(1,n_2,n_3),\tau}(t)\), such that
    \[
      f\bigl(\gamma(t)\bigr)
      =\gamma\bigl(\frac{1}{\delta}\,t\bigr),
      \quad t\in\R\cup\{\infty\}.
    \]

  \item The three blown-up points \(p_i^-\) lie on \(\mathcal{C}\) at
    \begin{equation}\label{E:bgamma}
      p_i^- \;=\;\gamma(t_i),
      \quad
      t_i
      =\delta^{n_i-1}\frac{1 + \delta^{\,n_{\tau(i)}} + \delta^{\,n_{\tau(i)}+n_{\tau^2(i)}}}
             { \delta^{\,1+n_2+n_3}-1},
      \quad i=1,2,3.
    \end{equation}

  \item The full base locus of \(f\) is
    \[
      O_f
      \;=\;
      \bigl\{\,
        \gamma(\delta^{-j}\,t_i)
        \;\big|\;
        1\le i\le3,\;0\le j\le n_i-1
      \bigr\}.
    \]
\end{enumerate}
\end{thm}

By Diller-Favre \cite{Diller-Favre:2001} and Blanc-Cantat \cite{blancdynamical}, whenever a quadratic birational map \(f\) is birationally conjugate to an automorphism of a rational surface, the polynomial \(\chi_{(n_1,n_2,n_3),\tau}(t)\) defined in \eqref{E:DdegreeChar} factors as the product of a Salem polynomial and cyclotomic factors.  Recall a Salem polynomial is a reciprocal integer polynomial with exactly two roots off the unit circle (namely \(\lambda\) and \(1/\lambda\)).  

Hence under the hypotheses of Theorem \ref{T:Diller}, the map \(f\) lifts to an automorphism 
\[
  F\colon X_{1,n_2,n_3}\to X_{1,n_2,n_3}
\]
and \(\chi_{(1,n_2,n_3)}(t)\) has exactly two real roots \(\delta=\lambda\) or \(\delta=1/\lambda\), where \(\lambda>1\).  Moreover, from \eqref{E:bgamma} all base points lie on the real locus of the invariant cubic \(\mathcal{C}\), so \(F\) preserves the real surface $X_\R \subset X_{(1,n_2,n_3)}$ over $\prr$,
and restricts to a real-analytic diffeomorphism 
\[
  F_\R\colon X_\R\;\longrightarrow\;X_\R.
\]

\medskip
\noindent\textbf{Notation.}  
For the remainder of this paper, we set
\[
 \check f_{(1,n_2,n_3)} := F_\R
  \quad\text{and}\quad
  \check X_{1,n_2,n_3} := X_\R.
\]
We also fix \(\delta=1/\lambda\), where \(\lambda>1\) is the Perron root (the largest real root) of \(\chi_{(1,n_2,n_3)}(t)\); note that choosing \(\lambda\) instead would correspond to replacing \(\check f\) by its inverse.  When no ambiguity arises, we further omit the orbit-length subscripts for simplicity.

\medskip

By abuse of notation, we continue to write \(E_i^\pm\) for both the lines in \(\prr\) and their strict transforms in \(\check X_{1,n_2,n_3}\), and likewise use \(\mathcal{C}\) for the real cubic in \(\prr\) and its strict transform.
For each \(i=1,2,3\) and \(0< j\le n_i\), let
\[
  \E_{i,j}
  \;\subset\;
  \check X_{1,n_2,n_3}
\]
be the exceptional curve lying over the blown-up point
\(\gamma(\delta^{1-j}\,t_i)=f^{j}(E_i^+)\). 

\vspace{1ex}
Since \(\check f_{(1,n_2,n_3)}\) and its inverse each send the divisor \(E_i^\pm\) only to the corresponding exceptional curve and preserve \(\mathcal{C}\), no other point can ever land on an exceptional curve under iteration.  More precisely:

\begin{lem}\label{L:limit}
Let \(p\in X_{1,n_2,n_3}\).  Then:

1. If \(p\notin\bigcup_{i=1}^3E_i^+\), then 
   \[
    \check f_{(1,n_2,n_3)}(p)\;\notin\;\bigcup_{i,j}\E_{i,j}.
   \]

2. For each \(i\) and \(1\le j<n_i\), 
   \[
    \text{ if }\ \  p\notin \E_{i,j},  \ \ \text{ then }\ \ 
    \check  f_{(1,n_2,n_3)}(p)\;\notin\;\E_{i,j+1}.
   \]

3. If 
   \[
     p\notin \E_{i,n_i},  \ \ \text{ then }\ \      \check f_{(1,n_2,n_3)}(p)\;\notin\;E_{\tau(i)}^-.
   \]

An analogous set of statements holds for the inverse map \(\check f_{(1,n_2,n_3)}^{-1}\).
\end{lem}

\section{To Diffeomorphisms on orientable surfaces with boundaries}\label{S:diffeo}

As noted in the Introduction, our primary goal is to compute the homotopy-growth rate of the real map \(\check f\).  Concretely, let 
\(\check X\) be the real rational surface obtained by blowing up \(n\) points in \(\prr\).  Then  
\[
  \pi_1(\check X)
  \;\cong\;
  \bigl\langle g_1,\dots,g_{n+1}
  \;\bigm|\;
  g_1^2\,g_2^2\cdots g_{n+1}^2=1
  \bigr\rangle,
\]
which is a finitely generated, non-abelian group with a single relation. Estimating word-length growth directly in such a group is quite challenging.  To bypass this difficulty, we instead pass to the \emph{cut surface} 
\[
  X
  \;=\;
  \check X\;\setminus\;\mathcal{C},
\]
obtained by cutting \(\check X\) along the invariant cubic \(\mathcal{C}\).  On \(X\), the fundamental group becomes free and admits simpler combinatorial control.

  \medskip

A result of Nicholls-Scherich-Shneidman \cite{Nicholls-Scherich-Shneidman:2023} says that if a simple closed curve on a non-orientable surface meets each cross-cap exactly once, then cutting along it produces an orientable surface. In our context, \(\check X_{1,n_2,n_3}\) is the connected sum of \(2+n_2+n_3\) copies of \(\prr\), and the invariant cubic \(\mathcal{C}\subset\prr\) lifts to a curve in \(\check X_{1,n_2,n_3}\) that meets each cross-cap exactly once.  Recall that any non-orientable surface of genus \(g\) is a sphere with \(g\) cross-caps. In our case $g=2+n_2+n_3$, so cutting \(\check X_{n_2,n_3}\) along \(\mathcal{C}\) produces an orientable surface.

\begin{thm}[{\cite[Prop.~3.5]{Nicholls-Scherich-Shneidman:2023}}]
Let \(X\) be a non-orientable surface of genus \(g\), and let \(C\subset X\) be a simple closed curve meeting each cross-cap exactly once.  Then \(X\setminus C\) is orientable.
\end{thm}

In particular, if \(n_2+n_3\) is odd then \(\mathcal{C}\) is one-sided in \(\check X_{1,n_2,n_3}\), and cutting along \(\mathcal{C}\) yields an orientable surface with one boundary component; if \(n_2+n_3\) is even, \(\mathcal{C}\) is two-sided, and the cut surface has two boundary components.  We denote
\[
  X_{1,n_2,n_3}
  \;:=\;
  \check X_{1,n_2,n_3}\,\setminus\,\mathcal{C},
\]
and note
\[
  \pi_1\bigl(X_{1,n_2,n_3}\bigr)
  \;\cong\;
  F_{\,1+n_2+n_3},
\]
a free group on \(1+n_2+n_3\) generators.

Finally, since \(\mathcal{C}\) is invaraint under $f$  (see Diller-Kim \cite{Diller-Kim}), by choosing $[\mathcal{C}]$ as one of the generators of $\pi_1(\check X_{1,n_2,n_3})$, we see that  the growth rates of the induced maps
\[
  \check f_{(1,n_2,n_3)*}\colon\pi_1(\check X_{1,n_2,n_3})\to\pi_1(\check X_{1,n_2,n_3})
  \quad\text{and}\quad
  f_{(1,n_2,n_3)*}\colon\pi_1(X_{1,n_2,n_3})\to\pi_1(X_{1,n_2,n_3})
\]
coincide. 
In this section we explain how to compute \(f_{(1,n_2,n_3)*}\) on the free group \(\pi_1(X_{1,n_2,n_3})\).

\subsection{Reading Curves}\label{SS:readingcurve}

To compute the induced action on \(\pi_1\bigl(X_{n_2,n_3}\bigr)\), we use the method of \emph{reading curves}, dual to the usual generating loops.  This technique was introduced by Cohen-Lustig \cite{Coh} and, for non-orientable surfaces, is carried out in detail by Kim-Klassen \cite{Kim-Klassen}. Reading curves provides a simple way to find the homotopy class for the given oriented simple closed curve.

\medskip

Choose a base point \(p\in X_{1,n_2,n_3}\).  A convenient generating set for \(\pi_1\bigl(X_{n_2,n_3}\bigr)\) is given by simple closed loops each of which passes exactly once through a single exceptional curve and otherwise avoids all others.  Concretely:

\begin{itemize}
\item[-] Denote again by \(\E_{i,j}\subset X_{1,n_2,n_3}\) the strict transform of the exceptional curve over \(\gamma(\delta^{-j} t_i)\).  
\item[-] Let \(b_1\) be a loop based at \(p\) that intersects \(\E_{1,1}\) exactly once and is disjoint from every other \(\E_{i,j}\).  
\item[-] Similarly, for \(j=1,\dots,n_2\), let \(c_j\) be a loop through \(\E_{2,j}\) only, and for \(j=1,\dots,n_3\), let \(a_j\) be a loop through \(\E_{3,j}\) only.  
\end{itemize}
When no confusion will arise, we denote both the simple loop (the ``generating curve") and its homotopy class by the same symbol \(a_j\), \(b_1\), or \(c_j\).

\medskip
The corresponding \emph{reading curves} are the exceptional divisors 
\(\mathcal{E}_{i,j}\subset X_{1,n_2,n_3}\), 
each of which is a one-point-removed simple closed curve, disjoint from all the others.  One can think of each boundary component of \(X_{1,n_2,n_3}\) as a marked point, so that the generating loops and reading curves form dual families on the orientable surface with marked points. 

\medskip
To determine how an oriented loop intersects a reading curve, we fix an orientation on each generating loop and assign a fixed sign to each side of the reading curve ---``+" to one side and ``-" to the other. Then, for any oriented loop crossing the reading curve, we can determine whether it intersects from the positive or negative side based on the direction of crossing. For simplicity of notation, we will also refer to the reading curve \(\mathcal{E}_{i,j}\) by the same symbol \(a_j\), \(b_1\), or \(c_j\) when it is clear from context which family is intended.

Concretely, to determine the homotopy class of any oriented simple closed curve \(c\subset X_{1,n_2,n_3}\) via reading curves, proceed as follows:

\begin{enumerate}
\item  Start at the base point \(p\) of \(X_{1,n_2,n_3}\) and traverse \(c\) once in its given orientation.
\item  Each time \(c\) crosses a reading curve-say the one denoted \(b_1\)-look at the local orientation:
\begin{itemize}
   \item[-] If \(c\) crosses into the ``positive" side of \(b_1\), record the symbol \(b_1\).
    \item[-] If \(c\) crosses into the ``negative" side of \(b_1\), record the symbol \(b_1^{-1}\).
\end{itemize}
\item  Continue around \(c\), concatenating these symbols in order.  The resulting word in the generators \(\{a_j, b_1, c_j\}\) (and their inverses) is precisely the homotopy class \([c]\in\pi_1(X_{1,n_2,n_3})\).
\end{enumerate}

Since the reading curves \(\{a_j, b_1, c_j\}\) are all disjoint and dual to the generating loops, this procedure unambiguously recovers the conjugacy class of the loop's word in the free group.


\subsection{Order of base locus}

To compute the homotopy class of the image \(f_*([c]) = [f(c)]\) of an oriented simple closed curve \(c\), we must record the sequence in which \(f(c)\) crosses the reading curves (i.e.\ the exceptional divisors \(\mathcal{E}_{i,j}\)).  By Lemma \ref{L:limit}, \(f(c)\) can only meet each \(\E_{i,j}\) after \(c\) has met the corresponding \(E_i^+\).  Equivalently, it suffices to know--in the order they occur--how \(c\) intersects each \(E_i^+\) and each exceptional curve.  

\medskip
Recall that the invariant cubic \(\mathcal{C}\) is parametrized by 
\[
  \mathcal{C} \;=\;\{\gamma(t)\colon t\in\R\cup\{\infty\}\},
\]
and that \(f\) restricts to 
\[
  f\bigl(\gamma(t)\bigr)
  \;=\;
  \gamma((1/\delta)\,t).
\]
Each line \(E_i^\pm\subset\prr\) meets \(\mathcal{C}\) in two points, so its strict transform meets \(\mathcal{C}\) in three points on \(X_{n_2,n_3}\).  Using Theorem \ref{T:Diller}, one checks that
\[
  E_i^+\cap \mathcal{C}
  \;=\;
  \bigl\{\gamma(\delta\,t_i),\,\gamma(\delta^{\,1-n_i}t_i),\,\gamma(\delta^{\,1-n_{i-1}}t_{i-1})\bigr\},
\]
\[
  E_i^-\cap \mathcal{C}
  \;=\;
  \bigl\{\gamma(\delta^{\,-n_{i-1}}t_{i-1}),\,\gamma(t_{i-1}),\,\gamma(t_{i})\bigr\},
\]
where indices are taken cyclically in \(\{1,2,3\}\), i.e.\ \(t_{j}\) is shorthand for \(t_{((j-1)\bmod3)+1}\).

Combining these intersection parameters with the formula for \(t_i\) in \eqref{E:bgamma}, one obtains the cyclic order of the three intersection points of each \(E_i^\pm\) along \(\mathcal{C}\).  Diller-Kim \cite[Table~1]{Diller-Kim} tabulates these orders; note that we are using the inverse map.

In total, there are six combinatorial types of how the three lines \(E_i^\pm\) sit on \(\mathcal{C}\).  In the remainder of this subsection, we describe each of these six cases in turn. 

\begin{rem}[Abuse of notation]
Each symbol
\[
  a_i,\;b_1,\;c_j
\]
will simultaneously denote
\begin{enumerate}
  \item the loop (generator) in \(\pi_1(X)\) based at our chosen base-point,
  \item the corresponding ``reading curve" in the cut surface,
  \item the exceptional divisor \(\mathcal E_{i,j}\subset X\), and
  \item the base-point \(\gamma(\delta^{-j}t_i)\in\mathcal C\) lying on that divisor.
\end{enumerate}
We trust that the context will make clear which of these four meanings is intended at any given moment.
\end{rem}

\noindent\textbf{Orbit data : $1,1,n\ge 8$.}
Recall that on the invariant cubic \[ \mathcal{C} \ = \ \{ \gamma(t): t \in \R \cup \{ \infty\}\}\]
the real map $f|_\mathcal{C}$ acts by
 \[ f|_\mathcal{C}(\gamma(t)) \ =\  \gamma((1/\delta)  t), \quad \delta =\lambda>1.\]
 For the $(1,1,n)$-family, the three orbits of exceptional curves -- or ``base"-- points on $\mathcal{C}$ are 
 \[ a_j \ = \ \gamma(\delta^{1-j} t_3), \quad j=1, \dots, n,\qquad \ \  b_1 \ =\  \gamma(t_1),\qquad \ \  c_1 \ =\  \gamma( t_2), \]
where $t_1, t_2, t_3$ come from the Diller parametrization in \eqref{E:bgamma}. (We continue to wirte $a_j, b_1, c_1$ for the corresponding base points on $\mathcal{C}$, their exceptional curves, and the associated generators/reading curves in $\pi_1(X)$.)  Since $1/\delta<1$, these points lie in strictly descending order of their $t$-values. Setting 
\[ a_0 \ =\  (f|_\mathcal{C})^{-1} (a_1),\qquad a_{n+1} \ =\  f|_\mathcal{C} (a_n),\]
one checks that each exceptional line $E_i^\pm$ meets $\mathcal{C}$ in exactly three of these points: 
\[ \begin{aligned}&E_1^+ \cap \mathcal{C} = \{ b_0,b_1,a_n\}, \quad E_2^+ \cap \mathcal{C} = \{ c_0,c_1,b_1\}, \quad E_3^+ \cap \mathcal{C} = \{ a_0,a_n,c_1\}, \\
 &E_1^-\cap \mathcal{C} = \{ a_1,b_1,a_{n+1}\}, \quad E_2^- \cap \mathcal{C} = \{ b_1,c_1,b_2\}, \quad E_3^- \cap \mathcal{C} = \{ a_1,c_1,c_2\}. \end{aligned}\]
Altogether, the full ordering of all intersection points on $\mathcal{C}$ (from $t=\infty$ down to $t=0$) is 
\[ \infty = \text{cusp} > a_0 >a_1  > b_0 > a_2> b_1> a_3>b_2=c_0 >a_4>c_1>a_5>c_2>a_6> \cdots a_n >a_{n+1} >0.\]
The inequality $>$ here means that the base points lie on $\mathcal{C}$ in the direction of decreasing parameter $t$. This determines the relative placement of the invariant cubic, the base points, and the exceptional lines. 

\medskip
\noindent\textbf{Notation:} Here we have
\begin{itemize}
\item the symbols $a_1, \dots, a_n, b_1, c_1$ denote the base points (centers of blow-up) of the exceptional curves in $X$
\item whereas \[a_0,b_0,c_0,a_{n+1},b_2,c_2\] denote the \textit{unique} intersection point of the correspondng exceptional curves $E_i^\pm$ with the invariant cubic $\mathcal{C}$. (Recall that each $E_i^\pm$ meets $\mathcal{C}$ in exactly one point after blowin up the base points.
\end{itemize}


\begin{center}
\begin{figure}
\includegraphics[scale=.38]{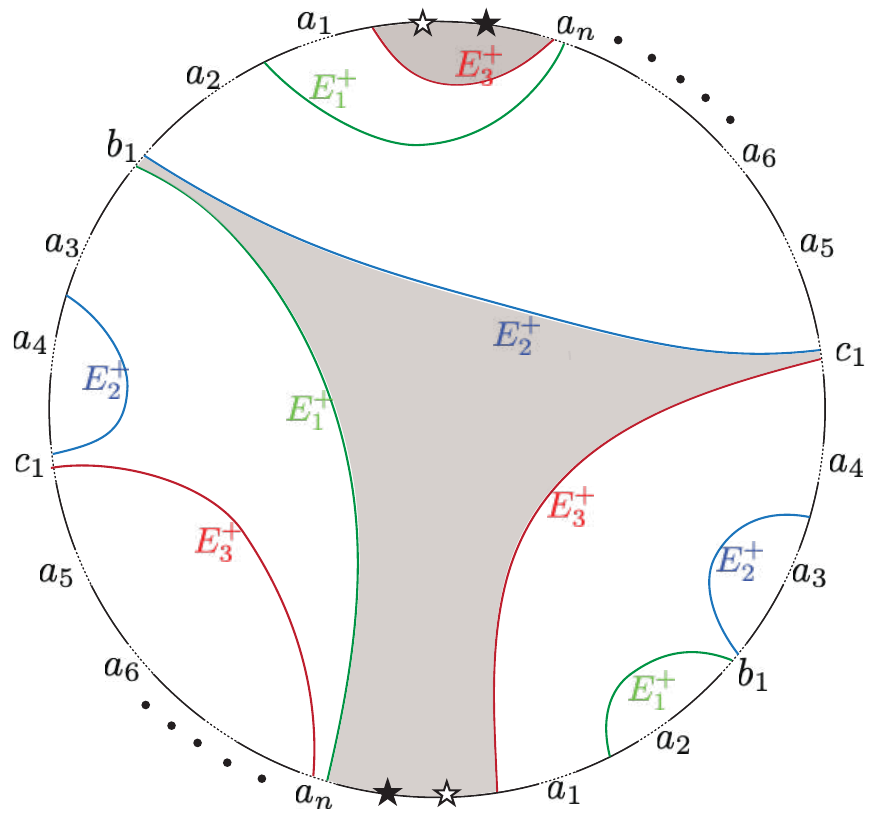} \quad\quad\ \   \includegraphics[scale=.38]{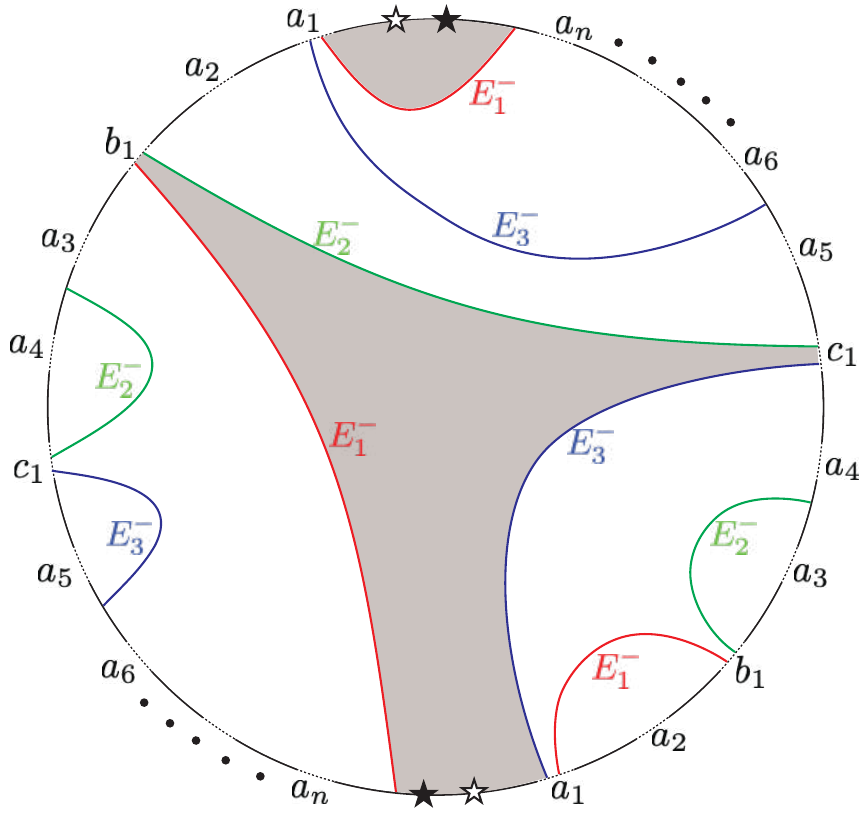} 
\caption{Reading curves sit on the boundary circle: when a path crosses a reading curve on the right semicircle (outward), we record the corresponding generator with exponent $+1$; when it crosses on the left semicircle (outward), we record it with exponent $-1$.  The hollow star is the fixed point $\gamma(\infty)$ at the cusp, and the filled star is the other fixed point $\gamma(0)$. The left panel depicts the exceptional curves $E_i^+$ for the forward map with respect to the reading curves, and the right panel depicts the exceptional curves $E_i^-$ for the inverse map.}\label{F:11nsetting}
\end{figure} 
\end{center}

Since the two fixed points lie in the shaded regions in Figure~\ref{F:11nsetting}, we choose our basepoint $p$ inside that shaded region. Under $f$, the left-hand shaded region (the one containing the two fixed points) is mapped diffeomorphically onto the right-hand shaded region, which also contains exactly two fixed points.

The left-hand shaded region is bounded by three exceptional curves $E_i^+$ and the curves lying over $a_n,b_1, c_1$. The map $f$ sends $E_i^+$ to the curves over $b_1,c_1,a_1$, and the curves over $a_n,b_1,c_1$ to $E_i^-$. Therefore, each left-hand bounded region is mapped to one of the right-hand bounded regions bounded by $E_i^-$ and the curves over $b_1,c_1,a_1$. In particular, the left-hand shaded region, which contains two fixed points, is mapped onto the right-hand shaded region, which also contains two fixed points.
By connecting $p$ to $f(p)$ within these triangles, we ensure that our choice of basepoint yields a well-defined homotopy class.

\medskip
Next, observe how $f$ acts on the exceptional curves in the $(1,1,n)$ family:
\[ f\ :\ \left\{ \ \begin{aligned} &E_1^+ \mapsto b_1 \mapsto E_2^-\\&E_2^+ \mapsto c_1 \mapsto E_3^-\\ &E_3^+ \mapsto a_1 \mapsto a_2 \mapsto \cdots \mapsto a_n \mapsto E_1^-\\ \end{aligned}\right. \]
It follws that for any simple closed curve $\gamma \subset X_{1,1,n}$, 
\begin{itemize}
\item $f(\gamma)$ will \textbf{not} intersect an exceptional curve $E_i^-$ unless $\gamma$ originally intersected its correponding exceptional curve ($a_n,b_1$ or $c_1$, respectively).
\item Conversely, $f(\gamma)$ will intersect an exceptional curve $a_1, b_1$ or $c_1$ only if $\gamma$ intersected the matching exceptional curve $E_3^+, E_1^+$ or $E_2^+$, respectively.
\end{itemize}
The coupling between how $f$ permutes the exceptional curves and reading curves is exactly what allows us to compute the induced action $f_*$ on $\pi_1(X_{1,1,n})$. In Figure \ref{F:11nmap}, we illustrate this method using a simple closed curve $g_{a_2}$ whose homotopy class is given by $a_2$-- the curve $g_{a_2}$ crosses $a_2$ on the right semicircle outward on the left-hand side. By choosing the minimum intersection, we see that the curve  $g_{a_2}$ passes successively through the exceptional curve $E_3^+, E_1^+$, the exceptional curve $a_2$, and then $E_2^+$. Thus its image $f(g_{a_2})$ must cross the exceptional curves $a_1,b_1, a_3$ and $c_1$ in that order -- avoiding all exceptional curves $E_i^-$. Since exceptional curves $ E_i^ -$ act as invisible barriers, there is only one possibility, which we show on the right-hand side in Figure \ref{F:11nmap}. From this, we conclude that the image curve $f(g_{a_2})$ enters $a_1$ from the positive side (the right semicircle), $b_1$ from the negative side, $a_3$ from the positive side, and then $c_1$ from the negative side. It follows that under $f_*$ we have 
\[ f_* (a_2) = a_1 \, \overline{b_1} \, a_3\, \overline{c_1}.\]
For each generator, the order in which the generating loop $g$ intersects the exceptional curves --- namely, the curves $E_i^+$ and the exceptional curve over the corresponding base point --- determines the homotopy class of its image under $f_*$. From Figure~\ref{F:11nsetting}, we see that the generators $a_i$ for $i=5, \dots, n-1$ have essentially the same pattern of intersection, and hence their images under $f_*$ are structurally similar. The only difference is that $f$ maps $a_i$ to $a_{i+1}$, reflecting the shift in base points. Therefore, the action of $f_{(1,1,n)*}$ can be readily determined for each $n\ge 7$.

\begin{equation}\label{E:haction11n}
f_{(1,1,n)*}:\left\{  \begin{aligned} & b_1 \ \mapsto \ a_1 \\ & c_1 \ \mapsto \ b_1 \\& a_1  \ \mapsto \ a_1\, \overline{a_2}  \,b_1\, \overline{c_1} \\&a_2  \ \mapsto \ a_1 \,\overline{b_1}\, a_3\, \overline{c_1} \\& a_3  \ \mapsto \ a_1\, c_1 \,\overline{a_4} \,b_1 \\& a_4  \ \mapsto \ a_1 \,a_5\, \overline{c_1} \,b_1\\& a_k  \ \mapsto \  c_1 \,\overline{a_{k+1}} \,\overline{a_1}\, b_1, \ \  5\le k\le n-1\\ &a_n  \ \mapsto \ c_1\end{aligned} \right.
\end{equation}

\begin{center}
\begin{figure}
\includegraphics[scale=.38]{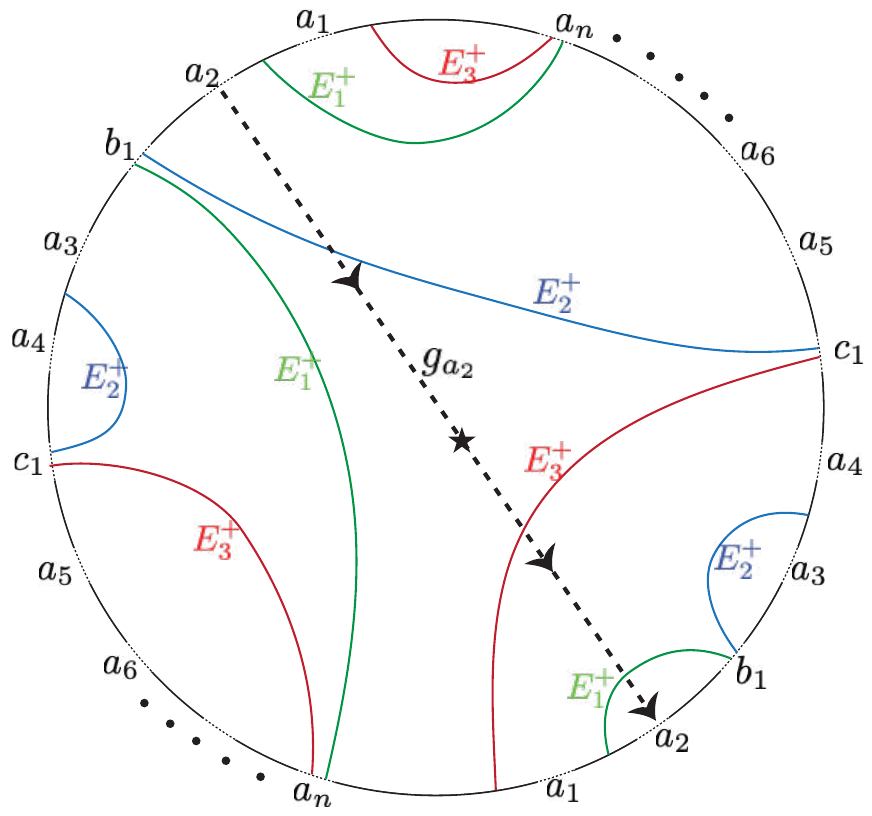} \quad\quad\ \   \includegraphics[scale=.38]{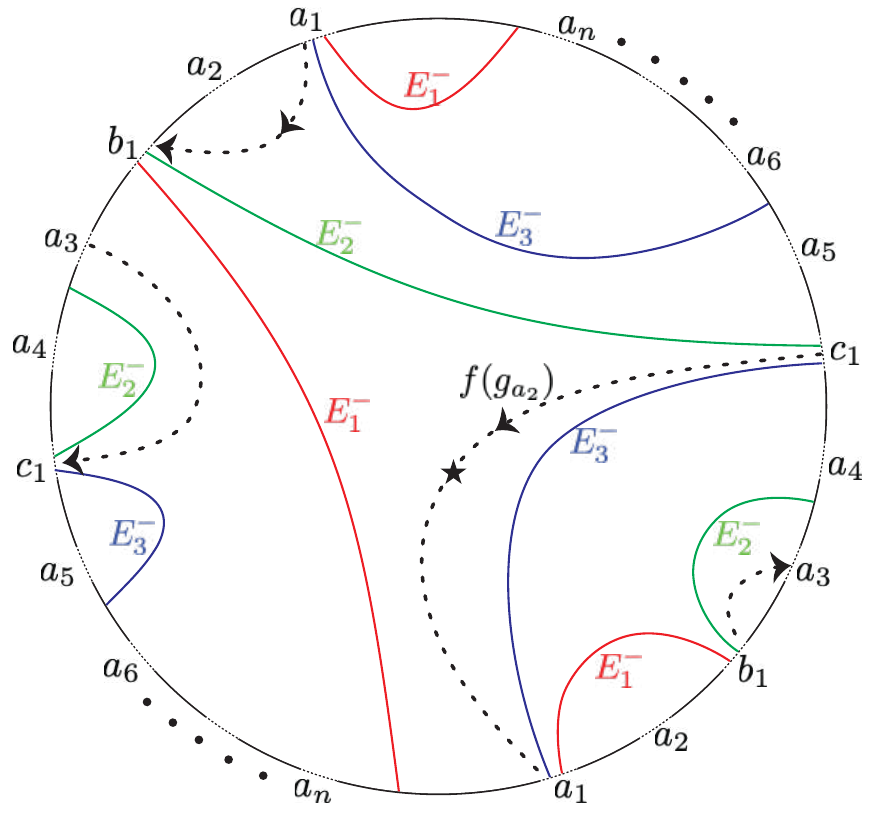} 
\caption{ Left: The loop $g_{a_2}$ (a dashed line with arrows) based at $\star$ corresponding to the generator $a_2$. Traversing with its chosen orientation, it passes successively through the exceptional curve $E_3^+, E_1^+$, the exceptional curve $a_2$, and then $E_2^+$. Right: Since $f$ carries each $E_i^+$ to the corresponding exceptional curve $b_1,c_1,a_1$, the image loop $f(g_{a_2})$ (a dashed curve with arrows) must cross the exceptional curves $a_1,b_1, a_3$ and $c_1$ in that order -- avoiding all exceptional curves $E_i^-$ because $g_{a_2}$ does not pass through $a_n, b_1$ or $c_1$. This path is illustrated on the right-hand side. Here is the $\star$ indicates the image of $\star$ on the left-hand side. By connecting $\star$ and $f(\star)$, we will obtain the same homotopy class. 
}\label{F:11nmap}
\end{figure} 
\end{center}

 \medskip

For any orbit datum $(1,m,n)$, each exceptional divisor $E_i^\pm$ meets the invariant cubic $\mathcal{C}$ in exactly three points:
\[ \begin{aligned}&E_1^+ \cap \mathcal{C} = \{ b_0,b_1,a_n\}, \quad E_2^+ \cap \mathcal{C} = \{ c_0,c_m,b_1\}, \quad E_3^+ \cap \mathcal{C} = \{ a_0,c_m,a_n\}, \\
 &E_1^-\cap \mathcal{C} = \{ a_1,b_1,a_{n+1}\}, \quad E_2^- \cap \mathcal{C} = \{ b_1,c_1,b_2\}, \quad E_3^- \cap \mathcal{C} = \{ a_1,c_1,c_{m+1}\}. \end{aligned}\]
 
 In order to compute the induced action on $\pi_1(X_{1,m,n})$ via reading curves, one must know the cyclic order in which these points (i.e.\ the labels $a_i,b_j,c_k$) appear around $\mathcal{C}$. Rather than repeating the same discussion in each case, we simply record those cyclic orderings below for all remaining families. The explicit action on the fundamental group---obtained by ``reading off" crossings in order---is then listed in Section \ref{S:seven}.  For the $(1,m,n)$-family, the three orbits of exceptional curves -- or ``base"-- points on $\mathcal{C}$ are 
 \[ a_i \ = \ \gamma(\delta^{1-i} t_3), \quad i=1, \dots, n,\qquad \ \  b_1 \ =\  \gamma( t_1),\qquad \ \  c_j \ =\  \gamma(\delta^{1-j} t_2) \quad j=1, \dots, m, \]
where $t_1, t_2, t_3$ come from the Diller parametrization \eqref{E:bgamma} in Theorem \ref{T:Diller}. Again, we want to emphasize that we wirte $a_i, b_1, c_j$ for the corresponding base points on $\mathcal{C}$, their exceptional curves, and the associated generators/reading curves in $\pi_1(X)$. 
 
 \medskip
\noindent\textbf{Orbit data : $1,2,n\ge 7$.}
 \[ 
 \begin{aligned}\infty = \text{cusp} > a_0 >a_1&  > b_0 > a_2> b_1> a_3>c_0>b_2>\\ & >a_4>c_1>a_5>c_2>a_6>c_3>a_7>a_8 \cdots a_n >a_{n+1} >0.\end{aligned}\]

  \medskip
\noindent\textbf{Orbit data : $1,m\ge 7,2$.}
 \[ 
 \begin{aligned}\infty = \text{cusp} > c_0 >a_0&  > c_1 > a_1> c_2> b_0>a_2>c_3>\\ & >b_1>a_3>c_4>b_2>c_5>c_6>c_7> \cdots c_m >c_{m+1} >0.\end{aligned}\]

  \medskip
\noindent\textbf{Orbit data : $1,n,n+1$.}
 \[ 
 \begin{aligned}\infty = \text{cusp} > a_0 >a_1&  > c_0 > b_0> a_2> c_1>b_1>a_3>\\ & >c_2>a_4>c_3> \cdots a_{n+1}>c_n>a_{n+2} >c_{n+1} >0.\end{aligned}\]
 
   \medskip
\noindent\textbf{Orbit data : $1,m,n \ge m+2$.}
 \[ 
 \begin{aligned}\infty = \text{cusp} > a_0 >a_1&  > b_0 > c_0> a_2> b_1>c_1>a_3>\\ & >b_2>c_2>a_4>c_3 >a_5> \cdots > c_m >a_{m+2} >\\& > c_{m+1} > a_{m+3} >a_{m+4} > \cdots> a_n>a_{n+1} >0.\end{aligned}\]
 
    \medskip
\noindent\textbf{Orbit data : $1,m\ge n+1,n$.}
 \[ 
 \begin{aligned}\infty = \text{cusp} > a_0 >c_0&  > a_1 > b_0> c_1> a_2>b_1>c_2>\\ & >a_3>b_2>c_3>a_4 >c_4> \cdots > c_{n-1} >a_{n} >\\& > c_{n} > a_{n+1} >c_{n+1}>c_{n+2}  > \cdots> c_m>c_{m+1} >0.\end{aligned}\]
 
Although both lie in the same combinatorial family, the cases $(1,n+1,n)$ and $(1,m,n)$ with $m\ge n+2$ behave differently. Accordinly, we split out analysis into two regimes:
\begin{enumerate}
\item \textbf{Adjacent case} : $(1,n+1,n)$ with $m-n=1$.
\item \textbf{Gap case} : $(1,m,n)$ with $m-n\ge 2$.
\end{enumerate}

\section{Seven Families}\label{S:seven}
The uniform behavior of the action on the fundamental group across all seven families is central to our argument. To highlight this, we present the formulas for each family here rather than relegating them to an appendix. For each case, we apply the same idea---determining the homotopy class of the image of a generating loop using reading curves---as we did for $f_{(1,1,n)*}$ in \eqref{E:haction11n}. The seven families are organized as follows:
\vspace{1ex}

\begin{center}
\begin{tikzpicture}[
  every node/.style={
    draw, rectangle, rounded corners,
    text width=2.5cm, align=center, font=\small
  },
  edge from parent/.style={->, draw},
  level distance=2cm,
  sibling distance=4.5cm
]
\node {Orbit datum\\\(\,(1,m,n)\,\)}
  child { node {\(m=1\) or \(n=1\)\\(Sec 4.1)} }
  child { node {\(m=2\) or \(n=2\)\\(Sec 4.2-4.3)} }
  child { node {\(m,n\ge3\)}
    child { node {\(\lvert m-n\rvert=1\)\\(Sec 4.4-4.5)} }
    child { node {\(\lvert m-n\rvert\ge2\)\\(Sec 4.6-4.7)} }
  };
\end{tikzpicture}
\end{center}


\begin{rem}
By the cyclic symmetry of the orbit-data triple, the case \((1,n,1)\) is obtained from the \((1,1,n)\) case by the relabeling
\[
  (a_i,b_1,c_1)\;\longmapsto\;(c_i,a_1,b_1).
\]
In particular, every formula for \(f_*\) above equation \eqref{E:11n} carries over under the substitution \(a\mapsto c\), \(b\mapsto a\), \(c\mapsto b\).  For example,
\[
  f_{(1,n,1)*}(c_1)
  \;=\;
  c_1\,\overline{c_2}\,a_1\,\overline{b_1}
  \quad
  \Longleftrightarrow
  \quad
  f_{(1,1,n)*}(a_1)
  \;=\;
  a_1\,\overline{a_2}\,b_1\,\overline{c_1}
\]
\end{rem}

\subsection{Orbit data : $1,1,n\ge 8$}\label{SS:11n}
\begin{equation}\label{E:11n}\pi_1(X) = \langle a_1, \dots, a_n, b_1, c_1 \rangle \cong F_{n+2}; \quad f_*:\left\{  \begin{aligned} & b_1 \ \mapsto \ a_1 \\ & c_1 \ \mapsto \ b_1 \\& a_1  \ \mapsto \ a_1\, \overline{a_2}  \,b_1\, \overline{c_1} \\&a_2  \ \mapsto \ a_1 \,\overline{b_1}\, a_3\, \overline{c_1} \\& a_3  \ \mapsto \ a_1\, c_1 \,\overline{a_4} \,b_1 \\& a_4  \ \mapsto \ a_1 \,a_5\, \overline{c_1} \,b_1\\& a_k  \ \mapsto \  c_1 \,\overline{a_{k+1}} \,\overline{a_1}\, b_1, \ \  5\le k\le n-1\\ &a_n  \ \mapsto \ c_1\end{aligned} \right.
\end{equation}


\subsection{Orbit data : $1,2,n\ge 7$}\label{SS:12n}
\begin{equation}\label{E:12n}\pi_1(X) = \langle a_1, \dots, a_n, b_1, c_1,c_2 \rangle \cong F_{n+3};\quad f_*: \left\{  \begin{aligned} & b_1 \ \mapsto \ a_1 \\ & c_1 \ \mapsto \ a_1\,c_2\,\overline{c_1}\,b_1 \\& c_2 \ \mapsto \ b_1 \\& a_1  \ \mapsto \ a_1\, \overline{a_2}  \,b_1\, \overline{c_1} \\&a_2  \ \mapsto \ a_1 \,\overline{b_1}\, a_3\, \overline{c_1} \\& a_3  \ \mapsto \ a_1\, c_1 \,\overline{a_4} \,b_1 \\& a_i  \ \mapsto \ a_1 \,a_{i+1}\, \overline{c_1} \,b_1,\ \ \ i=4,5\\& a_k  \ \mapsto \  c_1 \,\overline{a_{k+1}} \,\overline{a_1}\, b_1, \ \  6\le k\le  n-1\\ &a_n  \ \mapsto \ c_1\end{aligned} \right.
\end{equation}

\subsection{Orbit data : $1,n,2\ge 7$}\label{SS:1n2}
\begin{equation}\label{E:1n2} \pi_1(X) = \langle a_1, a_2,, b_1, c_1,\dots, c_n \rangle \cong F_{n+3};\quad f_*:\left\{  \begin{aligned} & b_1 \ \mapsto \ a_1 \\ & a_1 \ \mapsto \ c_1\,\overline{a_1}\,a_2\overline{b_1} \\& a_2 \ \mapsto \ c_1 \\& c_i  \ \mapsto \ c_1\, \overline{a_1}  \,c_{i+1}\, \overline{b_1},\ \  i=1,2 \\&c_3  \ \mapsto \ c_1 \,c_4\,\overline{b_1}\, a_1\\&  c_k  \ \mapsto \  b_1 \,\overline{c_{k+1}} \,\overline{c_1}\, a_1, \ \  4\le k \le n-1\\ &c_n  \ \mapsto \ b_1\end{aligned} \right.
\end{equation}

\subsection{Orbit data : $1,n,n+1$ with $n\ge 4$}\label{SS:1nnp}
\begin{equation}\label{E:1nnp1}\pi_1(X) = \langle a_1, c_1, \dots,a_n,c_n, a_{n+1}, b_1 \rangle \cong F_{2n+2} ;\quad f_*: \left\{ \begin{aligned} & b_1 \ \mapsto \ a_1 \\ & c_1 \ \mapsto \ a_1\,\overline{c_1}\,b_1\,\overline{c_2} \\& c_k \ \mapsto \ a_1\,c_{k+1}\,\overline{b_1}\, c_1,\ \  2\le k\le  n-1 \\& c_{n}  \ \mapsto \ b_1 \\& a_1 \ \mapsto \ a_1\, \overline{a_2}\, c_1 \,\overline{b_1} \\& a_2 \ \mapsto\ a_1\, \overline{c_1}\, b_1\, \overline{a_3}\\& a_k  \ \mapsto \  a_1 \, a_{k+1}\,\overline{b_1} \, c_1,\ \  3 \le k\le n\\ &a_{n+1}  \ \mapsto \ c_1\end{aligned} \right.
\end{equation}

\subsection{Orbit data : $1,n+1,n$ with $n\ge 4$}\label{SS:1npn}
\begin{equation}\label{E:1np1n}\pi_1(X) = \langle c_1, a_1, \dots,c_n,a_n, c_{n+1}, b_1 \rangle \cong F_{2n+2} ;\quad f_*: \left\{ \begin{aligned} & b_1 \ \mapsto \ a_1 \\ & c_1 \ \mapsto \ a_1\,\overline{c_1}\,b_1\,\overline{c_2} \\& c_k \ \mapsto \ a_1\,c_{k+1}\,\overline{b_1}\, c_1, \ \  2 \le k \le n-1 \\& c_{n}  \ \mapsto \ b_1\, \overline{c_{n+1}}\, \overline{a_1}\, c_1\\& c_{n+1}  \ \mapsto \ b_1 \\& a_1 \ \mapsto \ a_1\, \overline{c_1}\, a_2 \,\overline{b_1} \\& a_2 \ \mapsto\ a_1\, \overline{c_1}\, b_1\, \overline{a_3}\\& a_k  \ \mapsto \  a_1 \, a_{k+1}\,\overline{b_1} \, c_1,\ \ 3 \le k \le  n-1\\ &a_{n}  \ \mapsto \ c_1\end{aligned} \right.
\end{equation}

\subsection{Orbit data : $1,m,n$ with $m\ge 3, n\ge m+2$} \label{SS:1mn}
\begin{equation}\label{E:1mn}\pi_1(X) = \langle a_1, \dots, a_{n}, b_1, c_1,\dots,c_{m} \rangle \cong F_{1+m+n} ;\quad f_*: \left\{  \begin{aligned} & b_1 \ \mapsto \ a_1 \\ & c_k \ \mapsto \ a_1\, c_{k+1}\,\overline{c_1}\,b_1,\ \, 1\le k \le m-1 \\& c_m \ \mapsto \ b_1 \\& a_1 \ \mapsto \ a_1\, \overline{a_2}\, b_1 \,\overline{c_1} \\& a_2 \ \mapsto\ a_1\, \overline{b_1}\, c_1\, \overline{a_3}\\& a_i  \ \mapsto \  a_1 \, a_{i+1}\,\overline{c_1} \, b_1,\ \, \ 3\le i \le m+2\\& a_k  \ \mapsto \  c_1 \, \overline{a_{k+1}}\,\overline{a_1} \, b_1, \ \, m+2 \le k \le n-1\\ &a_{n}  \ \mapsto \ c_1\end{aligned} \right.
\end{equation}

\subsection{Orbit data : $1,m,n$ with $n\ge 3, m\ge n+2$} \label{SS:1mn}
\begin{equation}\label{E:1nm}\pi_1(X) = \langle a_1, \dots, a_{n}, b_1, c_1,\dots,c_{m} \rangle \cong F_{1+m+n} ;\quad f_*: \left\{ \begin{aligned} & b_1 \ \mapsto \ a_1 \\&c_1\ \mapsto\ a_1 \,\overline{c_1}\, b_1\, \overline{c_2}\\& c_i \ \mapsto\ a_1 \, c_{i+1}\, \overline{b_1}\, c_1, \ \, 2\le i\le  n\\ & c_k \ \mapsto \ b_1\, \overline{c_{k+1}}\,\overline{a_1}\,c_1,\ \, n+1\le k \le m-1 \\& c_m \ \mapsto \ b_1 \\& a_1 \ \mapsto \ a_1\, \overline{c_1}\, a_2 \,\overline{b_1} \\& a_2 \ \mapsto\ a_1\, \overline{c_1}\, b_1\, \overline{a_3}\\& a_k  \ \mapsto \  a_1 \, a_{k+1}\, \overline{b_1}\,c_1,  \ \, 3\le k \le n-1\\ &a_{n}  \ \mapsto \ c_1\end{aligned} \right.
\end{equation}

\section{Actions on the fundamental groups}\label{S:homotopy}

Let $X_{1,m,n}$ be the surface with one or two boundary components determined by the orbit data $(1,m,n)$, and let
 \[f_{(1,m,n)*}: \pi_1(X_{1,m,n}) \to \pi_1(X_{1,m,n})\] be the induced action where $\pi_1(X_{1,m,n})$ is a free group with $1+m+n$ generators:
 \[ \pi_1(X_{1,m,n}) = \langle a_1,a_2,\dots a_n, b_1, c_1,c_2,\dots, c_m \rangle.\]
 For any $\omega \in \pi_1(X_{1,m,n})$, let $\omega_\#$ denote its reduced word and set $\ell(\omega)= \ell(\omega_\#)$ for the length of that word. Write $\mathrm{Sp}\{ x_1, x_2, \dots, x_k\}$ for the set of all words in the letters $x_1^{\pm1}, \dots, x_k^{\pm1}$.   From equations \eqref{E:11n} -- \eqref{E:1nm}, we obtain:
 
 \noindent\textbf{Abbreviation.}  In what follows we set
\[
  f_* \;=\; f_{(1,m,n)*}\colon \pi_1(X_{1,m,n})\to\pi_1(X_{1,m,n}).
\]

\begin{lem} Under $f_*$, the images of generators satisfy:
\begin{enumerate}
\item $f_* b_1 = a_1$, $f_* c_m = b_1$, and $f_* a_n= c_1$.
\item For each $i= 1, \dots, m-1$, \[(f_* c_i)_\# \in \mathrm{Sp} \{ a_1, b_1, c_1, c_{i+1}\} \text{ and } \ell(f_*c_i) = 4.\] 
\item For each $i= 1, \dots, n-1$, \[(f_* a_i)_\# \in \mathrm{Sp} \{ a_1, b_1, c_1, a_{i+1}\} \text{ and } \ell(f_*a_i) = 4.\] 
\end{enumerate}
\end{lem}
\begin{proof}
Immediate from the defining relations \eqref{E:11n} -- \eqref{E:1nm}.
\end{proof}

\begin{lem}[Truncation-invariance]\label{L:TIa}
Fix integers \(m\ge1\), \(j\ge m\), and \(k\ge0\).  Then for every \(n\ge j+k\) and every
\[
  \omega\;\in\;\mathrm{Sp}\{\,b_1,\,c_1,\dots,c_m,\,
                         a_1,\dots,a_j\},
\]
one has
\[
  f_{(1,m,n)*}^{\,k}(\omega)
  \;=\;
  f_{(1,m,j+k)*}^{\,k}(\omega).
\]
We refer to this bounded-reach property as \textit{truncation-invariance}.
\end{lem}

\begin{proof}
Immediate from the fact (seen in the defining formulas \eqref{E:11n} -- \eqref{E:1nm}) that no generator of index exceeding \(j+k\) ever appears in the first \(k\) iterates of any word in 
\(\{b_1,c_1,\dots,c_m,a_1,\dots,a_j\}.\)
\end{proof}

Similarly, we also have truncation-invariance in the other orbit length:

\begin{lem}\label{L:TIc}
Fix integers \(n\ge1\), \(i\ge n\), and \(k\ge0\).  Then for every \(m\ge i+k\) and every
\[
  \omega\;\in\;\mathrm{Sp}\{\,b_1,\,c_1,\dots,c_i,\,
                         a_1,\dots,a_n\},
\]
one has
\[
  f_{(1,m,n)*}^{\,k}(\omega)
  \;=\;
  f_{(1,i+k,n)*}^{\,k}(\omega).
\]
\end{lem}

For the tail generators, we have: 
\begin{cor}\label{C:tail}
Fix integers \(m,n\ge1\) and \(k\ge0\).  Then:

1.  If \(n\ge m+k\),  for each \(j=0,1,\dots,k-1\) we have
   \[
     f_{(1,m,n)*}^{\,k}\bigl(a_{n-j}\bigr)
     \;=\;
     f_{(1,m,m+k)*}^{\,k}\bigl(a_{\,m+k-j}\bigr).
   \]

2.  If \(m\ge n+k\),  for each \(j=0,1,\dots,k-1\) we have
   \[
     f_{(1,m,n)*}^{\,k}\bigl(c_{m-j}\bigr)
     \;=\;
     f_{(1,n+k,n)*}^{\,k}\bigl(c_{\,n+k-j}\bigr).
   \]
\end{cor}

\begin{proof}
Note that for each $j=0,1,\dots, k-1$, \[f_{(1,m,n)*}^{j+1}a_{n_j} \in \mathrm{Sp} \{ a_1, \dots, a_{j+1}, b_1, c_1, \dots, c_{j+1}\}.\] The first assertion then follows immediately from Lemma \ref{L:TIa}.
Likewise, the second case is a direct consequence of Lemma \ref{L:TIc}.
\end{proof}

Combining Lemmas \ref{L:TIa} -- \ref{L:TIc}, Corollary \ref{C:tail}, and the defining relations 
\eqref{E:11n} -- \eqref{E:1nm}, we obtain:

\begin{cor}\label{C:special}
Fix integers \(m,\;k\ge0\).  Then:

\begin{enumerate}
  \item If \(n\ge m+k\), then for each \(j=k+1,\dots,n-k\) there exist
    \[
      \omega_1,\omega_2
      \;\in\;
      \mathrm{Sp}\{\,a_1,\dots,a_{k},\,b_1,\,c_1,\dots,c_{\min(m,k)}\}
    \]
    such that
    \[
      \bigl(f_{(1,m,n)*}^{\,k}(a_j)\bigr)_{\#}
      \;=\;\omega_1\,a_{\,j+k}\,\omega_2.
    \]

  \item If \(m\ge n+k\), then for each \(j=k+1,\dots,m-k\) there exist
    \[
      \eta_1,\eta_2
      \;\in\;
      \mathrm{Sp}\{\,a_1,\dots,a_{\min(n,k)},\,b_1,\,c_1,\dots,c_{k}\}
    \]
    such that
    \[
      \bigl(f_{(1,m,n)*}^{\,k}(c_j)\bigr)_{\#}
      \;=\;\eta_1\,c_{\,j+k}\,\eta_2.
    \]
\end{enumerate}
\end{cor}

\medskip

The next result collects and summarizes all of the inductive and truncation-invariance properties established above.  Since each item is an immediate consequence of Lemmas \ref{L:TIa} -- \ref{L:TIc}, Corollary \ref{C:tail}, and the defining relations \eqref{E:11n} -- \eqref{E:1nm}, we omit a detailed proof.

\begin{prop}\label{P:actionP}
Fix \(m\ge1\), \(k\ge0\), and \(J\ge m\), and let
\[
  \omega_1,\omega_2
  \;\in\;
  \mathrm{Sp}\{\,a_1,\dots,a_{J},\,b_1,\,c_1,\dots,c_{m}\}.
\]
Then for every \(n\ge J+k\) and every integer \(i\ge0\):

\begin{enumerate}
  \item For each \(j=1,\dots,n-k\),
    \[
      f_{(1,m,n+1)*}^{\,k}\bigl(\omega_1\,a_j\,\omega_2\bigr)
      \;=\;
      f_{(1,m,n)*}^{\,k}\bigl(\omega_1\,a_j\,\omega_2\bigr).
    \]

  \item For each \(j=n-k+2,\dots,n+1\),
    \[
      f_{(1,m,n+1)*}^{\,k}\bigl(\omega_1\,a_j\,\omega_2\bigr)
      \;=\;
      f_{(1,m,n)*}^{\,k}\bigl(\omega_1\,a_{\,j-1}\,\omega_2\bigr).
    \]

  \item There exist
    \[
      \eta_1,\eta_2
      \;\in\;
      \mathrm{Sp}\{\,a_1,\dots,a_{J+k},\,b_1,\,c_1,\dots,c_{m}\}
    \]
    such that
    \[
      f_{(1,m,n+i)*}^{\,k}\bigl(\omega_1\,a_{\,n+i-k}\,\omega_2\bigr)
      \;=\;
      \eta_1\,a_{\,n+i}\,\eta_2.
    \]
\end{enumerate}

An entirely analogous list of statements holds with all \(a_j\) replaced by \(c_j\), fixing \(n\) instead of \(m\).
\end{prop}

\medskip

\begin{rem}
Exactly the same collection of identities holds if every occurrence of \(a_j\) in the above statement is replaced by its inverse \(\overline a_j = a_j^{-1}\).  In other words, all of the inductive and truncation-invariance properties remain valid for words built from the letters \(\overline a_j\) in place of \(a_j\).

\end{rem}

By repeated use of the truncation-invariance properties in Proposition \ref{P:actionP}, we obtain the following corollaries.  Although the list of cases is long, each is handled in the same way.

\begin{cor}\label{C:boundary}
Let $\omega_{\mathcal C}\in\pi_1(X)$ be the homotopy class of the invariant cubic, written as a cyclically reduced word in the generators.  Denote by $\overline{\omega_{\mathcal C}}$ the word obtained by inverting each letter of~$\omega_{\mathcal C}$.  For example, in the $(1,1,8)$ case one has
\[
  \omega_{\mathcal C}
  = 
  a_1\,\overline{a_2}\,b_1\,\overline{a_3}\,a_4\,\overline{c_1}\,a_5\,\overline{a_6}\,a_7\,\overline{a_8},
  \quad
  \overline{\omega_{\mathcal C}}
  =
  \overline{a_1}\,a_2\,\overline{b_1}\,a_3\,\overline{a_4}\,c_1\,\overline{a_5}\,a_6\,\overline{a_7}\,a_8.
\]
Then under the induced action on $\pi_1(X)$ we have
\[
  f_*\bigl(\omega_{\mathcal C}\bigr)
  = 
  \omega_{\mathcal C}
  \quad\text{if $1+m+n$ is even,}
  \quad\text{and}\quad
  f_*\bigl(\omega_{\mathcal C}\,\overline{\omega_{\mathcal C}}\bigr)
  =
  \omega_{\mathcal C}\,\overline{\omega_{\mathcal C}}
  \quad\text{if $1+m+n$ is odd.}
\]
\end{cor}

\begin{cor}\label{C:np1}
Fix integers \(m_0,n_0\ge1\) and \(k\ge1\), and let
\[
  \omega_1,\omega_2
  \;\in\;
  \mathrm{Sp}\{\,a_1,\dots,a_{n_0},\,b_1,\,c_1,\dots,c_{m_0}\}.
\]
Then for every \(n\ge \max\{m_0,n_0\}+k\) and every \(\epsilon\in\{1,2\}\), the action
\[
  f_{(1,n,n+1)*}^{\,k}
\]
on the mixed generator \(\omega_1\,a_j\overline{c_{j-\epsilon}}\,\omega_2\) satisfies:

\begin{enumerate}
  \item[\(1.\)] If \(1\le j\le n-k\), then
    \[
      f_{(1,n,n+1)*}^{\,k}\bigl(\omega_1\,a_j\overline{c_{j-\epsilon}}\,\omega_2\bigr)
      =
      f_{(1,n-1,n)*}^{\,k}\bigl(\omega_1\,a_j\overline{c_{j-\epsilon}}\,\omega_2\bigr).
    \]

  \item[\(2.\)] If \(j = n-k+1\), then there exist
    \(\eta_1,\dots,\eta_4\in\mathrm{Sp}\{a_1,\dots,a_{n_0+k},\,b_1,\,c_1,\dots,c_{m_0+k}\}\)
    such that    
    \[
    \begin{aligned}
      &f_{(1,n,n+1)*}^{\,k}\bigl(\omega_1\,a_{n-k+1}\overline{c_{n-k-1}}\,\omega_2\bigr)
      = \eta_1\,a_{n+1}\,\overline{c_{n-1}}\,\eta_2,\\
      &f_{(1,n,n+1)*}^{\,k}\bigl(\omega_1\,a_{n-k+1}\overline{c_{n-k}}\,\omega_2\bigr)
      = \eta_3\,a_{n+1}\,\overline{c_n}\,\eta_4.\\    \end{aligned}
    \]

  \item[\(3.\)] If \(j = n-k+2\), then
    \begin{itemize}
      \item For \(\epsilon=2\), there exist
        \(\eta_1,\eta_2\in\mathrm{Sp}\{a_1,\dots,a_{n_0+k},\,b_1,\,c_1,\dots,c_{m_0+k}\}\)
        so that
        \[
          f_{(1,n,n+1)*}^{\,k}\bigl(\omega_1\,a_{n-k+2}\overline{c_{n-k}}\,\omega_2\bigr)
          = 
          \eta_1\,\overline{c_n}\,\eta_2.
        \]
      \item For \(\epsilon=1\) ,
        \[
          f_{(1,n,n+1)*}^{\,k}\bigl(\omega_1\,a_{n-k+2}\overline{c_{n-k+1}}\,\omega_2\bigr)
          =
          f_{(1,n-1,n)*}^{\,k}\bigl(\omega_1\,a_{n-k+1}\overline{c_{n-k}}\,\omega_2\bigr).
        \]
    \end{itemize}

  \item[\(4.\)] If \(n-k+3\le j\le n+1\), then
    \[
      f_{(1,n,n+1)*}^{\,k}\bigl(\omega_1\,a_j\overline{c_{j-\epsilon}}\,\omega_2\bigr)
      =
      f_{(1,n-1,n)*}^{\,k}\bigl(\omega_1\,a_{j-1}\overline{c_{j-1-\epsilon}}\,\omega_2\bigr).
    \]
\end{enumerate}
\end{cor}

\begin{proof}
Each case is a direct consequence of Proposition \ref{P:actionP} (truncation-invariance) applied to the appropriate indices.
\end{proof}

Similarly, we have:
\begin{cor}\label{C:p1n}
Fix integers \(m_0,n_0\ge1\) and \(k\ge1\), and let
\[
  \omega_1,\omega_2
  \;\in\;
  \mathrm{Sp}\{\,a_1,\dots,a_{n_0},\,b_1,\,c_1,\dots,c_{m_0}\}.
\]
Then for every \(n\ge \max\{m_0,n_0\}+k\) and every \(\epsilon\in\{0,1\}\), the action
\[
  f_{(1,n+1,n)*}^{\,k}
\]
on the mixed generator \(\omega_1\,a_j\overline{c_{j-\epsilon}}\,\omega_2\) satisfies:

\begin{enumerate}
  \item[\(1.\)] If \(1\le j\le n-k\), then
    \[
      f_{(1,n+1,n)*}^{\,k}\bigl(\omega_1\,a_j\overline{c_{j-\epsilon}}\,\omega_2\bigr)
      =
      f_{(1,n,n-1)*}^{\,k}\bigl(\omega_1\,a_j\overline{c_{j-\epsilon}}\,\omega_2\bigr).
    \]

  \item[\(2.\)] If \(j = n-k+1\), then there exist
        \(\eta_1,\eta_2,\eta_3,\eta_4\in\mathrm{Sp}\{a_1,\dots,a_{n_0+k},\,b_1,\,c_1,\dots,c_{m_0+k}\}\)
        such that
        \[ \begin{aligned}
          &f_{(1,n+1,n)*}^{\,k}\bigl(\omega_1\,a_{n-k+1}\overline{c_{n-k}}\,\omega_2\bigr)
          = 
          \eta_1\,c_n\,\eta_2,\\
       &f_{(1,n+1,n)*}^{\,k}\bigl(\omega_1\,a_{n-k+1}\overline{c_{n-k+1}}\,\omega_2\bigr)
          = 
          \eta_3\,c_{n+1}\,\eta_4.\\ \end{aligned}
        \]

  \item[\(3.\)] If \(j = n-k+2\), then
    \begin{itemize}
      \item For \(\epsilon=1\), there exist
        \(\eta_1,\eta_2\in\mathrm{Sp}\{a_1,\dots,a_{n_0+k},\,b_1,\,c_1,\dots,c_{m_0+k}\}\)
        so that
        \[
          f_{(1,n+1,n)*}^{\,k}\bigl(\omega_1\,a_{n-k+2}\overline{c_{n-k+1}}\,\omega_2\bigr)
          = 
          \eta_1\,c_{n+1}\,\eta_2.
        \]
      \item For \(\epsilon=0\),
        \[
          f_{(1,n+1,n)*}^{\,k}\bigl(\omega_1\,a_j\overline{c_{j}}\,\omega_2\bigr)
          =
          f_{(1,n,n-1)*}^{\,k}\bigl(\omega_1\,a_{j-1}\overline{c_{j-1}}\,\omega_2\bigr).
        \]
    \end{itemize}

  \item[\(4.\)] If \(n-k+3\le j\le n+1\), then
    \[
      f_{(1,n+1,n)*}^{\,k}\bigl(\omega_1\,a_j\overline{c_{j-\epsilon}}\,\omega_2\bigr)
      =
      f_{(1,n,n-1)*}^{\,k}\bigl(\omega_1\,a_{j-1}\overline{c_{j-1-\epsilon}}\,\omega_2\bigr).
    \]
\end{enumerate}
\end{cor}

Before stating the definition, we introduce an abstract property that all of our invariant semigroups share.  It formalizes the idea that, as the orbit parameters \((m,n)\) grow, each semigroup is generated by a fixed ``core" coming from a small base surface together with a family of ``tail" elements accounting for the new generators.

\begin{defn}[\textbf{Core-Tail Generation}]\label{D:coretail}
Fix base-parameters $m_{0},n_{0}$ and let
\[C=C_{m_0,n_0}\;\subset\;\pi_{1}\bigl(X_{1,m_{0},n_{0}}\bigr)\]
be a finite set of ``core" generators.  For each $(m,n)$ with $m> m_{0}$, $n> n_{0}$, define the ``tail-sets"
\[ T_{m,n}\ \subset \ \bigl\{\,w_{1}\,\alpha\,w_{2} :   \alpha \in\{a_{n}^\pm,\,c_{m}^\pm,\,a_{n}\overline{c_{n-j}},\,c_{n-j} \overline{a_n}\ |j=0,1,2\},   \;w_{1},w_{2}\in \pi_{1}\bigl(X_{1,m_{0},n_{0}}\bigr)   \bigr\},\]
and
\[\hat{T}_{m,n} \ \subset \ \bigl\{\,w_{1}\,\alpha\,w_{2} :   \alpha \in\{a_{n+i}^\pm,\,c_{m+i}^\pm: i\ge -1 \}   \;w_{1},w_{2}\in \pi_{1}\bigl(X_{1,m_{0},n_{0}}\bigr)   \bigr\}.\]
Notice that $T_{m,n}$ and $\hat{T}_{m,n}$ could be empty sets. 
Then set
\[ S_{m,n} \ =\  \left\langle\,   C\;\cup\;    \left( \bigcup_{\substack{m_{0}< m'\le m\\n_{0}< n'\le n}}T_{m',n'}  \right) \;\cup\; \hat{T}_{m,n} \right\rangle. \]
We say that the family of semigroups $\{\,S_{m,n}\}$ has the ``core-tail generation" property if each semigroup $S_{m,n}$ is generated in this way by the fixed finite ``core" $C$ together with all ``tail" elements $\hat{T}_{m,n}$ and $T_{m',n'}$ for $m_{0}< m'\le m$ and $n_{0}< n'\le n$.
\end{defn}

\begin{defn}[Concatenation-Compatibility]\label{D:nocancel}
Let \(S\subset\pi_1(X)\) be a semigroup generated by \(\{s_i\}\).  We say \(S\) has the ``concatenation-compatibility" property (or ``no-cancellation" property) if for every finite sequence \(s_{i_1},s_{i_2},\dots,s_{i_k}\in S\),
\[
  \bigl(s_{i_1}s_{i_2}\cdots s_{i_k}\bigr)_{\#}
  \;=\;
  (s_{i_1})_{\#}\,(s_{i_2})_{\#}\,\cdots\,(s_{i_k})_{\#}.
\]
In other words, concatenating generators and then cyclically reducing produces the same word as first cyclically reducing each generator and then concatenating.
\end{defn}

\medskip

Having verified one base case by direct computation and established the core-tail generation property, all subsequent invariance and positivity statements follow by the same truncation-invariance argument. The next result packages that argument into a single statement.

\begin{prop}[\textbf{Core-Tail Induction Principle}]\label{P:ctinduction}
Let \((1,m,n)\) be any orbit datum, and suppose \(\{S_{m,n}\}\) is a family of finitely generated semigroups satisfying both concatenation-compatibility and the core-tail generation property with base-parameters \(m_0,n_0\).  Fix an integer \(k\ge1\) and set
\[
  m_c = m_0 + k,\quad n_c = n_0 + k.
\]
If the semigroup \(S_{m_c,n_c}\) is invariant under \(F_*^k\) and \(F_*^k\) acts positively on it, then for every \(m\ge m_c\), \(n\ge n_c\), the same holds: each \(S_{m,n}\) is invariant under \(F_*^k\) and \(F_*^k\) acts positively on \(S_{m,n}\).
\end{prop}

\begin{proof}
Immediate from truncation-invariance (Lemmas 2.3-2.4) and the tail-index lemmas (and their corollaries), which together show that the action on the extra tail generators coincides with the base case up to conjugation by core elements.
\end{proof}

\section{Positivity on Invariant Semigroup}\label{S:positivity}
In this section we exhibit, for each orbit datum, a natural invariant semigroup inside \(\pi_1(X)\) and show that positivity is preserved under the induced action \(f_*\).  As we saw in Section \ref{S:homotopy}, \(f_*\) falls into exactly the seven cases given by \eqref{E:11n}--\eqref{E:1nm}.  We begin with the maps having orbit datum \((1,1,n)\).

\vspace{1ex}
Recall that \(\pi_1(X_{1,1,n})\) is a free group on \(n+2\) generators.  To study positivity and growth-rate, it is convenient to represent each element by its cyclically reduced word (of finite length) and then extend that block periodically in both directions, i.e.\ as a bi-infinite sequence (see, for example, \cite{Floyd:1980,Birman-Series,Bestvina-Handel:Tits}).  Under this correspondence, every group element corresponds to a doubly infinite string whose ``fundamental block" is its cyclic reduction. 

We denote by
\[
  f_*\colon \pi_1(X)\longrightarrow\pi_1(X)
\]
the action on finite words, and by
\[
  F_*\colon \{\text{bi-infinite sequences}\}\longrightarrow\{\text{bi-infinite sequences}\}
\]
the induced map.  We use this bi-infinite model to define our invariant semigroup and show that positivity is preserved under \(F_*\).

\vspace{1ex}
We carry out the same construction for each of the remaining orbit-datum types \((1,2,n)\), \((1,n,2)\), etc., in the subsequent subsections.

\medskip

\noindent\textbf{Convention (Action on concatenated words).}
Whenever we apply \(F_*\) to a product \(w_1w_2\in\pi_1(X)\), we first apply the finite-word action \(f_*\) and then cyclically reduce. In symbols:
\[
  F_*(w_1w_2)\;=\;\bigl(f_*(w_1w_2)\bigr)_{\#^c}.
\]
By induction, for any \(k\ge1\),
\[
  F_*^k(w_1w_2)
  \;=\;
  \bigl(f_*^k(w_1w_2)\bigr)_{\#^c}.
\]
Here we use $\omega_{\#^c}$ to denote its cyclically reduced word for $\omega \in \pi_1(X)$.
Thus, each iterate is the cyclic reduction of the corresponding \(f_*\)-image.

\noindent\textbf{Convention (Identification as a subgroup).}
The free group  $\pi_1(X_{1,m,n})$  is freely generated by $a_1, a_2,\dots, a_n, b_1,$ and $c_1, \dots, c_m$. For $m'\ge m$, and $n'\ge n$ we regard $\pi_1(X_{1,m,n})$ as a subgroup of  $\pi_1(X_{1,m',n'})$. 

\subsection{Orbit datum \((1,1,n)\) with \(n\ge8\)}

Recall that 
\[
  \pi_1\bigl(X_{1,1,n}\bigr)
  \;=\;
  \bigl\langle a_1,\dots,a_n,\;b_1,\;c_1\bigr\rangle
  \;\cong\;F_{n+2}.
\]

Let 
\[
  C_{1,9}=\{\,s_1,\dots,s_{11}\}
\]
be the following eleven cyclically reduced words starting with $a_1$ (all of which lie in the subgroup \(\pi_1(X_{1,1,9})\).):
\[ \begin{aligned}& s_1= a_1 \,a_6\, \overline{a_3}, \hspace{2.25cm} s_2= a_1 \,a_7\, \overline{a_3}, \hspace{2.25cm} s_3= a_1 \,a_7\, \overline{b_1}\,a_2\,a_5\,\overline{a_3}, \\ & s_4= a_1 \,a_7\, \overline{b_1}\,a_2\,a_5\,\overline{a_4}, \hspace{1cm} s_5= a_1 \,a_7\, \overline{b_1}\,a_2\,a_5\,\overline{c_1}, \hspace{1cm}  s_6= a_1 \,a_7\, \overline{b_1}\,a_2\,a_6\,\overline{a_3},\\& s_7= a_1 \,a_8\, \overline{b_1}\,a_2\,a_5\,\overline{a_3}, \hspace{1cm} s_8= a_1 \,a_9\, \overline{b_1}\,a_2\,a_5\,\overline{a_3}, \hspace{1cm} s_9= a_1 \,a_7\, \overline{b_1}\,a_2\,a_5\,\overline{a_3}\,a_2\,c_1\,\overline{a_4},\\ & s_{10}= a_1 \,a_7\, \overline{b_1}\,a_2\,a_5\,\overline{a_3}\,b_1\,c_1\,\overline{a_4}, \hspace{1cm} s_{11}= a_1\,a_9\,\overline{b_1}\,a_2\,c_1\,\overline{a_4}.\\  \end{aligned}\]

\medskip

\medskip
Next, for \(k\ge10\) define
\[
T_{1,k} = \{ \alpha_k\}, \qquad \text{where} \quad  \alpha_k \;=\; a_1\,a_k\,\overline{b_1}\,a_2\,c_1\,\overline{a_4}.
\]
Observe that, apart from the single occurrence of \(a_k\), every letter of \(\alpha_k\) lies in \(\pi_1(X_{1,1,9})\).

\medskip

\subsubsection{Invariant semigroup}
For \(n\ge16\), let \(S_{1,n}\subset\pi_1(X_{1,1,n})\) be the semigroup generated by the elements of \(C_{1,9}\) together with \(\alpha_{10},\dots,\alpha_n\):
\[
  S_{1,n}
  \;=\;
 \left\langle C_{1,9} \cup \bigcup_{10\le k\le n} T_{1,k} \right\rangle.
\]

\begin{lem}\label{L:1115}
When \(n=16\), for each generator $g$ of $S_{1,16}$
\[
  g\;\in\;\{ s_1,\dots,s_{11},\,\alpha_{10},\dots,\alpha_{16}\},
\]
there exists an integer \(m_g>0\) and a word
\[
  \delta \;=\; a_1\,\overline{a_2}\,b_1\,\overline{a_3}\,a_4\,\overline{c_1}\,a_5\,\overline{a_6}\,\overline{a_1}
\]
such that
\[
  f_{(1,1,16)*}^{\,6}(g)
  \;=\;
  \delta\;g_1\,g_2\cdots g_{m_g}\;\overline{\delta},
\]
where each 
\(\;g_i\in C_{1,9}\cup\{\alpha_{10},\dots,\alpha_{16}\}\).
\end{lem}

\begin{proof}
For \(n=16\), the semigroup \(S_{1,16}\) has \(18\) generators.  One simply applies the formulas in \eqref{E:11n} and checks by direct computation that
\[
  f_{(1,1,16)*}^{6}(g)
  \;=\;
  \delta\,g_1\cdots g_{m_g}\,\overline\delta
\]
for each generator \(g\in S_{1,16}\).  This completes the proof.
\end{proof}

\begin{lem}\label{L:11nS1}
For every \(n\ge16\) and each generator \(g\in C_{1,9}\), we have
\[
  f_{(1,1,n)*}^{6}(g)
  \;=\;
  f_{(1,1,16)*}^{6}(g).
\]
\end{lem}

\begin{proof}
This is an immediate consequence of Lemma \ref{L:TIa} by setting \(m=1\), \(j=9\), and \(k=6\).  Since \(n\ge j+k=15\), the truncation-invariance property yields the desired equality.
\end{proof}

From the direct computation in Lemma \ref{L:1115} we obtain
\begin{equation}\label{E:11nspecial}
  f_{(1,1,16)*}^{6}\bigl(\alpha_{10}\bigr)
  \;=\;
  \delta\,s_{3}\,\alpha_{16}\,s_{1}\,s_{7}\,\alpha_{10}\,\overline\delta.
\end{equation}
Then, by truncation-invariance (the first part of Corollary \ref{C:special}), we immediately get:

\begin{lem}\label{L:11nspecial}
For every \(n\ge16\),
\[
  f_{(1,1,n)*}^{6}\bigl(\alpha_{\,n-6}\bigr)
  \;=\;
  \delta\,s_{3}\,\alpha_{n}\,s_{1}\,s_{7}\,\alpha_{10}\,\overline\delta.
\]
\end{lem}

\begin{lem}\label{L:11nalpha}
Suppose \(n\ge16\).  Then for each \(k=10,\dots,n-7\),
\[
  f_{(1,1,n)*}^{6}(\alpha_k)
  \;=\;
  f_{(1,1,15)*}^{6}(\alpha_k),
\]
and for each \(k=n-5,n-3,\dots,n\),
\[
  f_{(1,1,n)*}^{6}(\alpha_k)
  \;=\;
  f_{(1,1,n-1)*}^{6}(\alpha_{k-1}).
\]
\end{lem}

\begin{proof}
The first set of equalities is a direct application of Lemma \ref{L:TIa} with \(m=1\), \(j=8\) (so \(j+1=9\)), and \(k=6\).  The second set follows from Corollary \ref{C:tail} (the analogous truncation-invariance for the other tail generators).  
\end{proof}
\medskip

\subsubsection{Positive Action}
Fix the word
\begin{equation}\label{E:delta_11n}
  \delta \;=\; a_1\,\overline{a_2}\,b_1\,\overline{a_3}\,a_4\,
                \overline{c_1}\,a_5\,\overline{a_6}\,\overline{a_1}.
\end{equation}
Observe that \(\delta\in\pi_1(X_{1,1,n})\) for every \(n\ge8\).  The next proposition extends Lemma \ref{L:1115} to all \(n\ge16\).

\begin{prop}
For every \(n\ge16\), the semigroup \(S_{1,n}\) is invariant under \(F^6_*\).  Moreover, \(F^6_*\) acts positively on \(S_{1,n}\): for each \(\eta\in S_{1,n}\) there exists an integer \(N_\eta>0\) and generators 
\[
  w_1,\dots,w_{N_\eta}
  \;\in\;
  C_{1,9}\cup\{\alpha_{10},\dots,\alpha_n\}
\]
such that
\[
  F^6_*(\eta)
  \;=\;
  w_1\,w_2\,\cdots\,w_{N_\eta},
\]
i.e.\ no inverses appear in the product.
\end{prop}

\begin{proof}
\textbf{Base case (\(n=16\)).}  
Every \(\eta\in S_{1,16}\) can be written as
\[
  \eta \;=\; g_1\,g_2\cdots g_{m_\eta},
  \quad 
  g_i\in C_{1,9}\cup\{\alpha_{10},\dots,\alpha_{16}\}.
\]
By Lemma \ref{L:1115}, for each such generator \(g_i\) there are words 
\(\;g_{(i,1)},\dots,g_{(i,m_i)}\in C_{1,9}\cup\{\alpha_{10},\dots,\alpha_{16}\}\)
and the fixed conjugator \(\delta\) of \eqref{E:delta_11n} so that
\[
  f_{(1,1,16)*}^{6}(g_i)
  \;=\;
  \delta\;g_{(i,1)}\cdots g_{(i,m_i)}\;\overline\delta.
\]
Since no cancellation can occur between these cyclically reduced blocks, it follows that
\[
  F_{*}^6(\eta)
  \;=\;
  \bigl(f_{(1,1,16)*}^{6}(\eta)\bigr)_{\#^c}
  \;=\;
  (\,g_{(1,1)}\cdots g_{(1,m_1)}\,)\,
  (\,g_{(2,1)}\cdots g_{(2,m_2)}\,)\,
  \cdots\,
  (\,g_{(m_\eta,1)}\cdots g_{(m_\eta,m_{m_\eta})}\,),
\]
which is a positive product in \(S_{1,16}\).  Hence \(S_{1,16}\) is \(F_*\)-invariant and \(F_*\) acts positively.

\[ \begin{aligned} F_*^6\ :\ &s_1 \mapsto s_3 \alpha_{12} s_1 s_{11},  \hspace{2cm} s_2 \mapsto s_3 \alpha_{13} s_1 s_{11},  \hspace{2cm}s_3 \mapsto s_3 \alpha_{13} s_1 s_{7}\alpha_{11} s_1s_{11},\\&s_4 \mapsto s_3 \alpha_{13} s_1 s_{7}\alpha_{11} s_1\alpha_1,  \hspace{0.86cm} s_5 \mapsto s_3 \alpha_{13} s_1 s_{7}\alpha_{11},  \hspace{1.56cm} s_6 \mapsto s_3 \alpha_{13} s_1 s_{7}\alpha_{12} s_1s_{11},\\ &s_7 \mapsto s_3 \alpha_{14} s_1 s_{7}\alpha_{11} s_1s_{11},  \hspace{0.8cm} s_8 \mapsto s_3 \alpha_{15} s_1 s_{7}\alpha_{11} s_1s_{11},  \hspace{0.7cm} s_9 \mapsto s_3 \alpha_{13} s_1 s_{7}\alpha_{11} s_1s_{11}s_7 \alpha_{10},\\&s_{10} \mapsto s_3 \alpha_{13} s_1 s_{7}\alpha_{11} s_1s_{8}\alpha_{10},  \hspace{0.28cm}s_{11} \mapsto s_3 \alpha_{15} s_1 s_{7}\alpha_{10},\\&\alpha_{10} \mapsto s_3 \alpha_{16} s_1 s_{7}\alpha_{10},  \hspace{1.47cm} \alpha_{11} \mapsto s_9s_1s_7 \alpha_{10},  \hspace{1.86cm}\alpha_{12} \mapsto s_{10} s_1s_7 \alpha_{10},\\&\alpha_{13} \mapsto s_4s_1s_7 \alpha_{10},  \hspace{2.05cm} \alpha_{14} \mapsto s_5s_1s_7 \alpha_{10},  \hspace{1.9cm}\alpha_{15} \mapsto s_6s_7\alpha_{10}, \\&\alpha_{16} \mapsto s_2s_7\alpha_{10}\\ \end{aligned}\]

\medskip
\textbf{Inductive step.}  
Suppose the result holds for some \(n\ge16\).  We check it for \(n+1\).  Note that
\[
  S_{1,n+1}
  \;=\;
  \bigl\langle C_{1,9},\;\alpha_9,\dots,\alpha_{n-6},\;\alpha_{n-5},\;\alpha_{n-4},\dots,\alpha_{n+1}\bigr\rangle.
\]
By Lemma \ref{L:11nS1}, the action on each \(s\in C_{1,9}\) agrees with the \(n=16\) case; by Lemma \ref{L:11nspecial}, we see that the special block \(\alpha_{n-5}\) satisfies \[f_{(1,1,n+1)*}^{6}\bigl(\alpha_{\,n-5}\bigr)
  \;=\;
  \delta\,s_{3}\,\alpha_{n+1}\,s_{1}\,s_{7}\,\alpha_{10}\,\overline\delta;\] and by Lemma \ref{L:11nalpha}, all other tail generators \(\alpha_k\) with \(9\le k\le n-6\) or \(n-4\le k\le n+1\) satisfy the necessary truncation-invariance.  It follows that \(F^6_*\) preserves \(S_{1,n+1}\) and remains positive there.

By induction, the proposition holds for all \(n\ge16\).
\end{proof}

In the remaining cases \(8\le n\le15\), invariance and positivity follow by direct computation using the formulas in \eqref{E:11n}.  Throughout these computations we use the same fixed conjugator
\[
  \delta \;=\; a_1\,\overline{a_2}\,b_1\,\overline{a_3}\,a_4\,
                \overline{c_1}\,a_5\,\overline{a_6}\,\overline{a_1},
\]
previously defined in \eqref{E:delta_11n}.  Concretely, starting from the single generator
\[
  s_1 = a_1\,a_7\,\overline{a_3},
\]
its \(6\)-fold iterate under \(F_*\) produces, step by step, exactly the \(n+2\) generators of the semigroup \(S_{1,n}\).  For example, when \(n=8\):

\medskip

1.  Applying \(f_*^6\):
    \[
      f_{(1,1,8)*}^{6}(s_1)
      \;=\;
      \delta\,
      a_1\,a_7\,\overline{b_1}\,a_2\,a_6\,\overline{a_3}\,
      a_2\,c_1\,\overline{a_4}\,\overline{\delta}.
    \]
    We set 
    \[
      s_2 =
      a_1\,a_7\,\overline{b_1}\,a_2\,a_6\,\overline{a_3}\,
      a_2\,c_1\,\overline{a_4}.
    \]

2.  Next,
    \[
      f_{(1,1,8)*}^{6}(s_2)
      =\;
      \delta\,
      (a_1\,a_7\,\overline{b_1}\,a_2\,a_6\,\overline{a_3})\,
      (a_1\,a_8\,\overline{b_1}\,a_2\,a_5\,\overline{c_1})\,
      (a_1\,a_6\,\overline{a_3}\,a_2\,c_1\,\overline{a_4})\,
      (a_1\,a_8\,\overline{b_1}\,a_2\,a_5\,\overline{a_3}\,b_1\,c_1\,\overline{a_4})\,
      \overline{\delta}.
    \]
    From the four central blocks we define
    \[ \begin{aligned}
      &s_3 = a_1\,a_7\,\overline{b_1}\,a_2\,a_6\,\overline{a_3},\quad
      s_4 = a_1\,a_8\,\overline{b_1}\,a_2\,a_5\,\overline{c_1},\\
      &s_5 = a_1\,a_6\,\overline{a_3}\,a_2\,c_1\,\overline{a_4},\quad
      s_6 = a_1\,a_8\,\overline{b_1}\,a_2\,a_5\,\overline{a_3}\,b_1\,c_1\,\overline{a_4}.\\ \end{aligned}
    \]
3. Then we have: 
\[ f_{(1,1,8)*}^6 (s_3)\  \ =\ \delta \, (a_1 \, a_7 \, \overline{b_1} \, a_2 \, a_5 \overline{c_1})\, s_3\, s_6 \overline{\delta},\]
and define
 \[ s_7 \ =\ a_1 \, a_7 \, \overline{b_1} \, a_2 \, a_5 \overline{c_1}.\]
4.  Continuing in this fashion yields the full list of ten generators for \(S_{1,8}\), all extracted using the ``same conjugator" \(\delta\) at each step.

\begin{prop}\label{P:11ncases}
For every \(8\le n\le15\), there is a semigroup 
\(\;S_{1,n}\subset\pi_1(X_{1,1,n})\)
generated by \(n+2\) cyclically reduced elements, which is invariant under \(F_*^6\).  Moreover, \(F_*^6\) acts positively on \(S_{1,n}\).
\end{prop}

  \subsection{Orbit Datum: $(1,2,n)$ with $n\ge 7$}
Recall that the fundamental group of $X_{1,2,n}$ is a free group with $n+3$ generators: \[ \pi_1(X_{1,2,n}) = \langle a_1, \dots, a_n, b_1, c_1,c_2\rangle.\]
Let 
\[
  C_{2,13}=\{\,s_1,\dots,s_{31}\}
\]
be the following thirty-one cyclically reduced words starting with the letter $a_1$ (all of which lie in the subgroup \(\pi_1(X_{1,2,13})\)):
\[\begin{aligned}& s_1=a_1\, a_6\, \overline{a_3},\hspace{2.25cm} s_2 = a_1\, c_2\, \overline{a_5}\, c_2 \overline{a_3}, \hspace{2.35cm} s_3= a_1\, a_8\,\overline{b_1}\, a_2\, c_1\, \overline{a_4} \\&s_4= a_1\, a_9\,\overline{b_1}\, a_2\, c_1\, \overline{a_4},\hspace{1cm} s_5 = a_1\, a_6\,\overline{b_1}\, a_2\,c_2\,\overline{a_5}\,c_2\,\overline{a_3},\hspace{1cm} s_6 = a_1\, a_6\,\overline{b_1}\, a_2\,c_2\,\overline{c_1}\,a_4\,\overline{c_1},\\
&s_7=a_1\, a_7 \, \overline{b_1} \, a_2 \, a_5 \, \overline{c_2} \, \overline{a_1} \, c_1 \, \overline{a_3}, \hspace{2.85cm} s_8=a_1\, a_8 \, \overline{b_1} \, a_2 \, a_5 \, \overline{c_2} \, \overline{a_1} \, c_1 \, \overline{a_3},\\&s_9=a_1 \, c_2 \, \overline{a_5} \, \overline{a_2} \, b_1 \, \overline{a_8} \, \overline{a_1} \, a_3 \, \overline{c_1}, \hspace{2.85cm} s_{10} = a_1 \, a_6 \, \overline{b_1} \, a_2 \, c_2 \, \overline{c_1} \, a_4 \, \overline{c_1} \, \overline{a_2} \, a_3 \, \overline{c_1},\\
&s_{11} = a_1 \, a_6 \, \overline{b_1} \, a_2 \, c_2 \, \overline{c_1} \, a_4 \, \overline{c_1} \, \overline{b_1} \, a_3 \, \overline{c_1}, \hspace{1.95cm} s_{12} = a_1 \, a_{10} \, \overline{b_1} \, a_2 \, c_1 \, \overline{a_4} \, c_1 \, \overline{c_2} \, \overline{a_2} \, b_1 \, \overline{a_6} \, \overline{a_1} \, c_1 \, \overline{a_3},\\& s_{13} = a_1 \, a_{11} \, \overline{b_1} \, a_2 \, c_1 \, \overline{a_4} \, c_1 \, \overline{c_2} \, \overline{a_2} \, b_1 \, \overline{a_6} \, \overline{a_1} \, c_1 \, \overline{a_3}, \hspace{.5cm}  s_{14} = a_1 \, a_{6} \, \overline{b_1} \, a_2 \, c_2 \, \overline{c_1} \, a_4 \, \overline{c_1} \, \overline{a_2} \, b_1 \, \overline{a_{10}} \, \overline{a_1} \, a_3 \, \overline{c_1},\\& s_{15} = a_1 \, a_{6} \, \overline{b_1} \, a_2 \, c_2 \, \overline{c_1} \, a_4 \, \overline{c_1} \, \overline{a_2} \, b_1 \, \overline{a_{11}} \, \overline{a_1} \, a_3 \, \overline{c_1}, \hspace{.5cm}  s_{16} = a_1 \, a_{6} \, \overline{b_1} \, a_2 \, c_2 \, \overline{c_1} \, a_4 \, \overline{c_1} \, \overline{a_2} \, b_1 \, \overline{a_{12}} \, \overline{a_1} \, a_3 \, \overline{c_1},\\& s_{17} = a_1 \, a_{7} \, \overline{b_1} \, a_2 \, c_2 \, \overline{c_1} \, a_4 \, \overline{c_1} \, \overline{a_2} \, b_1 \, \overline{a_{10}} \, \overline{a_1} \, a_3 \, \overline{c_1}, \hspace{.5cm}  s_{18} = a_1 \, a_{7} \, \overline{b_1} \, a_2 \, c_2 \, \overline{c_1} \, a_4 \, \overline{c_1} \, \overline{a_2} \, b_1 \, \overline{a_{9}} \, \overline{a_1} \, a_3 \, \overline{c_1}, \\&s_{19} = a_1 \, c_2 \, \overline{a_5} \, \overline{a_2} \, b_1 \, \overline{a_7} \, \overline{a_1} \, a_3 \, \overline{c_2} \, a_5 \, \overline{c_2} \, \overline{a_1} \, c_1 \, \overline{a_3}, \hspace{.6cm} s_{20} = a_1 \, c_2 \, \overline{a_5} \, \overline{a_2} \, b_1 \, \overline{a_7} \, \overline{a_1} \, a_4 \, \overline{c_1}\, \overline{a_2} \, b_1 \, \overline{a_8} \, \overline{a_1} \, a_3 \, \overline{c_1}, \\&s_{21} = a_1 \, c_2 \, \overline{a_5} \, \overline{a_2} \, b_1 \, \overline{a_7} \, \overline{a_1} \, a_3 \, \overline{a_6} \, \overline{a_1} \, a_4 \, \overline{c_1} \, \overline{a_2} \, b_1 \, \overline{a_8} \, \overline{a_1} \, a_3 \, \overline{c_1},\\&s_{22} = a_1 \, a_{10} \, \overline{b_1} \, a_2 \, c_1 \, \overline{a_4} \, c_1 \, \overline{c_2} \, \overline{a_2} \, b_1 \, \overline{a_7} \, \overline{a_1} \, a_3 \, \overline{c_2} \, a_5 \, \overline{c_2} \, \overline{a_1} \, c_1 \, \overline{a_3},\\&s_{23} = a_1 \, a_{9} \, \overline{b_1} \, a_2 \, c_1 \, \overline{a_4} \, c_1 \, \overline{c_2} \, \overline{a_2} \, b_1 \, \overline{a_7} \, \overline{a_1} \, a_3 \, \overline{c_2} \, a_5 \, \overline{c_2} \, \overline{a_1} \, c_1 \, \overline{a_3},\\&s_{24} = a_1 \, a_{11} \, \overline{b_1} \, a_2 \, c_1 \, \overline{a_4} \, c_1 \, \overline{c_2} \, \overline{a_2} \, b_1 \, \overline{a_6} \, \overline{a_1} \, a_4 \, \overline{c_1} \, \overline{a_2}\, b_1 \, \overline{a_9} \, \overline{a_1} \, a_3 \, \overline{c_1},\\&s_{25} = a_1 \, a_{12} \, \overline{b_1} \, a_2 \, c_1 \, \overline{a_4} \, c_1 \, \overline{c_2} \, \overline{a_2} \, b_1 \, \overline{a_6} \, \overline{a_1} \, a_4 \, \overline{c_1} \, \overline{a_2}\, b_1 \, \overline{a_9} \, \overline{a_1} \, a_3 \, \overline{c_1},\\&s_{26} = a_1 \, c_2 \, \overline{a_5} \, \overline{a_2} \, b_1 \, \overline{a_7} \, \overline{a_1} \, a_3 \, \overline{c_2} \, a_5 \, \overline{c_2} \, \overline{a_2} \, b_1 \, \overline{a_6} \, \overline{a_1} \, a_4 \, \overline{c_1} \, \overline{a_2} \, b_1 \, \overline{a_8} \, \overline{a_1} \, a_3 \, \overline{c_1},\\&s_{27} = a_1 \, a_{12} \, \overline{b_1} \, a_2 \, c_1 \, \overline{a_4} \, c_1 \, \overline{c_2} \, \overline{a_2} \, b_1 \, \overline{a_6} \, \overline{a_1} \, a_4 \, \overline{c_1} \, \overline{a_2} \, b_1 \, \overline{a_8} \, \overline{a_1} \, a_3 \, \overline{c_2} \, a_5 \, \overline{c_2} \, \overline{a_1} \, c_1 \, \overline{a_3},\\&s_{28} = a_1 \, a_{13} \, \overline{b_1} \, a_2 \, c_1 \, \overline{a_4} \, c_1 \, \overline{c_2} \, \overline{a_2} \, b_1 \, \overline{a_6} \, \overline{a_1} \, a_4 \, \overline{c_1} \, \overline{a_2} \, b_1 \, \overline{a_8} \, \overline{a_1} \, a_3 \, \overline{c_2} \, a_5 \, \overline{c_2} \, \overline{a_1} \, c_1 \, \overline{a_3},\\&s_{29} = a_1 \, c_2 \, \overline{a_5} \, \overline{a_2} \, b_1 \, \overline{a_7} \, \overline{a_1} \, a_3 \, \overline{c_2} \, a_5 \, \overline{c_2} \, \overline{a_1} \, c_1 \, \overline{a_4} \, c_1 \, \overline{c_2} \,  \overline{a_2} \, b_1 \,  \overline{a_6} \,  \overline{a_1} \, a_4 \,  \overline{c_1} \,  \overline{a_2} \, b_1 \,  \overline{a_8} \,  \overline{a_1} \, a_3 \,  \overline{c_1},\\&s_{30} =a_1 \, c_2 \,  \overline{a_5} \,  \overline{a_2} \, b_1 \,  \overline{a_7} \,  \overline{a_1} \, a_3 \,  \overline{c_2} \, a_5 \,  \overline{c_2} \,  \overline{a_1} \, c_1 \,  \overline{a_3}\,a_2 \, c_1 \,  \overline{a_4} \, c_1 \,  \overline{c_2} \,  \overline{a_2} \, b_1 \,  \overline{a_6} \,  \overline{a_1} \, a_4 \,  \overline{c_1} \,  \overline{a_2} \, b_1 \,  \overline{a_8}\, \overline{a_1},a_3\, \overline{c_1},\\&s_{31} =a_1 \, c_2 \,  \overline{a_5} \,  \overline{a_2} \, b_1 \,  \overline{a_7} \,  \overline{a_1} \, a_3 \,  \overline{c_2} \, a_5 \,  \overline{c_2} \,  \overline{a_1} \, c_1 \,  \overline{a_3}\,b_1 \, c_1 \,  \overline{a_4} \, c_1 \,  \overline{c_2} \,  \overline{a_2} \, b_1 \,  \overline{a_6} \,  \overline{a_1} \, a_4 \,  \overline{c_1} \,  \overline{a_2} \, b_1 \,  \overline{a_8}\, \overline{a_1},a_3\, \overline{c_1}.\\
\end{aligned}\]

Next, for each \(n\ge13\) define two cyclically reduced words
\[ \begin{aligned} & \alpha_n = a_1 \, a_6 \, \overline{b_1} \, a_2 \, c_2 \, \overline{c_1} \, a_4 \, \overline{c_1} \, \overline{a_2} \, b_1 \, \overline{a_n} \, \overline{a_1} \, a_3 \, \overline{c_1}\\& \beta_n= a_1 \, a_n \, \overline{b_1} \, a_2 \, c_1 \, \overline{a_4} \, c_1 \, \overline{c_2} \, \overline{a_2} \, b_1 \, \overline{a_6} \, \overline{a_1} \, a_4 \, \overline{c_1} \, \overline{a_2} \, b_1 \, \overline{a_8} \, \overline{a_1} \, a_3 \, \overline{c_1}.\\ \end{aligned} \]

Let
\[
  T_{2,n} \;=\;\{\alpha_n,\;\beta_n\}.
\]
Observe that, apart from the single occurrence of \(\overline{a_n}\) in \(\alpha_n\) (respectively \(a_n\) in \(\beta_n\)), every letter in \(\alpha_n\) (resp.\ \(\beta_n\)) lies in the subgroup 
\(\pi_1\bigl(X_{1,2,9}\bigr)
 = \langle a_1,\dots,a_9,\;b_1,\;c_1,\;c_2\rangle.
\)

\subsubsection{Invariant semigroup}
For \(n\ge19\), let \(S_{2,n}\subset\pi_1(X_{1,2,n})\) be the semigroup generated by the elements of \(C_{2,13}\) together with $T_{2,13}, \dots T_{2,n}$:

\[
  S_{2,n}
  \;=\;
 \left\langle \left( C_{2,13} \cup T_{2,13} \right) \cup \bigcup_{14\le k\le n} T_{2,k} \right\rangle.
\]

Since every generator of $S_{2,n}$ is cyclically reduced and begins with the common letter $a_1$, $S_{2,n}$ clearly satisfies the concatenation-compatibility property (Definition \ref{D:nocancel}).  Likewise, the core-tail generation property (Definition \ref{D:coretail}) is immediate from our choice of core $C_{2, 13} \cup T_{2,13}$ and tail-sets $T_{2,n}$ with $n\ge 14$.  Hence, by the Core-Tail Induction Principle (Proposition \ref{P:ctinduction}), it remains only to verify the base case $n = 19$ by direct computation.

Let $\delta_{2}$ be the word given by
\begin{equation}\label{E:delta_12n}
\delta_{2} \ =\  a_1\, \overline{a_2}\, b_1\overline{a_3}\, a_4\, \overline{c_1}\, a_5\, \overline{c_2} \, \overline{a_1}. 
\end{equation}

\vspace{1ex}

The family with orbit datum \((1,2,n)\) is handled in exactly the same way as the \((1,1,n)\) case.  We take \(F_*^5\) and \(n=18\) as our base case for the induction.  In the case $n=19$, the following lemma can be verified by direct computation using the formulas in \eqref{E:12n}.

\begin{lem}\label{L:1218}
When \(n=19\), we have
\[
  S_{2,19}
  \;=\;
  \langle\,s_1,\dots,s_{31},\;\alpha_{13},\dots,\alpha_{19},\;\beta_{13},\dots,\beta_{19}\rangle.
\]
Then for each generator \(g\in S_{2,19}\) there is an integer \(m_g>0\) and a fixed word \(\delta_2\) (as defined in \eqref{E:delta_12n}) such that
\[
  f_{(1,2,19)*}^{5}(g)
  \;=\;
  \delta_2\;g_1\,g_2\cdots g_{m_g}\;\overline{\delta_2},
\]
where each generator \(g_i\in S_{2,19}\).  
Moreover,, we have 
\[ f_{(1,2,19)*}^{5}(\alpha_{14})\ =\ \delta_2 \,s_{15}\, s_2\,s_{18}\, s_{19}\, \beta_{19}\,\overline{\delta_2}\]
and \[ f_{(1,2,19)*}^{5}(\beta_{14})\ =\ \delta_2 \alpha_{19}\,\,s_{2}\, s_7\,s_{23}\, s_{24}\,s_{19}\, \beta_{13}\,\overline{\delta_2}.\]
\end{lem}

By Proposition \ref{P:actionP} and the truncation-invariance lemmas, we obtain:

\begin{lem}
Fix \(n\ge19\).  Then the action \(f_{(1,2,n)*}^{5}\) on the generators of the semigroup \(S_{2,n}\) satisfies:
\begin{enumerate}
  \item For each \(s_i\in C_{2,13} \cup T_{2,13}\),
    \[
      f_{(1,2,n)*}^{5}(s_i)
      \;=\;
      f_{(1,2,18)*}^{5}(s_i).
    \]

  \item For each \(i=14,\dots,n-6\),
    \[
      f_{(1,2,n)*}^{5}(\alpha_i)
      \;=\;
      f_{(1,2,18)*}^{5}(\alpha_i),
      \quad
      f_{(1,2,n)*}^{5}(\beta_i)
      \;=\;
      f_{(1,2,18)*}^{5}(\beta_i).
    \]

  \item For the special index \(i=n-5\),
    \[
      f_{(1,2,n)*}^{5}(\alpha_{n-5})
      \;=\;
      \delta_2\,s_{15}\,s_{2}\,s_{18}\,s_{19}\,\beta_{n}\,\overline\delta_2,
    \]
    \[
      f_{(1,2,n)*}^{5}(\beta_{n-5})
      \;=\;
      \delta_2\,\alpha_{n}\,s_{2}\,s_{7}\,s_{23}\,s_{24}\,s_{19}\,\beta_{13}\,\overline\delta_2.
    \]

  \item For each \(i=n-4,n-3,\dots,n\),
    \[
      f_{(1,2,n)*}^{5}(\alpha_i)
      \;=\;
      f_{(1,2,n-1)*}^{5}(\alpha_{i-1}),
      \quad
      f_{(1,2,n)*}^{5}(\beta_i)
      \;=\;
      f_{(1,2,n-1)*}^{5}(\beta_{i-1}).
    \]
\end{enumerate}
\end{lem} 

Since the image of every generator has the common conjugator $\delta_2$, by the induction argument we have:

\begin{prop}
For every \(n\ge19\), the semigroup \(S_{2,n}\) is invariant under \(F^5_*\).  Moreover, \(F^5_*\) acts positively on \(S_{2,n}\): for each \(\eta\in S_{2,n}\) there exists an integer \(N_\eta>0\) and generators 
\[
  w_1,\dots,w_{N_\eta}
  \;\in\;
  C_{2,13}\cup\{\alpha_{13},\dots,\alpha_n,\beta_{13},\dots, \beta_n\}
\]
such that
\[
  F^5_*(\eta)
  \;=\;
  w_1\,w_2\,\cdots\,w_{N_\eta},
\]
i.e.\ no inverses appear in the product.
\end{prop}

Similar to the previous subsection, for the remaining cases \(7\le n\le18\), the invariance and positivity can be verified by direct computation using the formulas in \eqref{E:12n}.  In fact, one shows that iterating the single element 
\[
  a_1\,a_6\,\overline{a_3}
\]
under \(F_*^5\) produces exactly the \(2n+7\) generators of the invariant semigroup \(S_{2,n}\).

\begin{prop}
For every \(7\le n\le18\), there is a semigroup 
\(\;S_{2,n}\subset\pi_1(X_{1,2,n})\)\ 
generated by \(2n+7\) cyclically reduced elements, which is invariant under \(F_*^5\).  Moreover, \(F_*^5\) acts positively on \(S_{2,n}\).
\end{prop}

\subsection{Orbit datum \((1,n,2)\) with \(n\ge7\)}

The analysis for orbit datum \((1,n,2)\) parallels the \((1,1,n)\) and \((1,2,n)\) cases.  In this subsection we will:

\begin{enumerate}
  \item List the generators of the invariant semigroup \(S_{n,2}\).
  \item Give the base-case formulas for the action when \(n=17\) (the smallest nontrivial value).
  \item Exhibit the single tail-generator whose iterates produce the remaining semigroup elements for \(7\le n\le16\).
\end{enumerate}

By the Core-Tail Induction Principle (Proposition \ref{P:ctinduction} in Section \ref{S:homotopy}), once we verify that the family \(\{S_{n,2}\}_{n\ge 17}\) has the core-tail generation property and check the base case by direct computation, it follows that \(\{S_{n,2}\}\)  remains \(F_*^5\)-invariant and positive for all \(n\ge17\).  The finite-range argument then covers the cases \(7\le n\le16\).

\subsubsection{the Invariant semigroup}
Let 
\[
 C_{11,2}=\{\,s_1,\dots,s_{25}\}
\]
be the following twenty-seven cyclically reduced words starting with the common letter $a_1$ (all of which lie in the subgroup \(\pi_1(X_{1,11,2})\)):
\[ \begin{aligned}&s_1=a_1\, b_1\, \overline{c_3}\, a_2\, \overline{c_6}\, \overline{c_1}, \hspace{1.8cm}s_2=a_1\, b_1\, \overline{c_3}\, c_1\,c_6\, \overline{c_2}, \hspace{2.3cm} s_3=a_1\, c_4\, \overline{a_2}\, c_2\, \overline{c_8}\, \overline{c_1},\\
&s_4=a_1\, c_4\, \overline{a_2}\, c_2\, \overline{c_9}\, \overline{c_1}, \hspace{1.8cm} s_5=a_1\, c_4\, \overline{a_2}\, c_2\, \overline{c_{10}}\, \overline{c_1}, \hspace{1.8cm} s_6=\,a_1 \, b_1 \, \overline{c_3} \, a_2 \, \overline{c_5} \, \overline{a_1} \, c_2 \, \overline{c_6} \, \overline{c_1}, \\
&s_7=\,a_1 \, b_1 \, \overline{c_3} \, c_1 \, c_5 \, \overline{a_2} \, c_2 \, \overline{c_7} \, \overline{c_1}, \hspace{2.7cm} s_8=\,a_1 \, b_1 \, \overline{c_3} \, c_1 \, c_5 \, \overline{a_2} \, c_2 \, \overline{c_8} \, \overline{c_1}, \\&s_9=a_1 \, b_1 \, \overline{c_3} \, a_2 \, \overline{c_5} \, \overline{c_1} \, b_1 \, \overline{c_4} \, \overline{a_1} \, c_2 \, \overline{c_6} \, \overline{c_1}, \hspace{1.5cm} s_{10}=a_1 \, b_1 \, \overline{c_3} \, a_2 \, \overline{c_5} \, \overline{c_1} \, c_3 \, \overline{c_4} \, \overline{a_1} \, c_2 \, \overline{c_6} \, \overline{c_1}, \\& s_{11}=a_1 \, b_1 \, \overline{c_3} \, c_1 \, c_5  \, \overline{a_2} \, c_3 \, \overline{b_1} \, \overline{a_1} \, c_1 \, c_7 \, \overline{c_2},\hspace{1.33cm} s_{12} = a_1 \, b_1 \, \overline{c_3} \, a_2 \, \overline{c_5} \, \overline{c_1} \, c_3 \, \overline{b_1} \, \overline{c_2} \, a_2 \, \overline{c_4} \, \overline{a_1} \, c_2 \, \overline{c_6} \, \overline{c_1},\\
&s_{13}=a_1 \, b_1 \, \overline{c_3} \, a_2 \, \overline{c_5} \, \overline{c_1} \, c_3 \, \overline{b_1} \, \overline{a_1} \, c_1 \, c_{10} \, \overline{c_2} \, a_2 \, \overline{c_4} \, \overline{a_1} \, c_2 \, \overline{c_7} \, \overline{c_1},\\
&s_{14} = a_1 \, b_1 \, \overline{c_3} \, c_1 \, c_5 \, \overline{a_2} \, c_3 \, \overline{b_1} \, \overline{a_1} \, c_1 \, c_8 \, \overline{c_2} \, a_2 \, \overline{c_4} \, \overline{a_1} \, c_1 \, c_6 \, \overline{c_2},\\
&s_{15} = a_1 \, c_5  \, \overline{a_2} \, c_3  \, \overline{b_1}  \, \overline{a_1} \, c_1 \, c_7  \, \overline{c_2} \, a_2  \, \overline{c_5}  \, \overline{c_1} \, c_3  \, \overline{b_1}  \, \overline{a_1} \, c_1 \, c_9  \, \overline{c_2} \, a_2  \, \overline{c_4}  \, \overline{a_1} \, c_2  \, \overline{c_7}  \, \overline{c_1},\\
&s_{16}= a_1  \, b_1  \, \overline{c_3}  \, c_1  \, c_5  \, \overline{a_2}  \, c_3  \, \overline{b_1}  \, \overline{a_1}  \, c_1  \, c_7  \, \overline{c_2}  \, a_2  \, \overline{c_5}  \, \overline{c_1}  \, c_3  \, \overline{b_1}  \, \overline{a_1}  \, c_1  \, c_9  \, \overline{c_2}  \, a_2  \, \overline{c_4}  \, \overline{a_1}  \, c_1  \, c_6  \, \overline{c_2},\\
&s_{17}= a_1  \, b_1  \, \overline{c_3}  \, c_1  \, c_5  \, \overline{a_2}  \, c_3  \, \overline{b_1}  \, \overline{a_1}  \, c_1  \, c_7  \, \overline{c_2}  \, a_2  \, \overline{c_5}  \, \overline{c_1}  \, c_3  \, \overline{b_1}  \, \overline{a_1}  \, c_1  \, c_9  \, \overline{c_2}  \, a_2  \, \overline{c_4}  \, \overline{a_1}  \, c_2  \, \overline{c_7}  \, \overline{c_1},\\
&s_{18}= a_1  \, c_4  \, \overline{b_1}  \, c_1  \, c_5  \, \overline{a_2}  \, c_3  \, \overline{b_1}  \, \overline{a_1}  \, c_1  \, c_7  \, \overline{c_2}  \, a_2  \, \overline{c_5}  \, \overline{c_1}  \, c_3  \, \overline{b_1}  \, \overline{a_1}  \, c_1  \, c_9  \, \overline{c_2}  \, a_2  \, \overline{c_4}  \, \overline{a_1}  \, c_2  \, \overline{c_7}  \, \overline{c_1},\\
&s_{19}= a_1  \, c_4  \, \overline{c_3}  \, c_1  \, c_5  \, \overline{a_2}  \, c_3  \, \overline{b_1}  \, \overline{a_1}  \, c_1  \, c_7  \, \overline{c_2}  \, a_2  \, \overline{c_5}  \, \overline{c_1}  \, c_3  \, \overline{b_1}  \, \overline{a_1}  \, c_1  \, c_9  \, \overline{c_2}  \, a_2  \, \overline{c_4}  \, \overline{a_1}  \, c_2  \, \overline{c_7}  \, \overline{c_1},\\
&s_{20}= a_1 \, b_1 \, \overline{c_3} \, a_2 \, \overline{c_5} \, \overline{c_1} \, c_3 \, \overline{b_1} \, \overline{a_1} \, c_1 \, c_{10} \, \overline{c_2} \, a_2 \, \overline{c_4} \, \overline{a_1} \, c_2 \, \overline{c_6} \, \overline{c_1} \, c_3 \, \overline{b_1} \, \overline{a_1} \, c_1 \, c_8 \, \overline{c_2} \, a_2 \, \overline{c_4} \, \overline{a_1} \, c_1 \, c_6 \, \overline{c_2},\\
&s_{21}= a_1 \, b_1 \, \overline{c_3} \, a_2 \, \overline{c_5} \, \overline{c_1} \, c_3 \, \overline{b_1} \, \overline{a_1} \, c_1 \, c_{10} \, \overline{c_2} \, a_2 \, \overline{c_4} \, \overline{a_1} \, c_2 \, \overline{c_6} \, \overline{c_1} \, c_3 \, \overline{b_1} \, \overline{a_1} \, c_1 \, c_9 \, \overline{c_2} \, a_2 \, \overline{c_4} \, \overline{a_1} \, c_2 \, \overline{c_7} \, \overline{c_1},\\
&s_{22}= a_1 \, b_1 \, \overline{c_3} \, a_2 \, \overline{c_5} \, \overline{c_1} \, c_3 \, \overline{b_1} \, \overline{a_1} \, c_1 \, c_{11} \, \overline{c_2} \, a_2 \, \overline{c_4} \, \overline{a_1} \, c_2 \, \overline{c_6} \, \overline{c_1} \, c_3 \, \overline{b_1} \, \overline{a_1} \, c_1 \, c_8 \, \overline{c_2} \, a_2 \, \overline{c_4} \, \overline{a_1} \, c_1 \, c_6 \, \overline{c_2},\\
&s_{23}= a_1 \, c_4 \, \overline{a_2} \, c_2 \, b_1 \, \overline{c_3} \, c_1 \,c_5 \, \overline{a_2} \, c_3 \, \overline{b_1} \, \overline{a_1} \, c_1\, c_7\, \overline{c_2} \, a_2 \, \overline{c_5} \, \overline{c_1} \, c_3 \, \overline{b_1} \, \overline{a_1} \, c_1 \, c_9 \, \overline{c_2} \, a_2 \, \overline{c_4} \, \overline{a_1} \, c_2 \, \overline{c_7} \, \overline{c_1},\\
&s_{24}=a_1 \, b_1 \, \overline{c_3} \, a_2 \, \overline{c_5} \, \overline{c_1} \, c_3 \, \overline{b_1} \, \overline{a_1} \, c_1 \, c_{10} \, \overline{c_2} \, a_2 \, \overline{c_4} \, \overline{a_1} \, c_2 \, \overline{c_6} \, \overline{c_1} \, c_3 \, \overline{b_1} \, \overline{a_1} \, c_1 \, c_8 \, \overline{c_2} \, a_2 \, \overline{c_5} \, \overline{c_1} \, c_3 \, \overline{b_1} \, \overline{a_1} \, c_1 \, c_9 \, \overline{c_2} \, a_2 \, \overline{c_4} \, \overline{a_1} \, c_2 \, \overline{c_7} \, \overline{c_1},\\
&s_{25}=a_1 \, b_1 \, \overline{c_3} \, c_1 \, c_5 \, \overline{a_2} \, c_3 \, \overline{b_1} \, \overline{a_1} \, c_1 \, c_7 \, \overline{c_2} \, a_2 \, \overline{c_5} \, \overline{c_1} \, c_3 \, \overline{b_1} \, \overline{a_1} \, c_1 \, c_{10} \, \overline{c_2} \, a_2 \, \overline{c_4} \, \overline{a_1} \, c_2 \, \overline{c_6} \, \overline{c_1} \, c_3 \, \overline{b_1} \, \overline{a_1} \, c_1 \, c_8 \, \overline{c_2} \, a_2 \, \overline{c_4} \, \overline{a_1} \, c_1 \, c_6 \, \overline{c_2},\\
&s_{26}=a_1 \, b_1 \, \overline{c_3} \, a_2 \, \overline{c_5} \, \overline{c_1} \, c_3 \, \overline{b_1} \, \overline{a_1} \, c_1 \, c_{10} \, \overline{c_2} \, a_2 \, \overline{c_4} \, \overline{a_1} \, c_2 \, \overline{c_6} \, \overline{c_1} \, c_3 \, \overline{b_1} \, \overline{a_1} \, c_1 \, c_8 \, \overline{c_2} \, a_2 \, \overline{c_4} \, \overline{a_1} \, c_1 \, c_7 \, \overline{c_2} \, a_2 \, \overline{c_5} \\ &\phantom{AAAAAA} \overline{c_1} \, c_3 \, \overline{b_1} \, \overline{a_1} \, c_1 \, c_9 \, \overline{c_2} \, a_2 \, \overline{c_4} \, \overline{a_1} \, c_2 \, \overline{c_7} \, \overline{c_1},\\
&s_{27}=a_1 \, b_1 \, \overline{c_3} \, a_2 \, \overline{c_5} \, \overline{c_1} \, c_3 \, \overline{b_1} \, \overline{a_1} \, c_1 \, c_{10} \, \overline{c_2} \, a_2 \, \overline{c_4} \, \overline{a_1} \, c_2 \, \overline{c_6} \, \overline{c_1} \, c_3 \, \overline{b_1} \, \overline{a_1} \, c_1 \, c_8 \, \overline{c_2} \, a_2 \, \overline{c_4} \, \overline{a_1} \, c_1 \, c_6 \, \overline{a_2} \, c_3 \, \overline{b_1} \, \overline{a_1} \\ &\phantom{AAAAAA} c_1 \, c_7 \, \overline{c_2} \, a_2 \, \overline{c_5} \, \overline{c_1} \, c_3 \, \overline{b_1} \, \overline{a_1} \, c_1 \, c_9 \, \overline{c_2} \, a_2 \, \overline{c_4} \, \overline{a_1} \, c_2 \, \overline{c_7} \, \overline{c_1}.\\
\end{aligned}\]

Next, for each \(n\ge11\) define two cyclically reduced words
\[ \begin{aligned} & \alpha_n = a_1 \, c_4 \, \overline{a_2} \, c_2 \, \overline{c_{n}} \, \overline{c_1},\\& \beta_n= a_1 \, b_1 \, \overline{c_3} \, a_2 \, \overline{c_5} \, \overline{c_1} \, c_3 \, \overline{b_1} \, \overline{a_1} \, c_1 \, c_n \, \overline{c_2} \, a_2 \, \overline{c_4} \, \overline{a_1} \, c_2 \, \overline{c_6} \, \overline{c_1}.\\ \end{aligned} \]

Let
\[
  T_{n,2} \;=\;\{\alpha_n,\;\beta_n\}.
\]

For \(n\ge17\), let \(S_{n,2}\subset\pi_1(X_{1,n,2})\) be the semigroup generated by the elements of \(C_{11,2}\) together with $T_{11,2}, \dots T_{n,2}$:
\[
  S_{n,2}
  \;=\;
 \left\langle  \left( C_{11,2} \cup T_{11,2} \right) \cup \bigcup_{12\le k\le n} T_{k,2} \right\rangle.
\]

From our explicit lists of the fixed core $C_{11,2} \cup T_{11,2}$ and the tail-sets $T_{k,2}$ for $k\ge 12$, it follows immediately that $\{S_{n,2}\}_{n\ge 17}$ satisfies the concatenation-compatibility and core-tail generation property.

\subsubsection{the Base case : $n=17$}

Let $\delta_{2}$ be the word given by
\begin{equation}\label{E:delta_1n2}
\delta_{3} \ =\  c_1\, \overline{a_1}\, c_2\overline{a_2}\, c_3\, \overline{b_1}\, c_4\, \overline{c_5} \, \overline{c_1}. 
\end{equation}

\begin{lem}\label{L:1162}
When \(n=17\), for each generator $g$ of $S_{17,2}$
\[
  g\;\in\;\{ s_1,\dots,s_{27},\,\alpha_{11},\dots,\alpha_{17},\, \beta_{11}, \dots, \beta_{17}\},
\]
such that
\[
  f_{(1,17,2)*}^{\,5}(g)
  \;=\;
  \delta_3\;g_1\,g_2\cdots g_{m_g}\;\overline{\delta_3},
\]
where each 
\(\;g_i\in\;\{ s_1,\dots,s_{27},\,\alpha_{10},\dots,\alpha_{17},\, \beta_{10}, \dots, \beta_{17}\}\). Moreover,, we have 
\[ f_{(1,17,2)*}^{5}(\alpha_{12})\ =\ \delta_3 \,s_{7}\, \beta_{17}\,s_4\overline{\delta_3}\]
and \[ f_{(1,7,2)*}^{5}(\beta_{12})\ =\ \delta_3 \, s_{20}\,\alpha_{17}\,\,s_{17}\, \, \beta_{11}\,s_3\,\overline{\delta_3}.\]
\end{lem}

\begin{proof}
For \(n=17\), the semigroup \(S_{17,2}\) has \(41\) generators.  One simply applies the formulas in \eqref{E:1n2} and checks by direct computation that
\[
  f_{(1,17,2)*}^{5}(g)
  \;=\;
  \delta_3\,g_1\cdots g_{m_g}\,\overline{\delta_3}
\]
for each generator \(g\in S_{17,2}\).  This completes the proof.
\end{proof}

Again by the induction argument, we have:

\begin{prop}
For every \(n\ge17\), the semigroup \(S_{n,2}\) is invariant under \(F^5_*\).  Moreover, \(F^5_*\) acts positively on \(S_{n,2}\): for each \(\eta\in S_{n,2}\) there exists an integer \(N_\eta>0\) and generators 
\[
  w_1,\dots,w_{N_\eta}
  \;\in\;
  C_{11,2}\cup\{\alpha_{11},\dots,\alpha_n,\beta_{11},\dots, \beta_n\}
\]
such that
\[
  F^5_*(\eta)
  \;=\;
  w_1\,w_2\,\cdots\,w_{N_\eta},
\]
i.e.\ no inverses appear in the product.
\end{prop}

\subsubsection{the case : $n=7, \dots, 16$}
For the remaining cases \(7\le n\le16\), the invariance and positivity can be verified by direct computation using the formulas in \eqref{E:1n2}.  In fact, one shows that iterating the single element 
\[
  a_1\, b_1\, \overline{c_3} \, a_2\, \overline{c_6} \, \overline{c_1}
\]
under \(F_*^5\) produces exactly the \(2n+7\) generators of the invariant semigroup \(S_{n,2}\).

\begin{prop}
For every \(7\le n\le16\), there is a semigroup 
\(\;S_{n,2}\subset\pi_1(X_{1,n,2})\)\ 
generated by \(2n+7\) cyclically reduced elements, which is invariant under \(F_*^5\).  Moreover, \(F_*^5\) acts positively on \(S_{n,2}\).
\end{prop}

\subsection{Orbit datum \((1,n,n+1)\) with \(n\ge4\)}

The treatment of the \((1,n,n+1)\) case follows the same overall strategy as in the previous subsections.  In particular, we will:

\begin{enumerate}
  \item Specify a fixed finite core \(C_{6,7}\subset\pi_1(X_{1,6,7})\) and, for each \(n\ge13\), tail-sets \(T_{n,n+1}\) and \(\hat T_{n,n+1}\)
  \item List these generators explicitly to exhibit the semigroup \(S_{n,n+1}\).
  \item Verify by direct calculation the base case \(n=13\) (the smallest nontrivial orbit) that \(S_{13,14}\) is \(F_*^6\)-invariant and positive.
  \item Identify the single generator whose iterates produce the remaining semigroup elements for \(4\le n\le 12\).
\end{enumerate}

Once the concatenation-compatibility and core-tail generation property is evident from these lists and the base case and the cases with the finite range are checked, the Core-Tail Induction Principle (Proposition \ref{P:ctinduction}) guarantees that \(S_{n,n+1}\) remains \(F_*^6\)-invariant and positive for all \(n\ge4\).

\subsubsection{The core generators and tail-sets. }
Let 
\[
 C_{6,7}=\{\,s_1,\dots,s_{27}\}
\]
be the following twenty-seven cyclically reduced words starting with the common letter $a_1$ (all of which lie in the subgroup \(\pi_1(X_{1,6,7})\)):
\[
\begin{aligned}&s_1=a_1 \, a_5 \, \overline{c_3} \, \overline{a_2} \, c_1 \, \overline{b_1}, \hspace{2.25cm} s_2=a_1 \, c_3 \, \overline{a_4} \, \overline{a_2} \, c_1 \, \overline{b_1},\hspace{2.25cm} s_3=a_1 \, a_5 \, \overline{c_4} \, \overline{a_1} \, a_3 \, \overline{c_2} \, \overline{a_2} \, c_1 \, \overline{b_1},\\
&s_4=a_1 \,  a_5 \, \overline{c_4} \, \overline{a_1} \,  b_1 \, \overline{c_1} \,  a_2 \,  a_4 \, \overline{c_2}, \hspace{1cm}s_5=a_1 \,  a_6 \, \overline{c_4} \, \overline{a_1} \,  a_3 \, \overline{c_2} \, \overline{a_2} \,  c_1 \, \overline{b_1}, \hspace{1cm} s_6=a_1 \,  a_6 \, \overline{c_5} \, \overline{a_1} \,  b_1 \, \overline{c_1} \,  a_2 \,  c_2 \, \overline{a_3},\\
&s_7=a_1 \, a_7 \, \overline{c_6} \, \overline{a_1} \, b_1 \, \overline{c_1} \, a_2 \, c_2 \, \overline{a_3}, \hspace{1cm}s_8=a_1 \, c_4 \, \overline{a_5} \, \overline{a_1} \, a_3 \, \overline{c_2} \, \overline{a_2} \, c_1 \, \overline{b_1}, \hspace{1cm}s_9=a_1 \, c_5 \, \overline{a_7} \, \overline{a_1} \, b_1 \, \overline{c_1} \, a_2 \, c_2 \, \overline{a_3},\\
&s_{10} = a_1 \, c_6 \, \overline{a_7} \, \overline{a_1} \, b_1 \, \overline{c_1} \, a_2 \, c_2 \, \overline{a_3}, \hspace{3.75cm} s_{11}=a_1 \, a_5 \, \overline{c_4} \, \overline{a_1} \, b_1 \, \overline{c_1} \, a_2 \, a_4 \, \overline{c_3} \, \overline{a_1} \, a_3 \, \overline{c_2},\\
&s_{12}=a_1 \, a_5 \, \overline{c_4} \, \overline{a_1} \, b_1 \, \overline{c_1} \, a_2 \, a_4 \, \overline{c_3} \,
\overline{a_1} \, c_2 \, \overline{a_3},\hspace{2.5cm} s_{13}=a_1 \, a_5 \, \overline{c_4} \, \overline{a_1} \, b_1 \, \overline{c_1} \, a_2 \, a_4 \, \overline{c_3} \,
\overline{a_2} \, c_1 \, \overline{b_1},\\
&s_{14}=a_1 \, a_5 \, \overline{c_4} \, \overline{a_1} \, b_1 \, \overline{c_1} \, a_2 \, c_3 \, \overline{a_4} \,
\overline{a_2} \, c_1 \, \overline{b_1},\hspace{2.53cm} s_{15}=a_1 \, a_5 \, \overline{c_4} \, \overline{a_1} \, b_1 \, \overline{c_1} \, a_2 \, a_4 \, \overline{c_3} \,
\overline{a_1} \, a_3 \, \overline{b_1} \, c_1 \, c_2 \, \overline{a_3},\\
&s_{16}=a_1 \, a_5 \, \overline{c_4} \, \overline{a_1} \, b_1 \, \overline{c_1} \, a_2 \, a_4 \, \overline{c_3} \,
\overline{a_1} \, a_3 \, \overline{c_2} \, \overline{a_2} \, c_1 \, \overline{b_1},\hspace{1.3cm} s_{17}=a_1 \, a_5 \, \overline{c_4} \, \overline{a_1} \, b_1 \, \overline{c_1} \, a_2 \, a_4 \, \overline{c_3} \,
\overline{a_1} \, a_3 \, \overline{c_2} \, \overline{c_1} \, b_1 \, \overline{a_3},\\
&s_{18}=a_1 \, a_5 \, \overline{c_4} \, \overline{a_1} \, b_1 \, \overline{c_1} \, a_2 \, a_4 \, \overline{c_3} \,
\overline{a_1} \, c_2 \, \overline{a_4} \, \overline{a_2} \, c_1 \, \overline{b_1},\hspace{1.3cm} s_{19}=a_1 \, a_5 \, \overline{c_4} \, \overline{a_1} \, b_1 \, \overline{c_1} \, a_2 \, c_3 \, \overline{a_5} \,
\overline{a_1} \, a_3 \, \overline{c_2} \, \overline{a_2} \, c_1 \, \overline{b_1},\\
&s_{20}=a_1 \, a_6 \, \overline{c_5} \, \overline{a_1} \, b_1 \, \overline{c_1} \, a_2 \, a_4 \, \overline{c_3} \,
\overline{a_1} \, a_3 \, \overline{c_2} \, \overline{a_2} \, c_1 \, \overline{b_1},\hspace{1.3cm} s_{21}=a_1 \, a_7 \, \overline{c_5} \, \overline{a_1} \, b_1 \, \overline{c_1} \, a_2 \, a_4 \, \overline{c_3} \,
\overline{a_1} \, a_3 \, \overline{c_2} \, \overline{a_2} \, c_1 \, \overline{b_1},\\
&s_{22}=a_1 \, c_4 \, \overline{a_6} \, \overline{a_1} \, b_1 \, \overline{c_1} \, a_2 \, a_4 \, \overline{c_3} \,
\overline{a_1} \, a_3 \, \overline{c_2} \, \overline{a_2} \, c_1 \, \overline{b_1},\hspace{1.3cm} s_{23}=a_1 \, c_5 \, \overline{a_6} \, \overline{a_1} \, b_1 \, \overline{c_1} \, a_2 \, a_4 \, \overline{c_3} \,
\overline{a_1} \, a_3 \, \overline{c_2} \, \overline{a_2} \, c_1 \, \overline{b_1},\\
&s_{24}=a_1 \, a_5 \, \overline{c_4} \, \overline{a_1} \, b_1 \, \overline{c_1} \, a_2 \, a_4 \, \overline{c_3} \,
\overline{a_1} \, a_3 \, \overline{c_2} \, \overline{a_2} \, c_1 \, c_2 \, \overline{a_3},\hspace{0.9cm} s_{25}=a_1 \, a_5 \, \overline{c_4} \, \overline{a_1} \, b_1 \, \overline{c_1} \, a_2 \, a_4 \, \overline{c_3} \,
\overline{a_1} \, a_3 \, \overline{c_2} \, \overline{c_1} \, a_2 \, c_2 \, \overline{a_3},\\
&s_{26}=a_1 \, a_5 \, \overline{c_4} \, \overline{a_1} \, b_1 \, \overline{c_1} \, a_2 \, a_4 \, \overline{c_3} \,
\overline{a_1} \, a_3 \, \overline{c_2} \, \overline{a_2} \, b_1 \, \overline{c_1} \, a_2 \, c_2 \, \overline{a_3},\\
&s_{27}=a_1 \, a_5 \, \overline{c_4} \, \overline{a_1} \, b_1 \, \overline{c_1} \, a_2 \, a_4 \, \overline{c_3} \,
\overline{a_1} \, a_3 \, \overline{c_2} \, \overline{a_2} \, c_1 \, \overline{b_1} \, a_2 \, c_2 \, \overline{a_3}.\\
\end{aligned}\]

For this orbit data, there are two kinds of tail sets. 
For each $n\ge 7$ define four cyclically reduced words
\[ \begin{aligned}
& \alpha^{(1)}_n = a_1 \, a_{n+1} \, \overline{c_{n-1}} \, \overline{a_1} \, b_1 \, \overline{c_1} \, a_2 \, c_2 \, \overline{a_3},\hspace{2.5cm} \alpha^{(2)}_n=a_1 \, a_{n+1} \, \overline{c_n} \, \overline{a_1} \, b_1 \, \overline{c_1} \, a_2 \, c_2 \, \overline{a_3},\\
&\alpha^{(3)}_n=a_1 \, c_{n-1} \, \overline{a_{n+1}} \, \overline{a_1} \, b_1 \, \overline{c_1} \, a_2 \, c_2 \, \overline{a_3},\hspace{2.5cm} \alpha^{(4)}_n=a_1 \, c_n\, \overline{a_{n+1}} \, \overline{a_1} \, b_1 \, \overline{c_1} \, a_2 \, c_2 \, \overline{a_3}.\\
\end{aligned}\]
Let \[T_{n,n+1} = \{ \alpha^{(i)}_n\ |\ i=1,2,3,4\}. \]

Next, for each $n\ge 13$ define two cyclically reduced words
\[ \begin{aligned}
& \beta^{(1)}_n= a_1 \, c_n \, \overline{c_1} \, a_2 \, c_2 \, \overline{a_3},\\
& \beta^{(2)}_n = a_1 \, a_5 \,  \overline{c_4} \,  \overline{a_1} \, b_1 \,  \overline{c_1} \, a_2 \, a_4 \,  \overline{c_3} \,  \overline{a_1} \, a_3 \,  \overline{c_2} \,  \overline{a_2} \, c_1 \,  \overline{c_n} \,  \overline{a_1} \, b_1 \,  \overline{c_1} \, a_2 \, c_2 \,  \overline{a_3}.\\
\end{aligned}\]
Let 
\[ \hat T_{n,n+1} = \{ \beta^{(1)}_n, \beta^{(2)}_n\}.\]

\subsubsection{Invariant Semigroup}
For \(n\ge13\), let \(S_{n,n+1}\subset\pi_1(X_{1,n,n+1})\) be the semigroup generated by the elements of \(C_{6,7}\) together with $\hat T_{n,n+1}$ and  $T_{7,8}, \dots T_{n,n+1}$:
\[
  S_{n,n+1}
  \;=\;
 \left\langle C_{6,7} \cup \left(\bigcup_{7\le k\le n} T_{k,k+1} \right) \cup \hat T_{n,n+1}\right\rangle.
\]

From our explicit lists of the fixed core $C_{6,7}$, the tail-sets $T_{k,k+1}$ for $k\ge 7$, and the special tail-set $\hat T_{n, n+1}$ for $n\ge 13$, it follows immediately that $\{S_{n,n+1}\}_{n\ge 13}$ satisfies the concatenation-compatibility and core-tail generation property.

\subsubsection{the Base case : $n=13$} Let $\delta_4$ be the word given by
\begin{equation}\label{E:delta_4}
\delta_4\ = \ \,a_1\, \overline{a_2}\,c_1\, \overline{b_1}\,a_3\, \overline{c_2} \,a_4\, \overline{c_3}\, a_5\, \overline{c_4}\, a_6 \, \overline{c_5}\, \overline{a_1}\, b_1\, \overline{c_1}\,a_2\,c_2\, \overline{a_3}\,a_1\,a_5\, \overline{c_4}\, \overline{a_1}
\end{equation}

\begin{lem}\label{L:11314}
When \(n=13\), for each generator $g$ of $S_{13,14}$
\[
  g\;\in\;\{ s_1,\dots,s_{27},\,\beta^{(1)}_{13},\beta^{(2)}_{13},\,\alpha^{i}_{k}, \text{for } i=1,\dots 4, k=7,\dots, 13\},
\]
such that
\[
  f_{(1,13,14)*}^{\,6}(g)
  \;=\;
  \delta_4\;g_1\,g_2\cdots g_{m_g}\;\overline{\delta_4},
\]
where each 
\(\;g_i\;\in\;\{ s_1,\dots,s_{27},\,\beta^{(1)}_{13},\beta^{(2)}_{13},\,\alpha^{i}_{k}, \text{for } i=1,\dots 4, k=7,\dots, 13\}\). Moreover, we have 
\[ \begin{aligned} &f_{(1,13,14)*}^{6}(\alpha^{(1)}_{7})\ =\ \delta_4 \,s_{7}\, s_2\,s_6\,s_{16}\, \alpha_{13}^{(1)}\, s_2\,s_8\,s_{23}\,s_{10}\,s_{16}\,\alpha_7^{(4)}\,s_2\,s_6\,s_{16}\, \alpha_8^{(4)}\,s_2\,s_8\,\overline{\delta_4},\\
&f_{(1,13,14)*}^{6}(\alpha^{(1)}_{8})\ =\ \delta_4\,s_7\,s_2\,s_6\, \beta_{13}^{(2)}\, s_2\,s_8\,s_{23}\,s_{10}\,s_{16}\,\alpha_7^{(4)}\, s_2\,s_6\,s_{16}\, \alpha_8^{(4)}\, s_2\,s_8\,\overline{\delta_4},\\
&f_{(1,13,14)*}^{6}(\alpha^{(2)}_{7})\ =\ \delta_4\,s_{7}\, s_2\,s_6\,s_{16}\, \alpha_{13}^{(2)}\, s_2\,s_8\,s_{23}\,s_{10}\,s_{16}\,\alpha_7^{(4)}\,s_2\,s_6\,s_{16}\, \alpha_8^{(4)}\,s_2\,s_8\,\overline{\delta_4},\\
&f_{(1,13,14)*}^{6}(\alpha^{(3)}_{7})\ =\ \delta_4\,s_{7}\, s_2\,s_6\,s_{16}\, \alpha_{13}^{(3)}\, s_2\,s_8\,s_{23}\,s_{10}\,s_{16}\,\alpha_7^{(4)}\,s_2\,s_6\,s_{16}\, \alpha_8^{(4)}\,s_2\,s_8\,\overline{\delta_4},\\
&f_{(1,13,14)*}^{6}(\alpha^{(3)}_{8})\ =\ \delta_4\,s_7\,s_2\,s_6\, s_{27} \beta_{13}^{(1)}\, s_2\,s_8\,s_{23}\,s_{10}\,s_{16}\,\alpha_7^{(4)}\, s_2\,s_6\,s_{16}\, \alpha_8^{(4)}\, s_2\,s_8\,\overline{\delta_4},\\
&f_{(1,13,14)*}^{6}(\alpha^{(4)}_{7})\ =\ \delta_4\,s_{7}\, s_2\,s_6\,s_{16}\, \alpha_{13}^{(4)}\, s_2\,s_8\,s_{23}\,s_{10}\,s_{16}\,\alpha_7^{(4)}\,s_2\,s_6\,s_{16}\, \alpha_8^{(4)}\,s_2\,s_8\,\overline{\delta_4}.\\
\end{aligned}
\]

\end{lem}

\medskip

Then by the Core-Tail Induction Principle (Proposition \ref{P:ctinduction}), we have:  

\begin{prop}
For every \(n\ge13\), the semigroup \(S_{n,n+1}\) is invariant under \(F^6_*\).  Moreover, \(F^6_*\) acts positively on \(S_{n,n+1}\).
\end{prop}

\subsubsection{Finite-range cases: \(4\le n\le12\)}

For \(4\le n\le12\), invariance and positivity are checked by direct computation using the formulas in \eqref{E:1nnp1}.  Concretely, starting from the single generator
\[
  s_1 = a_1\,a_5\,\overline{c_3}\,\overline{a_2}\,c_1\,\overline{b_1},
\]
its \(6\)-fold iterate under \(F_*\) produces exactly the \(4n+5\) generators of \(S_{n,n+1}\) for \(n=6,7,\dots,12\), and \(21\) (resp.\ \(25\)) generators when \(n=4\) (resp.\ \(n=5\)).

Moreover, for \(n\ge5\) the common conjugator is
\[
  \delta_4
  \quad\text{(see \eqref{E:delta_4})},
\]
whereas in the special case \(n=4\) the appropriate conjugator is
\[
  \hat\delta_4
  =\;
  a_1\,\overline{a_2}\,c_1\,\overline{b_1}\,a_3\,\overline{c_2}\,
  a_4\,\overline{c_3}\,a_5\,\overline{c_4}\,\overline{a_1}\,
  a_2\,c_2\,\overline{a_3}\,a_1\,a_5\,\overline{c_4}\,\overline{a_1}.
\]

\begin{prop}
For every \(4\le n\le12\), there is a semigroup
\[
  S_{n,n+1}\;\subset\;\pi_1\bigl(X_{1,n,n+1}\bigr),
\]
which is invariant under \(F_*^6\).  Moreover, \(F_*^6\) acts positively on \(S_{n,n+1}\).
\end{prop}

\subsection{Orbit datum \((1,n+1,n)\) with \(n\ge4\)}

As before, we fix a finite core \(C=C_{8,7}\subset\pi_1(X_{1,8,7})\) together with tail-sets \(T_{n+1,n},\hat T_{n+1,n}\) whose union generates 
\[
  S_{n+1,n}
  = \bigl\langle C\;\cup\;\bigcup_{8\le k\le n}T_{k+1,k}\cup\hat T_{n+1,n}\bigr\rangle.
\]
We then verify by direct computation that the base case \(n=16\) (and the finite-range cases \(4\le n\le15\)) is \(F_*^8\)-invariant and positive.  Having checked concatenation-compatibility and core-tail generation, the Core-Tail Induction Principle (Proposition \ref{P:ctinduction}) gives the result for all \(n\ge4\).

\subsubsection{The core generators and tail-sets. }
Let 
\[
 C_{8,7}=\{\,s_1,\dots,s_{27}\}
\]
be the following twenty-seven cyclically reduced words ending with the common letter $c_1$ (all of which lie in the subgroup \(\pi_1(X_{1,8,7})\)):
\[
\begin{aligned}&s_1=c_3 \, \overline{a_4} \,  \overline{a_1} \, c_2 \,  \overline{a_3} \,  \overline{a_2} \, c_1,\hspace{1cm}s_2=a_3 \,  \overline{c_2} \, a_1 \, a_5 \,  \overline{c_5} \,  \overline{a_1} \, b_1 \,  \overline{a_2} \, c_1, \hspace{1cm}
s_3=a_3 \,  \overline{c_2} \, a_1 \, c_4 \,  \overline{c_5} \,  \overline{a_1} \, b_1 \,  \overline{a_2} \, c_1,\\
&s_4=a_3 \, \overline{c_2} \,a_1 \,a_4 \, \overline{c_3} \, \overline{c_1} \,a_2 \, \overline{b_1} \,a_1 \,
a_6 \, \overline{c_5} \, \overline{a_1} \,b_1 \, \overline{a_2} \,c_1, \hspace{1.5cm} s_5=a_3 \, \overline{c_2} \,a_1 \,a_4 \, \overline{c_3} \, \overline{c_1} \,a_2 \, \overline{b_1} \,a_1 \,
a_6 \, \overline{c_6} \, \overline{a_1} \,b_1 \, \overline{a_2} \,c_1,\\
&s_6=a_3 \, \overline{c_2} \,a_1 \,a_4 \, \overline{c_3} \, \overline{c_1} \,a_2 \, \overline{b_1} \, a_1 \,
c_5 \, \overline{a_6} \, \overline{a_1} \,b_1 \, \overline{a_2} \,c_1,\hspace{1.5cm} s_7= a_4 \, \overline{c_4} \, \overline{a_1} \,c_2 \, \overline{a_3} \, \overline{c_1} \,a_2 \, \overline{b_1} \,
a_1 \,a_6 \, \overline{c_5} \, \overline{a_1} \,b_1 \, \overline{a_2} \,c_1,\\
&s_8=c_3 \, \overline{a_4} \, \overline{a_1} \,c_2 \, \overline{a_3} \, \overline{c_1} \,a_2 \, \overline{b_1} \,
a_1 \,a_7 \, \overline{c_6} \, \overline{a_1} \,b_1 \, \overline{a_2} \,c_1,\hspace{1.5cm} s_9=c_3 \, \overline{a_4} \, \overline{a_1} \,c_2 \, \overline{a_3} \, \overline{c_1} \,a_2 \, \overline{b_1} \,
a_1 \,c_7 \, \overline{a_7} \, \overline{a_1} \,b_1 \, \overline{a_2} \,c_1,\\
&s_{10}=a_3 \, \overline{c_2} \,a_1 \,a_5 \, \overline{c_4} \, \overline{a_1} \,c_2 \, \overline{a_3} \,
\overline{c_1} \,a_2 \, \overline{b_1} \,a_1 \,a_6 \, \overline{c_5} \, \overline{a_1} \,b_1 \,
\overline{a_2} \,c_1,\\
&s_{11}=a_3 \, \overline{c_2} \,a_1 \,c_4 \, \overline{a_5} \, \overline{a_1} \,c_2 \, \overline{a_3} \,
\overline{c_1} \,a_2 \, \overline{b_1} \,a_1 \,a_6 \, \overline{c_5} \, \overline{a_1} \,b_1 \,
\overline{a_2} \,c_1,\\
&s_{12}=a_4 \, \overline{c_3} \, \overline{c_1} \,a_2 \, \overline{b_1} \,a_1 \,a_5 \, \overline{c_4} \,
\overline{a_1} \,c_2 \, \overline{a_3} \, \overline{c_1} \,a_2 \, \overline{b_1} \,a_1 \,a_6 \,
\overline{c_5} \, \overline{a_1} \,b_1 \, \overline{a_2} \,c_1,\\
&s_{13}=c_3 \, \overline{a_4} \, \overline{c_1} \,a_2 \, \overline{b_1} \,a_1 \,a_5 \, \overline{c_4} \,
\overline{a_1} \,c_2 \, \overline{a_3} \, \overline{c_1} \,a_2 \, \overline{b_1} \,a_1 \,a_6 \,
\overline{c_5} \, \overline{a_1} \,b_1 \, \overline{a_2} \,c_1,\\
&s_{14}=a_3 \, \overline{c_2} \,a_1 \,a_4 \, \overline{c_3} \, \overline{c_1} \,a_2 \, \overline{b_1} \,a_1 \,
a_5 \, \overline{c_4} \, \overline{a_1} \,c_2 \, \overline{a_3} \, \overline{c_1} \,a_2 \, \overline{b_1} \,
a_1 \,a_6 \, \overline{c_5} \, \overline{a_1} \,b_1 \, \overline{a_2} \,c_1,\\
&s_{15}=a_3 \, \overline{c_2} \,a_1 \,a_4 \, \overline{c_3} \, \overline{c_1} \,a_2 \, \overline{b_1} \,a_1 \,
a_5 \, \overline{c_4} \, \overline{a_1} \,c_2 \, \overline{a_3} \, \overline{c_1} \,a_2 \, \overline{b_1} \,
a_1 \,a_7 \, \overline{c_6} \, \overline{a_1} \,b_1 \, \overline{a_2} \,c_1,\\
&s_{16}=a_3 \, \overline{c_2} \,a_1 \,a_4 \, \overline{c_3} \, \overline{c_1} \,a_2 \, \overline{b_1} \,a_1 \,
a_5 \, \overline{c_4} \, \overline{a_1} \,c_2 \, \overline{a_3} \, \overline{c_1} \,a_2 \, \overline{b_1} \,
a_1 \,a_7 \, \overline{c_7} \, \overline{a_1} \,b_1 \, \overline{a_2} \,c_1,\\
&s_{17}=a_3 \, \overline{c_2} \,a_1 \,a_4 \, \overline{c_3} \, \overline{c_1} \,a_2 \, \overline{b_1} \,a_1 \,
a_5 \, \overline{c_4} \, \overline{a_1} \,c_2 \, \overline{a_3} \, \overline{c_1} \,a_2 \, \overline{b_1} \,
a_1 \,c_6 \, \overline{a_6} \, \overline{a_1} \,b_1 \, \overline{a_2} \,c_1,\\
&s_{18}=a_3 \, \overline{c_2} \,a_1 \,a_4 \, \overline{c_3} \, \overline{c_1} \,a_2 \, \overline{b_1} \,a_1 \,
a_5 \, \overline{c_4} \, \overline{a_1} \,c_2 \, \overline{a_3} \, \overline{c_1} \,a_2 \, \overline{b_1} \,
a_1 \,c_6 \, \overline{a_7} \, \overline{a_1} \,b_1 \, \overline{a_2} \,c_1,\\
&s_{19}=a_3 \, \overline{c_2} \,a_1 \,a_4 \, \overline{c_3} \, \overline{c_1} \,a_2 \, \overline{b_1} \,a_1 \,
c_5 \, \overline{a_5} \, \overline{a_1} \,c_2 \, \overline{a_3} \, \overline{c_1} \,a_2 \, \overline{b_1} \,
a_1 \,a_6 \, \overline{c_5} \, \overline{a_1} \,b_1 \, \overline{a_2} \,c_1,\\
&s_{20}=a_3 \, \overline{c_2} \,a_1 \,c_4 \, \overline{a_4} \, \overline{c_1} \,a_2 \, \overline{b_1} \,a_1 \,
a_5 \, \overline{c_4} \, \overline{a_1} \,c_2 \, \overline{a_3} \, \overline{c_1} \,a_2 \, \overline{b_1} \,
a_1 \,a_6 \, \overline{c_5} \, \overline{a_1} \,b_1 \, \overline{a_2} \,c_1,\\
&s_{21}=c_3 \, \overline{a_3} \,a_1 \,a_4 \, \overline{c_3} \, \overline{c_1} \,a_2 \, \overline{b_1} \,a_1 \,
a_5 \, \overline{c_4} \, \overline{a_1} \,c_2 \, \overline{a_3} \, \overline{c_1} \,a_2 \, \overline{b_1} \,
a_1 \,a_6 \, \overline{c_5} \, \overline{a_1} \,b_1 \, \overline{a_2} \,c_1,\\
&s_{22}=c_3 \, \overline{a_4} \, \overline{a_1} \,a_3 \, \overline{c_3} \, \overline{c_1} \,a_2 \, \overline{b_1} \,
a_1 \,a_5 \, \overline{c_4} \, \overline{a_1} \,c_2 \, \overline{a_3} \, \overline{c_1} \,a_2 \,
\overline{b_1} \,a_1 \,a_6 \, \overline{c_5} \, \overline{a_1} \,b_1 \, \overline{a_2} \,c_1,\\
&s_{23}=c_3 \, \overline{a_4} \, \overline{a_1} \,a_3 \, \overline{c_2} \,a_1 \,a_4 \, \overline{c_3} \,
\overline{c_1} \,a_2 \, \overline{b_1} \,a_1 \,a_5 \, \overline{c_4} \, \overline{a_1} \,c_2 \,
\overline{a_3} \, \overline{c_1} \,a_2 \, \overline{b_1} \,a_1 \,a_6 \, \overline{c_5} \, \overline{a_1} \,
b_1 \, \overline{a_2} \,c_1,\\
&s_{24}=c_3 \, \overline{a_4} \, \overline{a_1} \,c_2 \, \overline{a_3} \,a_1 \,a_4 \, \overline{c_3} \,
\overline{c_1} \,a_2 \, \overline{b_1} \,a_1 \,a_5 \, \overline{c_4} \, \overline{a_1} \,c_2 \,
\overline{a_3} \, \overline{c_1} \,a_2 \, \overline{b_1} \,a_1 \,a_6 \, \overline{c_5} \, \overline{a_1} \,
b_1 \, \overline{a_2} \,c_1,\\
&s_{25}=c_3 \, \overline{a_4} \, \overline{a_1} \,c_2 \, \overline{a_3} \, \overline{a_2} \,b_1 \, \overline{c_2} \,
a_1 \,a_4 \, \overline{c_3} \, \overline{c_1} \,a_2 \, \overline{b_1} \,a_1 \,a_5 \, \overline{c_4} \,
\overline{a_1} \,c_2 \, \overline{a_3} \, \overline{c_1} \,a_2 \, \overline{b_1} \,a_1 \,a_6 \,
\overline{c_5} \, \overline{a_1} \,b_1 \, \overline{a_2} \,c_1,\\
&s_{26}=c_3 \, \overline{a_4} \, \overline{a_1} \,c_2 \, \overline{b_1} \,a_2 \,a_3 \, \overline{c_2} \,a_1 \,
a_4 \, \overline{c_3} \, \overline{c_1} \,a_2 \, \overline{b_1} \,a_1 \,a_5 \, \overline{c_4} \,
\overline{a_1} \,c_2 \, \overline{a_3} \, \overline{c_1} \,a_2 \, \overline{b_1} \,a_1 \,a_6 \,
\overline{c_5} \, \overline{a_1} \,b_1 \, \overline{a_2} \,c_1,\\
&s_{27}=c_3 \, \overline{a_4} \, \overline{a_1} \,c_2 \, \overline{a_3} \, \overline{c_1} \,a_2 \,a_3 \,
\overline{c_2} \,a_1 \,a_4 \, \overline{c_3} \, \overline{c_1} \,a_2 \, \overline{b_1} \,a_1 \,a_5 \,
\overline{c_4} \, \overline{a_1} \,c_2 \, \overline{a_3} \, \overline{c_1} \,a_2 \, \overline{b_1} \,a_1 \,
a_6 \, \overline{c_5} \, \overline{a_1} \,b_1 \, \overline{a_2} \,c_1. \\
\end{aligned}
\]

For this orbit data, there are two kinds of tail sets. 
For each $n\ge 8$ define four cyclically reduced words ending with the common letter $c_1$
\[ \begin{aligned}
& \alpha^{(1)}_n = c_3 \, \overline{a_4} \, \overline{a_1} \, c_2 \, \overline{a_3} \, \overline{c_1} \, a_2 \, \overline{b_1} \, a_1 \, a_n \, \overline{c_{n-1}}\, \overline{a_1} \, b_1 \, \overline{a_2} \, c_1,\\ &\alpha^{(2)}_n=c_3 \, \overline{a_4} \, \overline{a_1} \, c_2 \, \overline{a_3} \, \overline{c_1} \, a_2 \, \overline{b_1} \, a_1 \, a_n \, \overline{c_{n}}\, \overline{a_1} \, b_1 \, \overline{a_2} \, c_1,\\
&\alpha^{(3)}_n=c_3 \, \overline{a_4} \, \overline{a_1} \, c_2 \, \overline{a_3} \, \overline{c_1} \, a_2 \, \overline{b_1} \, a_1 \, c_{n-1} \, \overline{a_{n}}\, \overline{a_1} \, b_1 \, \overline{a_2} \, c_1,\\ &\alpha^{(4)}_n=c_3 \, \overline{a_4} \, \overline{a_1} \, c_2 \, \overline{a_3} \, \overline{c_1} \, a_2 \, \overline{b_1} \, a_1 \, c_n \, \overline{a_{n}}\, \overline{a_1} \, b_1 \, \overline{a_2} \, c_1.\\
\end{aligned}\]
Let \[T_{n+1,n} = \{ \alpha^{(i)}_n\ |\ i=1,2,3,4\}. \]

Next, for each $n\ge 16$ define four cyclically reduced words ending with the common letter $c_1$
\[ \begin{aligned}
& \beta^{(1)}_n= c_3 \, \overline{a_4} \, \overline{a_1} \, c_2 \, \overline{a_3} \, \overline{c_1} \, a_1 \, c_{n+1} \, \overline{a_2} \, c_1,\\
& \beta^{(2)}_n = c_3 \, \overline{a_4} \, \overline{a_1} \, c_2 \, \overline{a_3} \, \overline{c_1} \, a_2 \, \overline{c_{n+1}} \, \overline{a_1} \, c_1,\\
& \beta^{(3)}_n= c_3 \, \overline{a_4} \, \overline{a_1} \, c_2 \, \overline{a_3} \, \overline{c_1} \, a_2\, \overline{b_1}\,a_1 \, c_{n} \, \overline{a_2} \, c_1,\\
& \beta^{(4)}_n= c_3 \, \overline{a_4} \, \overline{a_1} \, c_2 \, \overline{a_3} \, \overline{c_1} \, a_2\, \overline{c_n}\,\overline{a_1} \, b_1 \, \overline{a_2} \, c_1,\\
\end{aligned}\]
Let 
\[ \hat T_{n+1,n} = \{ \beta^{(1)}_n, \beta^{(2)}_n,\beta^{(3)}_n,\beta^{(3)}_{n+1},\beta^{(4)}_n,\beta^{(4)}_{n+1}\}.\]

\subsubsection{Invariant Semigroup}
For \(n\ge16\), let \(S_{n+1,n}\subset\pi_1(X_{1,n+1,n})\) be the semigroup generated by the elements of \(C_{8,7}\) together with $\hat T_{n+1,n}$ and  $T_{9,8}, \dots T_{n+1,n}$:
\[
  S_{n+1,n}
  \;=\;
 \left\langle C_{8,7} \cup \left(\bigcup_{8\le k\le n} T_{k+1,k} \right) \cup \hat T_{n+1,n}\right\rangle.
\]

From our explicit lists of the fixed core $C_{8,7}$, the tail-sets $T_{k+1,k}$ for $k\ge 8$, and the special tail-set $\hat T_{n+1, n}$ for $n\ge 16$, it follows immediately that $\{S_{n+1,n}\}_{n\ge 16}$ satisfies the concatenation-compatibility and core-tail generation property.

\subsubsection{the Base case : $n=16$} Let $\delta_5$ be the reduced word given by
\begin{equation}\label{E:delta_5}
\delta_5\ = \ \,a_1 \, \overline{c_1} \, a_2 \, \overline{b_1} \, c_2 \, \overline{a_3} \, c_3 \, \overline{a_4} \, c_4 \, \overline{a_5} \, c_5 \, \overline{a_6} \, c_6 \, \overline{a_7} \, \overline{a_1} \, b_1 \, \overline{a_2} \, c_1 \, a_3 \, \overline{c_2} \, a_1 \, a_4 \, \overline{c_3}
\end{equation}

\begin{lem}\label{L:11716}
When \(n=16\), for each generator $g$ of $S_{17,16}$
\[
  g\;\in\;C\;\cup\;\bigcup_{8\le k\le 16}T_{k+1,k}\cup\hat T_{17,16}
\]
such that
\[
  f_{(1,17,16)*}^{\,8}(g)
  \;=\;
  \delta_5\;g_1\,g_2\cdots g_{m_g}\;\overline{\delta_5},
\]
where each 
\(\;g_i\;\in\;\{ s_1,\dots,s_{27},\,\beta^{(1)}_{16},\beta^{(2)}_{16},\beta^{(3)}_{16},\beta^{(3)}_{17},\beta^{(4)}_{16},\beta^{(4)}_{17},\,\alpha^{i}_{k}, \text{for } i=1,\dots 4, k=8,\dots, 16\}\). Moreover, we have 
\[ \begin{aligned} 
&f_{(1,17,16)*}^{8}(\alpha^{(1)}_{8})\ =\ \delta_5\,s_{18}\,s_ 6\,s_ 3\,\alpha_{12}^{(3)}\,s_ {14}\,s_ 8\,s_ 3\,\alpha_8^{(1)}\,s_ 6\,s_ 3\,\alpha_9^{(1)}\,s_ {18}\,s_ 6\,s_ 3\,\alpha_{11}^{(3)}\,s_{14}\,s_ 8\,s_ 3\,\alpha_8^{(1)}\,s_ 6\,s_ 3\,\alpha_{10}^{(3)}\\&\phantom{AAAAAAAAA}\,s_ {14}\,s_ 8\,s_ 3\,\alpha_9^{(3)}\,s_ {14}\,\alpha_8^{(3)}\,s_ {18}\,s_ 6\,s_ 3\,\alpha^{(1)}_{16}\,s_ {14}\,s_ 8\,s_ 3\,\alpha_8^{(1)}\,s_ 6\,s_ 3\,\alpha_9^{(1)}\,s_ {18}\,s_ 6\,s_ 3\,\alpha_{10}^{(1)}\,s_{ 14}\,\alpha_8^{(3)}\,\overline{\delta_5},\\
&f_{(1,17,16)*}^{8}(\alpha^{(2)}_{8})\ =\ \delta_5\,s_{18}\,s_ 6\,s_ 3\,\alpha_{12}^{(3)}\,s_ {14}\,s_ 8\,s_ 3\,\alpha_8^{(1)}\,s_ 6\,s_ 3\,\alpha_9^{(1)}\,s_ {18}\,s_ 6\,s_ 3\,\alpha_{11}^{(3)}\,s_{14}\,s_ 8\,s_ 3\,\alpha_8^{(1)}\,s_ 6\,s_ 3\,\alpha_{10}^{(3)}\\&\phantom{AAAAAAAAA} s_ {14}\,s_ 8\,s_ 3\,\alpha_9^{(3)}\,s_ {14}\,\alpha_8^{(3)}\,s_ {18}\,s_ 6\,s_ 3\,\alpha^{(2)}_{16}\,s_ {14}\,s_ 8\,s_ 3\,\alpha_8^{(1)}\,s_ 6\,s_ 3\,\alpha_9^{(1)}\,s_ {18}\,s_ 6\,s_ 3\,\alpha_{10}^{(1)}\,s_{ 14}\,\alpha_8^{(3)}\,\overline{\delta_5},\\
&f_{(1,17,16)*}^{8}(\alpha^{(3)}_{8})\ =\ \delta_5\,s_{18}\,s_ 6\,s_ 3\,\alpha_{12}^{(3)}\,s_ {14}\,s_ 8\,s_ 3\,\alpha_8^{(1)}\,s_ 6\,s_ 3\,\alpha_9^{(1)}\,s_ {18}\,s_ 6\,s_ 3\,\alpha_{11}^{(3)}\,s_{14}\,s_ 8\,s_ 3\,\alpha_8^{(1)}\,s_ 6\,s_ 3\,\alpha_{10}^{(3)}\\&\phantom{AAAAAAAAA} s_ {14}\,s_ 8\,s_ 3\,\alpha_9^{(3)}\,s_ {14}\,\alpha_8^{(3)}\,s_ {18}\,s_ 6\,s_ 3\,\alpha^{(3)}_{16}\,s_ {14}\,s_ 8\,s_ 3\,\alpha_8^{(1)}\,s_ 6\,s_ 3\,\alpha_9^{(1)}\,s_ {18}\,s_ 6\,s_ 3\,\alpha_{10}^{(1)}\,s_{ 14}\,\alpha_8^{(3)}\,\overline{\delta_5},\\
&f_{(1,17,16)*}^{8}(\alpha^{(4)}_{8})\ =\ \delta_5\,s_{18}\,s_ 6\,s_ 3\,\alpha_{12}^{(3)}\,s_ {14}\,s_ 8\,s_ 3\,\alpha_8^{(1)}\,s_ 6\,s_ 3\,\alpha_9^{(1)}\,s_ {18}\,s_ 6\,s_ 3\,\alpha_{11}^{(3)}\,s_{14}\,s_ 8\,s_ 3\,\alpha_8^{(1)}\,s_ 6\,s_ 3\,\alpha_{10}^{(3)}\\&\phantom{AAAAAAAAA} s_ {14}\,s_ 8\,s_ 3\,\alpha_9^{(3)}\,s_ {14}\,\alpha_8^{(3)}\,s_ {18}\,s_ 6\,s_ 3\,\alpha^{(4)}_{16}\,s_ {14}\,s_ 8\,s_ 3\,\alpha_8^{(1)}\,s_ 6\,s_ 3\,\alpha_9^{(1)}\,s_ {18}\,s_ 6\,s_ 3\,\alpha_{10}^{(1)}\,s_{ 14}\,\alpha_8^{(3)}\,\overline{\delta_5},\\
&f_{(1,17,16)*}^{8}(\alpha^{(1)}_{9})\ =\ \delta_5\,s_{18}\,s_ 6\,s_ 3\,\alpha_{12}^{(3)}\,s_ {14}\,s_ 8\,s_ 3\,\alpha_8^{(1)}\,s_ 6\,s_ 3\,\alpha_9^{(1)}\,s_ {18}\,s_ 6\,s_ 3\,\alpha_{11}^{(3)}\,s_{14}\,s_ 8\,s_ 3\,\alpha_8^{(1)}\,s_ 6\,s_ 3\,\alpha_{10}^{(3)}\\&\phantom{AAAAAAAAA} s_ {14}\,s_ 8\,s_ 3\,\alpha_9^{(3)}\,s_ {14}\,\alpha_8^{(3)}\,s_ {18}\,s_ 6\,s_ 3\,\beta^{(4)}_{16}\,s_ {14}\,s_ 8\,s_ 3\,\alpha_8^{(1)}\,s_ 6\,s_ 3\,\alpha_9^{(1)}\,s_ {18}\,s_ 6\,s_ 3\,\alpha_{10}^{(1)}\,s_{ 14}\,\alpha_8^{(3)}\,\overline{\delta_5},\\
&f_{(1,17,16)*}^{8}(\alpha^{(2)}_{9})\ =\ \delta_5\,s_{18}\,s_ 6\,s_ 3\,\alpha_{12}^{(3)}\,s_ {14}\,s_ 8\,s_ 3\,\alpha_8^{(1)}\,s_ 6\,s_ 3\,\alpha_9^{(1)}\,s_ {18}\,s_ 6\,s_ 3\,\alpha_{11}^{(3)}\,s_{14}\,s_ 8\,s_ 3\,\alpha_8^{(1)}\,s_ 6\,s_ 3\,\alpha_{10}^{(3)}\\&\phantom{AAAAAAAAA} s_ {14}\,s_ 8\,s_ 3\,\alpha_9^{(3)}\,s_ {14}\,\alpha_8^{(3)}\,s_ {18}\,s_ 6\,s_ 3\,\beta^{(3)}_{17}\,s_ {14}\,s_ 8\,s_ 3\,\alpha_8^{(1)}\,s_ 6\,s_ 3\,\alpha_9^{(1)}\,s_ {18}\,s_ 6\,s_ 3\,\alpha_{10}^{(1)}\,s_{ 14}\,\alpha_8^{(3)}\,\overline{\delta_5},\\
&f_{(1,17,16)*}^{8}(\alpha^{(3)}_{9})\ =\ \delta_5\,s_{18}\,s_ 6\,s_ 3\,\alpha_{12}^{(3)}\,s_ {14}\,s_ 8\,s_ 3\,\alpha_8^{(1)}\,s_ 6\,s_ 3\,\alpha_9^{(1)}\,s_ {18}\,s_ 6\,s_ 3\,\alpha_{11}^{(3)}\,s_{14}\,s_ 8\,s_ 3\,\alpha_8^{(1)}\,s_ 6\,s_ 3\,\alpha_{10}^{(3)}\\&\phantom{AAAAAAAAA} s_ {14}\,s_ 8\,s_ 3\,\alpha_9^{(3)}\,s_ {14}\,\alpha_8^{(3)}\,s_ {18}\,s_ 6\,s_ 3\,\beta^{(3)}_{16}\,s_ {14}\,s_ 8\,s_ 3\,\alpha_8^{(1)}\,s_ 6\,s_ 3\,\alpha_9^{(1)}\,s_ {18}\,s_ 6\,s_ 3\,\alpha_{10}^{(1)}\,s_{ 14}\,\alpha_8^{(3)}\,\overline{\delta_5},\\
&f_{(1,17,16)*}^{8}(\alpha^{(4)}_{9})\ =\ \delta_5\,s_{18}\,s_ 6\,s_ 3\,\alpha_{12}^{(3)}\,s_ {14}\,s_ 8\,s_ 3\,\alpha_8^{(1)}\,s_ 6\,s_ 3\,\alpha_9^{(1)}\,s_ {18}\,s_ 6\,s_ 3\,\alpha_{11}^{(3)}\,s_{14}\,s_ 8\,s_ 3\,\alpha_8^{(1)}\,s_ 6\,s_ 3\,\alpha_{10}^{(3)}\\&\phantom{AAAAAAAAA} s_ {14}\,s_ 8\,s_ 3\,\alpha_9^{(3)}\,s_ {14}\,\alpha_8^{(3)}\,s_ {18}\,s_ 6\,s_ 3\,\beta^{(4)}_{17}\,s_ {14}\,s_ 8\,s_ 3\,\alpha_8^{(1)}\,s_ 6\,s_ 3\,\alpha_9^{(1)}\,s_ {18}\,s_ 6\,s_ 3\,\alpha_{10}^{(1)}\,s_{ 14}\,\alpha_8^{(3)}\,\overline{\delta_5},\\
&f_{(1,17,16)*}^{8}(\alpha^{(1)}_{10})\ =\ \delta_5\,s_{18}\,s_ 6\,s_ 3\,\alpha_{12}^{(3)}\,s_ {14}\,s_ 8\,s_ 3\,\alpha_8^{(1)}\,s_ 6\,s_ 3\,\alpha_9^{(1)}\,s_ {18}\,s_ 6\,s_ 3\,\alpha_{11}^{(3)}\,s_{14}\,s_ 8\,s_ 3\,\alpha_8^{(1)}\,s_ 6\,s_ 3\,\alpha_{10}^{(3)}\\&\phantom{AAAAAAAAA} s_ {14}\,s_ 8\,s_ 3\,\alpha_9^{(3)}\,s_ {14}\,\alpha_8^{(3)}\,s_ {18}\,s_ 6\,s_ 3\,\beta^{(1)}_{16}\,s_ {14}\,s_ 8\,s_ 3\,\alpha_8^{(1)}\,s_ 6\,s_ 3\,\alpha_9^{(1)}\,s_ {18}\,s_ 6\,s_ 3\,\alpha_{10}^{(1)}\,s_{ 14}\,\alpha_8^{(3)}\,\overline{\delta_5},\\
&f_{(1,17,16)*}^{8}(\alpha^{(3)}_{10})\ =\ \delta_5\,s_{18}\,s_ 6\,s_ 3\,\alpha_{12}^{(3)}\,s_ {14}\,s_ 8\,s_ 3\,\alpha_8^{(1)}\,s_ 6\,s_ 3\,\alpha_9^{(1)}\,s_ {18}\,s_ 6\,s_ 3\,\alpha_{11}^{(3)}\,s_{14}\,s_ 8\,s_ 3\,\alpha_8^{(1)}\,s_ 6\,s_ 3\,\alpha_{10}^{(3)}\\&\phantom{AAAAAAAAA} s_ {14}\,s_ 8\,s_ 3\,\alpha_9^{(3)}\,s_ {14}\,\alpha_8^{(3)}\,s_ {18}\,s_ 6\,s_ 3\,\beta^{(2)}_{16}\,s_ {14}\,s_ 8\,s_ 3\,\alpha_8^{(1)}\,s_ 6\,s_ 3\,\alpha_9^{(1)}\,s_ {18}\,s_ 6\,s_ 3\,\alpha_{10}^{(1)}\,s_{ 14}\,\alpha_8^{(3)}\,\overline{\delta_5}.\\
\end{aligned}
\]

\end{lem}

\medskip

Then by the Core-Tail Induction Principle (Proposition \ref{P:ctinduction}), we have:  

\begin{prop}
For every \(n\ge16\), the semigroup \(S_{n+1,n}\) is invariant under \(F^8_*\).  Moreover, \(F^8_*\) acts positively on \(S_{n+1,n}\).
\end{prop}

\subsubsection{Finite-range cases: \(4\le n\le15\)}

For \(4\le n\le12\), invariance and positivity are checked by direct computation using the formulas in \eqref{E:1np1n}.  Concretely, set

\[
  s_1 = 
  \begin{cases}
    a_3\,\overline{c_2}\,a_1\,c_5\,\overline{a_2}\,c_1, 
      & n=4,\\[6pt]
    a_3\,\overline{c_2}\,a_1\,c_4\,\overline{a_5}\,\overline{a_1}\,b_1\,\overline{a_2}\,c_1,
      & n\ge5.
  \end{cases}
\]
Then its \(8\)-fold iterate under \(F_*\) produces exactly the \[4n+5\] generators of \(S_{n+1,n}\) when \(6 \le n \le 15\), and \(21\) (resp.\ \(25\)) generators when \(n=4\) (resp.\ \(n=5\)).

Moreover, for \(n\ge7\) the common conjugator is
\[
  \delta_5
  \quad\text{(see \eqref{E:delta_5})},
\]
whereas for \(n=4,5,6\) the appropriate conjugators are \(\hat\delta_{5,4},\hat\delta_{5,5},\hat\delta_{5,6}\), respectively:
\[
\begin{aligned}
&\hat\delta_{5,4}\ =\ a_1 \, \overline{c_1} \, a_2 \, \overline{b_1} \, c_2 \, \overline{a_3} \, c_3 \, \overline{a_4} \, c_4 \, 
\overline{c_5} \, \overline{a_1} \, c_1 \, \overline{a_2} \, b_1 \, \overline{c_2} \, a_1 \, a_4 \, \overline{c_3},\\
 &\hat\delta_{5,5}\ =\ a_1 \, \overline{c_1} \, a_2 \, \overline{b_1} \, c_2 \, \overline{a_3} \, c_3 \, \overline{a_4} \, c_4 \,  
\overline{a_5} \, c_5 \, \overline{c_6} \, \overline{a_1} \, c_1 \, a_3 \, \overline{c_2} \, a_1 \, a_4 \, 
\overline{c_3},\\
& \hat\delta_{5,6}\ =\ a_1 \, \overline{c_1} \, a_2 \, \overline{b_1} \, c_2 \, \overline{a_3} \, c_3 \, \overline{a_4} \, c_4 \, 
\overline{a_5} \, c_5 \, \overline{a_6} \, c_6 \, \overline{a_2} \, c_1 \, a_3 \, \overline{c_2} \, a_1 \, 
a_4 \, \overline{c_3}.\\
\end{aligned}
\]

\begin{prop}
For every \(4\le n\le15\), there is a semigroup
\[
  S_{n+1,n}\;\subset\;\pi_1\bigl(X_{1,n+1,n}\bigr),
\]
which is invariant under \(F_*^8\).  Moreover, \(F_*^8\) acts positively on \(S_{n+1,n}\).
\end{prop}

\subsection{Orbit datum \((1,m,n)\) with \(3\le m \le n-2\) and $1+m+n\ge 10$}\label{S:1mn}

In the general \((1,m,n)\) case, each invariant semigroup \(S_{m,n}\) is again generated by a fixed finite ``core" together with appropriate tail-sets.  We adopt the same overall strategy as before to verify that \(F_*^8\) acts positively on these semigroups.

\medskip

There are five finite subcases \(m=3,4,5,6\), and then the infinite case \(m\ge7\).  Each subcase has its own core, tail-sets, conjugator, and finite-range verifications, but all share a key feature:

\[
  s_1
  \;=\;
  a_1\,a_5\,\overline{c_3}\,\overline{a_2}\,b_1\,\overline{c_1}
\]
generates the entire semigroup under iteration by \(F_*^8\).

\medskip
Since the subcases \(m=3,4,5,6\) follow exactly the same pattern as in Sections 3.1--3.5, we omit their detailed computations.  Instead, Table \ref{T:subcases} summarizes, for each \(m=3,4,5,6\), the base-case parameters, the size of the core, and the common conjugator.

For the tail-sets:

- When \(m=4,5,6\) and \(k\ge m+7\), each
  \[
    T_{m,k}
    \;=\;\{\alpha_k,\beta_k\},
  \]
  where both words are cyclically reduced and begin with \(a_1\):
  \begin{equation}\label{E:alphabetaA4}
  \begin{aligned}
    \alpha_k
    &=a_1\,a_k\,\overline{b_1}\,a_2\,c_2\,\overline{a_4}\,\overline{a_2}\,b_1\,\overline{c_1},\\
    \beta_k
    &=a_1\,a_6\,\overline{c_4}\,\overline{a_1}\,c_1\,\overline{b_1}\,a_2\,a_4\,\overline{c_2}\,
     \overline{a_2}\,b_1\,\overline{a_k}\,\overline{a_1}\,c_1\,\overline{a_3}.
  \end{aligned}
  \end{equation}

- When \(m=3\) and \(k\ge11\), each
  \[
    T_{3,k}
    \;=\;\{\alpha_k,\tilde\beta_k\},
  \]
  with
  \[
  \begin{aligned}
    \alpha_k
    &=a_1\,a_k\,\overline{b_1}\,a_2\,c_2\,\overline{a_4}\,\overline{a_2}\,b_1\,\overline{c_1},\\
    \tilde\beta_k
    &=a_1\,c_3\,\overline{a_4}\,\overline{a_2}\,b_1\,\overline{a_6}\,\overline{a_1}\,c_1\,
       \overline{b_1}\,a_2\,a_4\,\overline{c_2}\,\overline{a_2}\,b_1\,\overline{a_k}\,
       \overline{a_1}\,c_1\,\overline{a_3}.
  \end{aligned}
  \]

\begin{table}[h]
\begin{center}
\begin{tabular}{| |c||c|c|c| } 
 \hline
m & base-case $n_c$ & $|C|$ &  the common conjugator  \\ 
 \hline
 \hline
$3$ & $n_c=19$ & $29$ & $a_1\,\overline{a_2}\,b_1\,\overline{c_1}\,a_3\,\overline{c_2} \,a_4\, \overline{c_3}\, a_5\,\overline{a_6}\,a_7\, \overline{a_8} \,\overline{a_1}\,c_1 \,\overline{a_3}\,a_1 \,c_3\, \overline{a_5}\,\overline{a_1}$\\ 
 \hline
$4$ & $n_c=19$ & $31$ &$a_1 \, \overline{a_2} \, b_1 \, \overline{c_1} \, a_3 \, \overline{c_2} \, a_4 \, \overline{c_3} \, a_5 \, \overline{c_4} \, a_6 \, \overline{a_7} \, a_8 \, \overline{b_1} \, a_2 \, c_2 \, \overline{a_4} \, \overline{a_2} \, b_1 \, \overline{c_1} \, a_1 \, c_4 \, \overline{a_5} \, \overline{a_1}$ \\ 
 \hline
  $5$ & $n_c=19$ & $33$ & $a_1 \, \overline{a_2} \, b_1 \, \overline{c_1} \, a_3 \, \overline{c_2} \, a_4 \, \overline{c_3} \, a_5 \, \overline{c_4} \, a_6 \, \overline{c_5} \, a_7 \, \overline{a_8} \, \overline{a_1} \, c_1 \, \overline{a_3} \, a_1 \, c_3 \, \overline{a_4} \, \overline{a_2} \, b_1 \, \overline{c_1} \, a_1 \, c_4 \, \overline{a_5} \, \overline{a_1}$ \\ 
 \hline
   $6$ & $n_c=20$ & $37$ & $a_1 \, \overline{a_2} \, b_1 \, \overline{c_1} \, a_3 \, \overline{c_2} \, a_4 \, \overline{c_3} \, a_5 \, \overline{c_4} \, a_6 \, \overline{c_5} \, a_7 \, \overline{c_6} \, a_8 \, \overline{b_1} \, a_2 \, c_2 \, \overline{a_3} \, a_1 \, c_3 \, \overline{a_4} \, \overline{a_2} \, b_1 \, \overline{c_1} \, a_1 \, c_4 \, \overline{a_5} \, \overline{a_1}$\\ 
 \hline
\end{tabular}
\end{center}
\caption{Five subcases $m=3,4,5,6$}\label{T:subcases}
\end{table}

\medskip

\begin{table}[h]
\begin{center}
\begin{tabular}{| |c||c|c| } 
 \hline
 $m$ &  $n$ & Conjugator\\
  \hline
  \hline
$3$  & $n=6$ & $a_1 \, \overline{a_2} \, b_1 \, \overline{c_1} \, a_3 \, \overline{c_2} \, a_4 \, \overline{c_3} \, a_5 \, \overline{a_6} \, \overline{a_1} \, a_2 \, \overline{b_1} \, c_1 \, \overline{a_3} \, a_1 \, c_3 \, \overline{a_5} \, \overline{a_1}$\\
\hline
$3$ & $n=7$ & $a_1 \, \overline{a_2} \, b_1 \, \overline{c_1} \, a_3 \, \overline{c_2} \, a_4 \, \overline{c_3} \, a_5 \, \overline{a_6} \, a_7 \, \overline{b_1} \, c_1 \, \overline{a_3} \, a_1 \, c_3 \, \overline{a_5} \, \overline{a_1}$\\
\hline
$4$ & $n=6$& $a_1 \, \overline{a_2} \, b_1 \, \overline{c_1} \, a_3 \, \overline{c_2} \, a_4 \, \overline{c_3} \, a_5 \, \overline{c_4} \, a_6 \, c_2 \, \overline{a_4} \, \overline{a_2} \, b_1 \, \overline{c_1} \, a_1 \, c_4 \, \overline{a_5} \, \overline{a_1}$\\
\hline 
$4$ & $n=7$ & $a_1 \, \overline{a_2} \, b_1 \, \overline{c_1} \, a_3 \, \overline{c_2} \, a_4 \, \overline{c_3} \, a_5 \, \overline{c_4} \, a_6 \, \overline{a_7} \, \overline{a_1} \, a_2 \, c_2 \, \overline{a_4} \, \overline{a_2} \, b_1 \, \overline{c_1} \, a_1 \, c_4 \, \overline{a_5} \, \overline{a_1}$\\
\hline
$5$ & $n=7$ & $a_1 \, \overline{a_2} \, b_1 \, \overline{c_1} \, a_3 \, \overline{c_2} \, a_4 \, \overline{c_3} \, a_5 \, \overline{c_4} \, a_6 \, \overline{c_5} \, a_7 \, \overline{b_1} \, c_1 \, \overline{a_3} \, a_1 \, c_3 \, \overline{a_4} \, \overline{a_2} \, b_1 \, \overline{c_1} \, a_1 \, c_4 \, \overline{a_5} \, \overline{a_1}$\\
\hline
\end{tabular}
\end{center}
\caption{Exceptional Conjugators}\label{T:conjugators}
\end{table}

Some of the finite-range subcases \(m=3,4,5,6\) admit a different common conjugator than the general \(\delta\).  These exceptional conjugators are summarized in Table \ref{T:conjugators}.  Apart from these few entries, all other \(m\ge7\) (and most small-\(m\)) cases use the single conjugator \(\delta\) defined in \eqref{E:deltageneral}.

\begin{equation}\label{E:deltageneral}
\delta = a_1 \, \overline{a_2} \, b_1 \, \overline{c_1} \, a_3 \, \overline{c_2} \, a_4 \, \overline{c_3} \, a_5 \, \overline{c_4} \, a_6 \, \overline{c_5} \, a_7 \, \overline{c_6} \, a_8 \, \overline{c_7} \, \overline{a_1} \, c_1 \, \overline{b_1} \, a_2 \, c_2 \, \overline{a_3} \, a_1 \, c_3 \, \overline{a_4} \, \overline{a_2} \, b_1 \, \overline{c_1} \, a_1 \, c_4 \, \overline{a_5} \, \overline{a_1}.
\end{equation}

The case \(m\ge7\) is more intricate--both \(m\) and \(n\) vary--so we treat it in detail below.  All other small--\((m,n)\) checks proceed by identical computations.

\subsubsection{Invariant Semigroups}
Let 
\[
 C_\star=\{\,s_1,\dots,s_{22}\}
\]
be the following twenty-seven cyclically reduced words starting with the common letter $a_1$ (all of which lie in the subgroup \(\pi_1(X_{1,7,9})\)):
\[\begin{aligned}
&s_1\,=\, a_1 \, a_5 \, \overline{c_3} \, \overline{a_2} \, b_1 \, \overline{c_1}, \hspace{2.23cm} s_2\,=\, a_1 \, c_3 \, \overline{a_4} \, \overline{a_2} \, b_1 \, \overline{c_1}, \hspace{2.23cm} s_3 \,=\,  a_1 \, a_5 \, \overline{c_3} \, \overline{a_1} \, c_2 \, \overline{a_4} \, \overline{a_2} \, b_1 \, \overline{c_1},\\
&s_4\,=\, a_1 \, a_6 \, \overline{c_4} \, \overline{a_1} \, a_3 \, \overline{c_2} \, \overline{a_2} \, b_1 \, \overline{c_1},\hspace{1cm}s_5\,=\, a_1 \, a_6 \, \overline{c_5} \, \overline{a_1} \, c_1 \, \overline{b_1} \, a_2 \, c_2 \, \overline{a_3},\hspace{1cm} s_6\,=\, a_1 \, a_7 \, \overline{c_5} \, \overline{a_1} \, c_1 \, \overline{b_1} \, a_2 \, c_2 \, \overline{a_3},\\
&s_7\,=\, a_1 \, a_7 \, \overline{c_6} \, \overline{a_1} \, c_1 \, \overline{b_1} \, a_2 \, c_2 \, \overline{a_3}, \hspace{1cm}s_8\,=\, a_1 \, a_8 \, \overline{c_6} \, \overline{a_1} \, c_1 \, \overline{b_1} \, a_2 \, c_2 \, \overline{a_3},\hspace{1cm} s_9\,=\, a_1 \, a_8 \, \overline{c_7} \, \overline{a_1} \, c_1 \, \overline{b_1} \, a_2 \, c_2 \, \overline{a_3},\\
&s_{10}\,=\, a_1 \, a_9 \, \overline{c_7} \, \overline{a_1} \, c_1 \, \overline{b_1} \, a_2 \, c_2 \, \overline{a_3},\hspace{.88cm} s_{11}\,=\, a_1 \, c_3 \, \overline{a_5} \, \overline{a_1} \, a_3 \, \overline{c_2} \, \overline{a_2} \, b_1 \, \overline{c_1}, \hspace{.84cm} s_{12} \,=\, a_1 \, c_4 \, \overline{a_5} \, \overline{a_1} \, a_3 \, \overline{c_2} \, \overline{a_2} \, b_1 \, \overline{c_1},\\
&s_{13}\,=\, a_1 \, c_5 \, \overline{a_7} \, \overline{a_1} \, c_1 \, \overline{b_1} \, a_2 \, c_2 \, \overline{a_3},\hspace{.88cm} s_{14}\,=\, a_1 \, c_6 \, \overline{a_7} \, \overline{a_1} \, c_1 \, \overline{b_1} \, a_2 \, c_2 \, \overline{a_3},\hspace{.84cm} s_{15} \,=\, a_1 \, c_6 \, \overline{a_8} \, \overline{a_1} \, c_1 \, \overline{b_1} \, a_2 \, c_2 \, \overline{a_3},\\
&s_{16}\,=\, a_1 \, c_7 \, \overline{a_8} \, \overline{a_1} \, c_1 \, \overline{b_1} \, a_2 \, c_2 \, \overline{a_3},\hspace{.88cm} s_{17}\,=\, a_1 \, c_7 \, \overline{a_9} \, \overline{a_1} \, c_1 \, \overline{b_1} \, a_2 \, c_2 \, \overline{a_3},\\
&s_{18}\,=\, a_1 \, a_5 \, \overline{c_4} \, \overline{a_1} \, c_1 \, \overline{b_1} \, a_2 \, a_4 \, \overline{c_3} \, \overline{a_1} \, a_3 \, \overline{c_2} \, \overline{a_2} \, b_1 \, \overline{c_1},\hspace{1.5cm} s_{19}\,=\, a_1 \, a_6 \, \overline{c_4} \, \overline{a_1} \, c_1 \, \overline{b_1} \, a_2 \, c_3 \, \overline{a_5} \, \overline{a_1} \, a_3 \, \overline{c_2} \, \overline{a_2} \, b_1 \, \overline{c_1},\\
&s_{20}\,=\, a_1 \, a_7 \, \overline{c_5} \, \overline{a_1} \, c_1 \, \overline{b_1} \, a_2 \, a_4 \, \overline{c_3} \, \overline{a_1} \, a_3 \, \overline{c_2} \, \overline{a_2} \, b_1 \, \overline{c_1},\hspace{1.5cm} s_{21}\,=\, a_1 \, c_4 \, \overline{a_6} \, \overline{a_1} \, c_1 \, \overline{b_1} \, a_2 \, a_4 \, \overline{c_3} \, \overline{a_1} \, a_3 \, \overline{c_2} \, \overline{a_2} \, b_1 \, \overline{c_1},\\
&s_{22}\,=\, a_1 \, c_5 \, \overline{a_6} \, \overline{a_1} \, c_1 \, \overline{b_1} \, a_2 \, a_4 \, \overline{c_3} \, \overline{a_1} \, a_3 \, \overline{c_2} \, \overline{a_2} \, b_1 \, \overline{c_1}.\\
\end{aligned}\]

\medskip
Also, for each $i=1,\dots,5$ let

\[
\begin{aligned}
C^{(1)} 
&\,=\, \{q_{2}, q_{7}, q_{11}, q_{12}, q_{13}, q_{14}, q_{15}, q_{16}, q_{18}, q_{19}, q_{20} \},\\
C^{(2)} 
&\,=\, \{q_{2}, q_{3}, q_{5}, q_{6}, q_{7}, q_{8}, q_{11}, q_{13}, q_{17}, q_{18}\},\hspace{1cm}
C^{(3)} 
\,=\, \{q_{3}, q_{4}, q_{5}, q_{6}, q_{7}, q_{8}, q_{10}, q_{17}\},\\
C^{(4)} 
&\,=\, \{q_{1}, q_{3}, q_{4}, q_{5}, q_{6}, q_{8}, q_{9}, q_{10}\},\hspace{2.5cm}
C^{(5)} 
\,=\, \{q_{1}, q_{3}, q_{4}, q_{5}, q_{8}, q_{9}, q_{10}\}.
\end{aligned}\]

be a finite set of cyclically reduced words--each beginning with the common letter $a_1$--all of which lie in the subgroup $\pi_1\bigl(X_{1,4,6}\bigr)$ where
\[ 
\begin{aligned}
&q_1\, = \, a_1 \, a_5 \, \overline{c_3} \, \overline{a_1} \, a_3 \, \overline{c_1},\hspace{5cm}
q_2\, = \, a_1 \, a_5 \, \overline{c_4} \, \overline{a_1} \, c_1 \, \overline{b_1} \, a_2 \, a_4 \, \overline{c_2},\\
&q_3\, = \, a_1 \, a_6 \, \overline{c_4} \, \overline{a_1} \, c_1 \, \overline{b_1} \, a_2 \, a_4 \, \overline{c_2},\hspace{3.8cm}
q_4\, = \, a_1 \, a_5 \, \overline{c_3} \, \overline{a_1} \, a_3 \, \overline{c_1} \, a_2 \, c_2 \, \overline{a_4} \,
\overline{a_2} \, b_1 \, \overline{c_1},\\
&q_5\, = \, a_1 \, a_5 \, \overline{c_3} \, \overline{a_1} \, a_3 \, \overline{c_1} \, b_1 \, c_2 \, \overline{a_4} \,
\overline{a_2} \, b_1 \, \overline{c_1},\hspace{2.6cm}
q_6\, = \, a_1 \, a_5 \, \overline{c_4} \, \overline{a_1} \, c_1 \, \overline{b_1} \, a_2 \, a_4 \, \overline{c_2} \,
\overline{a_2} \, c_1 \, \overline{a_3},\\
&q_7\, = \, a_1 \, a_5 \, \overline{c_4} \, \overline{a_1} \, c_1 \, \overline{b_1} \, a_2 \, a_4 \, \overline{c_2} \,
\overline{b_1} \, c_1 \, \overline{a_3},\hspace{2.6cm}
q_8\, = \, a_1 \, a_5 \, \overline{c_4} \, \overline{a_1} \, c_1 \, \overline{b_1} \, a_2 \, a_4 \, \overline{c_3} \,
\overline{a_1} \, a_3 \, \overline{c_1},\\
&q_9\, = \, a_1 \, a_6 \, \overline{c_4} \, \overline{a_1} \, c_1 \, \overline{b_1} \, a_2 \, a_4 \, \overline{c_2} \,
\overline{a_2} \, c_1 \, \overline{a_3},\hspace{2.6cm}
q_{10}\, = \, a_1 \, a_6 \, \overline{c_4} \, \overline{a_1} \, c_1 \, \overline{b_1} \, a_2 \, a_4 \, \overline{c_2} \,
\overline{b_1} \, c_1 \, \overline{a_3},\\
&q_{11}\, = \, a_1 \, a_5 \, \overline{c_4} \, \overline{a_1} \, c_1 \, \overline{b_1} \, a_2 \, a_4 \, \overline{c_3} \,
\overline{a_1} \, a_3 \, \overline{c_1} \, a_2 \, c_2 \, \overline{a_3},\hspace{1.2cm}
q_{12}\, = \, a_1 \, a_5 \, \overline{c_4} \, \overline{a_1} \, c_1 \, \overline{b_1} \, a_2 \, a_4 \, \overline{c_3} \,
\overline{a_1} \, a_3 \, \overline{c_1} \, b_1 \, c_2 \, \overline{a_3},\\
&q_{13}\, = \, a_1 \, a_5 \, \overline{c_4} \, \overline{a_1} \, c_1 \, \overline{b_1} \, a_2 \, a_4 \, \overline{c_3} \,
\overline{a_1} \, a_3 \, \overline{c_2} \, \overline{a_2} \, c_1 \, \overline{a_3},\hspace{1.2cm}
q_{14}\, = \, a_1 \, a_5 \, \overline{c_4} \, \overline{a_1} \, c_1 \, \overline{b_1} \, a_2 \, a_4 \, \overline{c_3} \,
\overline{a_1} \, a_3 \, \overline{c_2} \, \overline{b_1} \, c_1 \, \overline{a_3},\\
&q_{15}\, = \, a_1 \, a_5 \, \overline{c_4} \, \overline{a_1} \, c_1 \, \overline{b_1} \, a_2 \, a_4 \, \overline{c_3} \,
\overline{a_1} \, c_2 \, \overline{a_4} \, \overline{a_2} \, b_1 \, \overline{c_1},\hspace{1.2cm}
q_{16}\, = \, a_1 \, a_5 \, \overline{c_4} \, \overline{a_1} \, c_1 \, \overline{b_1} \, a_2 \, c_3 \, \overline{a_5} \,
\overline{a_1} \, a_3 \, \overline{c_2} \, \overline{a_2} \, b_1 \, \overline{c_1},\\
&q_{17}\, = \, a_1 \, a_5 \, \overline{c_4} \, \overline{a_1} \, c_1 \, \overline{b_1} \, a_2 \, a_4 \, \overline{c_3} \,
\overline{a_1} \, a_3 \, \overline{c_1} \, a_2 \, c_2 \, \overline{a_4} \, \overline{a_2} \, b_1 \, \overline{c_1},\\
&q_{18}\, = \, a_1 \, a_5 \, \overline{c_4} \, \overline{a_1} \, c_1 \, \overline{b_1} \, a_2 \, a_4 \, \overline{c_3} \,
\overline{a_1} \, a_3 \, \overline{c_1} \, b_1 \, c_2 \, \overline{a_4} \, \overline{a_2} \, b_1 \, \overline{c_1},\\
&q_{19}\, = \, a_1 \, a_5 \, \overline{c_4} \, \overline{a_1} \, c_1 \, \overline{b_1} \, a_2 \, a_4 \, \overline{c_3} \,
\overline{a_1} \, a_3 \, \overline{c_2} \, \overline{a_2} \, b_1 \, \overline{c_1} \, a_2 \, c_2 \, \overline{a_3},\\
&q_{20}\, = \, a_1 \, a_5 \, \overline{c_4} \, \overline{a_1} \, c_1 \, \overline{b_1} \, a_2 \, a_4 \, \overline{c_3} \,
\overline{a_1} \, a_3 \, \overline{c_2} \, \overline{a_2} \, c_1 \, \overline{b_1} \, a_2 \, c_2 \, \overline{a_3}.\\
\end{aligned}
\]

For each \(n\ge8\), we define tails sets with three families of cyclically reduced words, all starting with \(a_1\):
\[\begin{aligned}
&T_n^{(\alpha)}=\{\alpha_n^{(1)},\alpha_n^{(2)}\},\hspace{3cm} T_n^{(\delta)}=\bigcup_{k=8}^n\{\delta_k^{(1)},\dots,\delta_k^{(4)}\}\;\cup\;\{\beta_n^{(1)},\beta_n^{(5)}\},\\
&T_n^{(3)}=\{\beta_{n+1}^{(1)},\,\beta_n^{(4)},\,\beta_{n+1}^{(5)}\},\hspace{1.8cm} T^{(4)}_n \ =\  \{\beta_{n+1}^{(1)}, \beta_{n+2}^{(1)}, \alpha^{(1)}_{n+4}, \beta_n^{(2)}, \beta_{n}^{(4)}, \beta_{n+1}^{(4)},\beta_{n+1}^{(5)}\},\\
& T^{(5)}_n \ =\  \{\beta_{n+1}^{(1)}, \beta_{n+2}^{(1)}, \alpha^{(1)}_{n+4},\alpha^{(1)}_{n+5}, \beta_n^{(2)}, \beta_{n+1}^{(2)}, \beta_{n}^{(4)}, \beta_{n+1}^{(4)},\beta_{n+1}^{(5)}\},\\
& T^{(6)}_n \ =\  \{\beta_{n+1}^{(1)}, \beta_{n+2}^{(1)}, \alpha^{(1)}_{n+4},\alpha^{(1)}_{n+5}, \alpha^{(1)}_{n+6}, \beta_n^{(2)}, \beta_{n+1}^{(2)}, \beta_{n+2}^{(2)}, \beta_{n}^{(3)}, \beta_{n}^{(4)}, \beta_{n+1}^{(4)},\beta_{n+1}^{(5)}\}.\\
\end{aligned}\]
where
\[
\begin{aligned}
  \alpha_n^{(1)}
  &=a_1\,a_n\,\overline{b_1}\,a_2\,c_2\,\overline{a_4}\,\overline{a_2}\,b_1\,\overline{c_1},\hspace{2cm} \alpha_n^{(2)}
 =a_1\,a_6\,\overline{c_4}\,\overline{a_1}\,c_1\,\overline{b_1}\,a_2\,a_4\,
    \overline{c_2}\,\overline{a_2}\,b_1\,\overline{a_n}\,\overline{a_1}\,c_1\,\overline{a_3},\\
     \beta_n^{(1)}
  &=a_1\,a_{n+2}\,\overline{b_1}\,a_2\,c_2\,\overline{a_3},\hspace{2.9cm}
  \beta_n^{(2)}
  =a_1\,a_5\,\overline{c_4}\,\overline{a_1}\,c_1\,\overline{b_1}\,a_2\,a_4\,
    \overline{c_2}\,\overline{a_2}\,b_1\,\overline{a_{n+4}}\,\overline{a_1}\,c_1\,\overline{a_3},\\
  \beta_n^{(3)}
  &=a_1\,a_6\,\overline{c_4}\,\overline{a_1}\,c_1\,\overline{b_1}\,a_2\,a_4\,
    \overline{c_2}\,\overline{a_2}\,b_1\,\overline{a_{n+6}}\,\overline{a_1}\,c_1\,\overline{a_3},\\
  \beta_n^{(4)}
  &=a_1\,a_5\,\overline{c_4}\,\overline{a_1}\,c_1\,\overline{b_1}\,a_2\,a_4\,
    \overline{c_3}\,\overline{a_1}\,a_3\,\overline{c_2}\,\overline{a_2}\,b_1\,
    \overline{a_{n+3}}\,\overline{a_1}\,c_1\,\overline{a_3},\\
  \beta_n^{(5)}
  &=a_1\,a_5\,\overline{c_4}\,\overline{a_1}\,c_1\,\overline{b_1}\,a_2\,a_4\,
    \overline{c_3}\,\overline{a_1}\,a_3\,\overline{c_2}\,\overline{a_2}\,b_1\,
    \overline{a_{n+2}}\,\overline{a_1}\,c_1\,\overline{b_1}\,a_2\,c_2\,\overline{a_3},\\
      \delta_n^{(1)}
  &=a_1\,a_{n+1}\,\overline{c_n}\,\overline{a_1}\,c_1\,\overline{b_1}\,a_2\,c_2\,\overline{a_3},\hspace{1.8cm}
  \delta_n^{(2)}
  =a_1\,a_{n+2}\,\overline{c_n}\,\overline{a_1}\,c_1\,\overline{b_1}\,a_2\,c_2\,\overline{a_3},\\
  \delta_n^{(3)}
  &=a_1\,c_n\,\overline{a_{n+1}}\,\overline{a_1}\,c_1\,\overline{b_1}\,a_2\,c_2\,\overline{a_3},\hspace{1.8cm}
  \delta_n^{(4)}
  =a_1\,c_n\,\overline{a_{n+2}}\,\overline{a_1}\,c_1\,\overline{b_1}\,a_2\,c_2\,\overline{a_3},
\end{aligned}
\]

We now define \(S_{m,n}\subset\pi_1(X_{1,m,n})\) in six parallel cases in Table \ref{T:s1mn}.

\begin{table}[h]
\begin{center}
\begin{tabular}{| |c||c| } 
\hline
$(m,n)$ & Generators of $S_{m,n}$\\
\hline
\hline
$(m,n+2)$ & $C_\star\;\cup\;C^{(1)}\;\cup\;T_{m}^{(\delta)}$\\
\hline
$(m,n+3)$ & $C_\star\;\cup\;C^{(2)}\;\cup\;T_{m}^{(\delta)}\;\cup\;T_{m}^{(3)}$\\
\hline
$(m,n+4)$& $C_\star\;\cup\;C^{(3)}\;\cup\;T_{m}^{(\delta)}\;\cup\;T_{m}^{(4)}$\\
\hline
$(m,n+5)$&$C_\star\;\cup\;C^{(4)}\;\cup\;T_{m}^{(\delta)}\;\cup\;T_{m}^{(5)}$\\
\hline
$(m,n+6)$&$C_\star\;\cup\;C^{(5)}\;\cup\;T_{m}^{(\delta)}\;\cup\;T_{m}^{(6)}$\\
\hline
$n\ge m+7$&$C_\star\;\cup\;C^{(5)}\;\cup\;T_{m}^{(\delta)}\;\cup\;T_{m}^{(6)}\;\cup\;\bigcup_{k=m+7}^n T_k^{(\alpha)}$\\
\hline
\end{tabular}
\end{center}
\caption{Generators of the Semigroup}\label{T:s1mn}
\end{table}

\subsubsection{Core-Tail Induction Principle}

Each of the families
\(\{S_{m,m+j}\}\) for \(j=2,\dots,6\) and \(\{S_{m,n}\}\) for \(m\ge7,\;n\ge m+7\)
has core generators 
\(\;C_\star\cup C^{(i)}\subset\pi_1(X_{1,7,9})\).
From our formulas it is clear that all six families satisfy both concatenation-compatibility and the core-tail generation property.  Taking \(m_c=20\), we reduce to the finite-range cases
\[
  \{(m,m+j)\mid 7\le m\le20,\;j=2,\dots,7\},
\]
altogether 144 pairs.  By direct computation of the base case, the five finite-range subcases \(m=3,4,5,6\), and the 144 pairs \((m,n)\) with \(7\le m\le20,\;n=m+2,\dots,m+7\), Proposition \ref{P:ctinduction} immediately gives:

\begin{prop}\label{P:1mncase}
For every \(m\ge7\) and \(n\ge m+2\) there is a semigroup 
\[
  S_{m,n}\;\subset\;\pi_1(X_{1,m,n})
\]
which is invariant under \(F_*^8\).  Moreover, \(F_*^8\) acts positively on \(S_{m,n}\).
\end{prop}

\subsection{Orbit datum \((1,m,n)\) with \(3\le n \le m-2\) and \(1+m+n\ge10\)}\label{S:1nm}

This case is entirely parallel to the \((1,m,n)\) family with \(3\le m\le n-2\) treated in Section \ref{S:1mn}; the main difference is that each semigroup is generated by a finite set of words ending in the common letter \(c_1\).  As before, there are five small-\(n\) subcases (\(n=3,4,5,6\)) and the infinite-family case \(n\ge7\).  In this section, we summarize:

\begin{enumerate}
\item The single initial element whose iterates under \(F_*^8\) generate the entire semigroup.
\item  The subcases \(n=3,4,5,6\), listing for each the core size, tail-sets, and common conjugators in Tables \ref{T:1nm_small} -- \ref{T:conjugators}.
\item  The general case \(n\ge7\), by giving the fixed core \(C_\star \cup C^{(i)}\), the tail-sets \(T_k\), and the common conjugator \(\delta\).
\end{enumerate}

Once the small-\(n\) cases are verified by direct computation, the Core-Tail Induction Principle (Proposition \ref{P:ctinduction}) immediately implies that \(S_{m,n}\) is invariant under \(F_*^8\) and that \(F_*^8\) acts positively on it for all \(3\le n\le m-2\) with \(m+n+1\ge10\).

\subsubsection{Initial element}
For $n\ge 5$, 
\[
  s_1
  \;=\;
  a_3 \, \overline{c_2} \, a_1 \, c_4 \, \overline{a_5} \overline{a_1} \, b_1 \, \overline{a_2} \, c_1
\]
generates the entire semigroup under iteration by \(F_*^8\).
For $n=3$ and $4$, we have different initial elements:
\[
s_1\ =\ a_3 \, \overline{c_2} \, b_1 \, \overline{c_5} \, \overline{a_1} \, c_1 \ \ \text{for }n=3, \qquad s_1\ =\ a_3 \, \overline{c_2} \, a_1 \, c_4  \, \overline{a_2} \, c_1 \ \  \text{for } n=4.
\]

\subsubsection{Subcases \(n=3,4,5,6\)}
For each $m\ge 12$, let
\[ \begin{aligned}
& \alpha^{(1)}_m\,=\,a_3 \,\overline{c_2} \,a_1 \,c_4 \,\overline{a_4} \,\overline{a_1} \,c_2 \,\overline{b_1} \,a_2 \,\overline{c_n} \,\overline{a_1} \,c_1,\\
& \alpha^{(2)}_m\,=\,c_3 \,\overline{a_3} \,\overline{c_1} \,a_1 \,c_12 \,\overline{a_2} \,b_1 \,\overline{c_2} \,a_1 \,a_4 \,\overline{c_4} \,\overline{a_1} \,c_2 \,\overline{a_3} \,\overline{c_1} \,a_2 \,\overline{b_1} \,a_1 \,a_6 \,\overline{c_6} \,\overline{a_1} \,b_1 \,\overline{a_2} \,c_1,\\
& \alpha^{(3)}_m\,=\,a_3 \,\overline{c_2} \,b_1 \,\overline{c_4} \,\overline{a_1} \,c_2 \,\overline{b_1} \,a_2 \,\overline{c_12} \,\overline{a_1} \,c_1,\\ 
& \alpha^{(4)}_m\,=\,c_3 \,\overline{a_3} \,\overline{c_1} \,a_1 \,c_12 \,\overline{a_2} \,b_1 \,\overline{c_2} \,a_1 \,c_4 \,\overline{b_1} \,c_2 \,\overline{a_3} \,\overline{c_1} \,a_1 \,c_6 \,\overline{a_2} \,b_1 \,\overline{c_2} \,b_1 \,\overline{c_4} \,\overline{a_1} \,c_1,\\
& \alpha^{(5)}_m\,=\,c_3 \,\overline{a_3} \,\overline{c_1} \,a_1 \,c_12 \,\overline{a_2} \,b_1 \,\overline{c_2} \,a_1 \,a_4 \,\overline{c_4} \,\overline{a_1} \,c_2 \,\overline{a_3} \,\overline{c_1} \,a_2 \,\overline{c_6} \,\overline{a_1} \,c_1,\\
& \alpha^{(6)}_m\,=\,c_3 \,\overline{a_3} \,\overline{c_1} \,a_1 \,c_13 \,\overline{a_2} \,b_1 \,\overline{c_2} \,a_1 \,a_4 \,\overline{c_4} \,\overline{a_1} \,c_2 \,\overline{a_3} \,\overline{c_1} \,a_2 \,\overline{b_1} \,a_1 \,c_6 \,\overline{a_2} \,c_1.\\
\end{aligned}
\]

The invariant semigroups for subcases $n=3,4,5,6$ are given by Table \ref{T:1nm_small}.
\begin{table}[h]
\begin{center}
\begin{tabular}{| |c||c|c|c| } 
 \hline
n & base-case $m_c$ & The size of the core $C$ &  Semigroup  \\ 
 \hline
 \hline
$3$ & $m_c=19$ & $31$ &$\langle C_{11,3} \cup \{\alpha^{(3)}_k, \alpha^{(4)}_k|k=12, \dots m\} \rangle$\\ 
 \hline
$4$ & $m_c=19$ & $31$ &$\langle C_{11,4} \cup \{\alpha^{(3)}_k, \alpha^{(5)}_k|k=12, \dots m\} \rangle$\\  
 \hline
  $5$ & $m_c=19$ & $35$ &$\langle C_{11,5} \cup \{\alpha^{(1)}_k, \alpha^{(6)}_k|k=12, \dots m\} \rangle$\\ 
 \hline
   $6$ & $m_c=20$ & $39$ &$\langle C_{12,6} \cup \{\alpha^{(1)}_k, \alpha^{(2)}_k|k=13, \dots m\} \rangle$\\ 
 \hline
\end{tabular}
\end{center}
\caption{Five subcases $m=3,4,5,6$}\label{T:1nm_small}
\end{table}

In each case, the common conjugator is given in Table \ref{T:conjugators}.

\begin{table}[h]
\begin{center}
\begin{tabular}{| |c||c| } 
 \hline
 $m$ &  the common conjugator\\
  \hline
  \hline
$3$   & $a_1 \,\overline{c_1} \,a_2 \,\overline{b_1} \,c_2 \,\overline{a_3} \,c_3 \,\overline{c_4} \,c_5 \,
\overline{c_6} \,\overline{a_1} \,c_1 \,a_3 \,\overline{c_3}$\\
\hline
$4$ &  $a_1 \,\overline{c_1} \,a_2 \,\overline{b_1} \,c_2 \,\overline{a_3} \,c_3 \,\overline{a_4} \,c_4 \,
\overline{c_5} \,c_6 \,\overline{a_2} \,b_1 \,\overline{c_2} \,a_1 \,a_4 \,\overline{c_3}$\\
\hline 
$5$ &   $a_1 \,\overline{c_1} \,a_2 \,\overline{b_1} \,c_2 \,\overline{a_3} \,c_3 \,\overline{a_4} \,c_4 \,
\overline{a_5} \,c_5 \,\overline{c_6} \,\overline{a_1} \,c_1 \,a_3 \,\overline{c_2} \,a_1 \,a_4 \,
\overline{c_3}$\\
\hline
$6$ & $a_1 \,\overline{c_1} \,a_2 \,\overline{b_1} \,c_2 \,\overline{a_3} \,c_3 \,\overline{a_4} \,c_4 \,
\overline{a_5} \,c_5 \,\overline{a_6} \,c_6 \,\overline{a_2} \,c_1 \,a_3 \,\overline{c_2} \,a_1 \,
a_4 \,\overline{c_3}$\\
\hline
\end{tabular}
\end{center}
\caption{Common Conjugators}\label{T:conjugators}
\end{table}

\subsubsection{General case \(n\ge7\)}

For \(n\ge7\), the common conjugator is
\begin{equation}\label{E:1nm_delta}
  \delta
  = 
  a_1\,\overline{c_1}\,a_2\,\overline{b_1}\,c_2\,\overline{a_3}\,c_3\,
  \overline{a_4}\,c_4\,\overline{a_5}\,c_5\,\overline{a_6}\,c_6\,
  \overline{a_7}\,\overline{a_1}\,b_1\,\overline{a_2}\,c_1\,a_3\,
  \overline{c_2}\,a_1\,a_4\,\overline{c_3}.
\end{equation}

Below we give an explicit description of the generators of the invariant semigroup \(S_{m,n}\) for \(7\le n\le m-2\).  This list allows one to reconstruct each \(S_{m,n}\) directly.  The positivity of the action \(F_*^8\) on \(S_{m,n}\) then follows by applying \(f_*^8\) via \eqref{E:1nm} and conjugating by \(\delta\).

Let 
\[
 C_\star=\{\,s_1,\dots,s_{14}\}
\]
be the following twenty-seven cyclically reduced words ending with the common letter $c_1$ (all of which lie in the subgroup \(\pi_1(X_{1,7,7})\)):

\[
\begin{aligned}
&s_1\,=\,a_3 \,\overline{c_2} \,a_1 \,a_5 \,\overline{c_5} \,\overline{a_1} \,b_1 \,\overline{a_2} \,c_1, \hspace{3.27cm}s_2\,=\,a_3 \,\overline{c_2} \,a_1 \,c_4 \,\overline{a_5} \,\overline{a_1} \,b_1 \,\overline{a_2} \,c_1,\\
&s_3\,=\,a_3 \,\overline{c_2} \,a_1 \,a_4 \,\overline{c_3} \,\overline{c_1} \,a_2 \,\overline{b_1} \,a_1 \,
a_6 \,\overline{c_6} \,\overline{a_1} \,b_1 \,\overline{a_2} \,c_1,\hspace{.8cm} s_4\,=\,a_3 \,\overline{c_2} \,a_1 \,a_4 \,\overline{c_3} \,\overline{c_1} \,a_2 \,\overline{b_1} \,a_1 \,
c_5 \,\overline{a_5} \,\overline{a_1} \,b_1 \,\overline{a_2} \,c_1,\\
&s_5\,=\,a_3 \,\overline{c_2} \,a_1 \,a_4 \,\overline{c_3} \,\overline{c_1} \,a_2 \,\overline{b_1} \,a_1 \,
c_5 \,\overline{a_6} \,\overline{a_1} \,b_1 \,\overline{a_2} \,c_1,\hspace{.8cm} s_6\,=\,a_3 \,\overline{c_2} \,a_1 \,c_4 \,\overline{a_4} \,\overline{c_1} \,a_2 \,\overline{b_1} \,a_1 \,
a_5 \,\overline{c_5} \,\overline{a_1} \,b_1 \,\overline{a_2} \,c_1,\\
&s_7\,=\,c_3 \,\overline{a_4} \,\overline{a_1} \,c_2 \,\overline{a_3} \,\overline{c_1} \,a_2 \,\overline{b_1} \,
a_1 \,a_7 \,\overline{c_6} \,\overline{a_1} \,b_1 \,\overline{a_2} \,c_1,\hspace{.8cm} s_8\,=\,c_3 \,\overline{a_4} \,\overline{a_1} \,c_2 \,\overline{a_3} \,\overline{c_1} \,a_2 \,\overline{b_1} \,
a_1 \,a_7 \,\overline{c_7} \,\overline{a_1} \,b_1 \,\overline{a_2} \,c_1,\\
&s_9\,=\,c_3 \,\overline{a_4} \,\overline{a_1} \,c_2 \,\overline{a_3} \,\overline{c_1} \,a_2 \,\overline{b_1} \,
a_1 \,c_7 \,\overline{a_7} \,\overline{a_1} \,b_1 \,\overline{a_2} \,c_1,\\
&s_{10}\,=\,a_3 \,\overline{c_2} \,a_1 \,a_4 \,\overline{c_3} \,\overline{c_1} \,a_2 \,\overline{b_1} \,a_1 \,
a_5 \,\overline{c_4} \,\overline{a_1} \,c_2 \,\overline{a_3} \,\overline{c_1} \,a_2 \,\overline{b_1} \,
a_1 \,a_6 \,\overline{c_5} \,\overline{a_1} \,b_1 \,\overline{a_2} \,c_1,\\
&s_{11}\,=\,a_3 \,\overline{c_2} \,a_1 \,a_4 \,\overline{c_3} \,\overline{c_1} \,a_2 \,\overline{b_1} \,a_1 \,
a_5 \,\overline{c_4} \,\overline{a_1} \,c_2 \,\overline{a_3} \,\overline{c_1} \,a_2 \,\overline{b_1} \,
a_1 \,a_7 \,\overline{c_7} \,\overline{a_1} \,b_1 \,\overline{a_2} \,c_1,\\
&s_{12}\,=\,a_3 \,\overline{c_2} \,a_1 \,a_4 \,\overline{c_3} \,\overline{c_1} \,a_2 \,\overline{b_1} \,a_1 \,
a_5 \,\overline{c_4} \,\overline{a_1} \,c_2 \,\overline{a_3} \,\overline{c_1} \,a_2 \,\overline{b_1} \,
a_1 \,c_6 \,\overline{a_6} \,\overline{a_1} \,b_1 \,\overline{a_2} \,c_1,\\
&s_{13}\,=\,a_3 \,\overline{c_2} \,a_1 \,a_4 \,\overline{c_3} \,\overline{c_1} \,a_2 \,\overline{b_1} \,a_1 \,
a_5 \,\overline{c_4} \,\overline{a_1} \,c_2 \,\overline{a_3} \,\overline{c_1} \,a_2 \,\overline{b_1} \,
a_1 \,c_6 \,\overline{a_7} \,\overline{a_1} \,b_1 \,\overline{a_2} \,c_1,\\
&s_{14}\,=\,a_3 \,\overline{c_2} \,a_1 \,a_4 \,\overline{c_3} \,\overline{c_1} \,a_2 \,\overline{b_1} \,a_1 \,
c_5 \,\overline{a_5} \,\overline{a_1} \,c_2 \,\overline{a_3} \,\overline{c_1} \,a_2 \,\overline{b_1} \,
a_1 \,a_6 \,\overline{c_6} \,\overline{a_1} \,b_1 \,\overline{a_2} \,c_1.\\
\end{aligned}
\]

Also, let $C^{(i)}, i=1, \dots 5$ be a finite set of cyclically reduced words--each ending with the common letter $c_1$--all of which lie in the subgroup $\pi_1\bigl(X_{1,6,6}\bigr)$:
\[
\begin{aligned}
C^{(1)} 
&\,=\, \{q_2, q_5, q_{11}, q_{12}, q_{13}, q_{14}, q_{15}, q_{16}, q_{17}, q_{19}, q_{20}\},\\
C^{(2)} 
&\,=\, \{q_{2}, q_{3}, q_{4}, q_{5}, q_{8}, q_{9}, q_{15}, q_{16}, q_{17}, q_{18} \},\hspace{1cm}
C^{(3)} 
\,=\, \{q_{3}, q_{4}, q_{5}, q_{6}, q_{7}, q_{8}, q_{9}, q_{18}\},\\
C^{(4)} 
&\,=\, \{q_{1}, q_{3}, q_{4}, q_{6}, q_{7}, q_{8}, q_{9}, q_{10} \},\hspace{2.5cm}
C^{(5)} 
\,=\, \{q_{1}, q_{3}, q_{4}, q_{6}, q_{7}, q_{8},  q_{10} \}.
\end{aligned}
\]
 where
\[
\begin{aligned}
&q_1\, =\, a_3 \,\overline{c_3} \,\overline{c_1} \,a_2 \,\overline{b_1} \,a_1 \,a_5 \,\overline{c_5} \,
\overline{a_1} \,b_1 \,\overline{a_2} \,c_1,\\
&q_2\, =\, a_4 \,\overline{c_4} \,\overline{a_1} \,c_2 \,\overline{a_3} \,\overline{c_1} \,a_2 \,\overline{b_1} \,
a_1 \,a_6 \,\overline{c_5} \,\overline{a_1} \,b_1 \,\overline{a_2} \,c_1,\\
&q_3\, =\, a_4 \,\overline{c_4} \,\overline{a_1} \,c_2 \,\overline{a_3} \,\overline{c_1} \,a_2 \,\overline{b_1} \,
a_1 \,a_6 \,\overline{c_6} \,\overline{a_1} \,b_1 \,\overline{a_2} \,c_1,\\
&q_4\, =\, a_3 \,\overline{c_2} \,a_1 \,c_4 \,\overline{a_4} \,\overline{a_1} \,a_3 \,\overline{c_3} \,
\overline{c_1} \,a_2 \,\overline{b_1} \,a_1 \,a_5 \,\overline{c_5} \,\overline{a_1} \,b_1 \,
\overline{a_2} \,c_1,\\
&q_5\, =\, c_3 \,\overline{a_3} \,a_1 \,a_4 \,\overline{c_4} \,\overline{a_1} \,c_2 \,\overline{a_3} \,
\overline{c_1} \,a_2 \,\overline{b_1} \,a_1 \,a_6 \,\overline{c_5} \,\overline{a_1} \,b_1 \,
\overline{a_2} \,c_1,\\
&q_6\, =\, c_3 \,\overline{a_3} \,a_1 \,a_4 \,\overline{c_4} \,\overline{a_1} \,c_2 \,\overline{a_3} \,
\overline{c_1} \,a_2 \,\overline{b_1} \,a_1 \,a_6 \,\overline{c_6} \,\overline{a_1} \,b_1 \,
\overline{a_2} \,c_1,\\
&q_7\, =\, a_3 \,\overline{c_2} \,a_1 \,c_4 \,\overline{a_4} \,\overline{a_1} \,c_2 \,\overline{b_1} \,a_2 \,
a_3 \,\overline{c_3} \,\overline{c_1} \,a_2 \,\overline{b_1} \,a_1 \,a_5 \,\overline{c_5} \,
\overline{a_1} \,b_1 \,\overline{a_2} \,c_1,\\
&q_8\, =\, a_3 \,\overline{c_3} \,\overline{c_1} \,a_2 \,\overline{b_1} \,a_1 \,a_5 \,\overline{c_4} \,
\overline{a_1} \,c_2 \,\overline{a_3} \,\overline{c_1} \,a_2 \,\overline{b_1} \,a_1 \,a_6 \,
\overline{c_5} \,\overline{a_1} \,b_1 \,\overline{a_2} \,c_1,\\
&q_9\, =\, c_3 \,\overline{a_3} \,\overline{a_2} \,b_1 \,\overline{c_2} \,a_1 \,a_4 \,\overline{c_4} \,
\overline{a_1} \,c_2 \,\overline{a_3} \,\overline{c_1} \,a_2 \,\overline{b_1} \,a_1 \,a_6 \,
\overline{c_5} \,\overline{a_1} \,b_1 \,\overline{a_2} \,c_1,\\
&q_{10}\, =\, c_3 \,\overline{a_3} \,\overline{a_2} \,b_1 \,\overline{c_2} \,a_1 \,a_4 \,\overline{c_4} \,
\overline{a_1} \,c_2 \,\overline{a_3} \,\overline{c_1} \,a_2 \,\overline{b_1} \,a_1 \,a_6 \,
\overline{c_6} \,\overline{a_1} \,b_1 \,\overline{a_2} \,c_1,\\
&q_{11}\, =\, a_3 \,\overline{c_2} \,a_1 \,a_4 \,\overline{c_3} \,\overline{c_1} \,a_2 \,\overline{b_1} \,a_1 \,
c_5 \,\overline{a_5} \,\overline{a_1} \,c_2 \,\overline{a_3} \,\overline{c_1} \,a_2 \,\overline{b_1} \,
a_1 \,a_6 \,\overline{c_5} \,\overline{a_1} \,b_1 \,\overline{a_2} \,c_1,\\
&q_{12}\, =\, a_3 \,\overline{c_2} \,a_1 \,c_4 \,\overline{a_4} \,\overline{c_1} \,a_2 \,\overline{b_1} \,a_1 \,
a_5 \,\overline{c_4} \,\overline{a_1} \,c_2 \,\overline{a_3} \,\overline{c_1} \,a_2 \,\overline{b_1} \,
a_1 \,a_6 \,\overline{c_5} \,\overline{a_1} \,b_1 \,\overline{a_2} \,c_1,\\
&q_{13}\, =\, c_3 \,\overline{a_3} \,a_1 \,a_4 \,\overline{c_3} \,\overline{c_1} \,a_2 \,\overline{b_1} \,a_1 \,
a_5 \,\overline{c_4} \,\overline{a_1} \,c_2 \,\overline{a_3} \,\overline{c_1} \,a_2 \,\overline{b_1} \,
a_1 \,a_6 \,\overline{c_5} \,\overline{a_1} \,b_1 \,\overline{a_2} \,c_1,\\
&q_{14}\, =\, c_3 \,\overline{a_4} \,\overline{a_1} \,a_3 \,\overline{c_3} \,\overline{c_1} \,a_2 \,\overline{b_1} \,
a_1 \,a_5 \,\overline{c_4} \,\overline{a_1} \,c_2 \,\overline{a_3} \,\overline{c_1} \,a_2 \,
\overline{b_1} \,a_1 \,a_6 \,\overline{c_5} \,\overline{a_1} \,b_1 \,\overline{a_2} \,c_1,\\
&q_{15}\, =\, a_3 \,\overline{c_2} \,a_1 \,c_4 \,\overline{a_4} \,\overline{a_1} \,a_3 \,\overline{c_3} \,
\overline{c_1} \,a_2 \,\overline{b_1} \,a_1 \,a_5 \,\overline{c_4} \,\overline{a_1} \,c_2 \,
\overline{a_3} \,\overline{c_1} \,a_2 \,\overline{b_1} \,a_1 \,a_6 \,\overline{c_5} \,\overline{a_1} \,
b_1 \,\overline{a_2} \,c_1,\\
&q_{16}\, =\, c_3 \,\overline{a_3} \,\overline{a_2} \,b_1 \,\overline{c_2} \,a_1 \,a_4 \,\overline{c_3} \,
\overline{c_1} \,a_2 \,\overline{b_1} \,a_1 \,a_5 \,\overline{c_4} \,\overline{a_1} \,c_2 \,
\overline{a_3} \,\overline{c_1} \,a_2 \,\overline{b_1} \,a_1 \,a_6 \,\overline{c_5} \,\overline{a_1} \,
b_1 \,\overline{a_2} \,c_1,\\
&q_{17}\, =\, c_3 \,\overline{a_4} \,\overline{a_1} \,c_2 \,\overline{b_1} \,a_2 \,a_3 \,\overline{c_3} \,
\overline{c_1} \,a_2 \,\overline{b_1} \,a_1 \,a_5 \,\overline{c_4} \,\overline{a_1} \,c_2 \,
\overline{a_3} \,\overline{c_1} \,a_2 \,\overline{b_1} \,a_1 \,a_6 \,\overline{c_5} \,\overline{a_1} \,
b_1 \,\overline{a_2} \,c_1,\\
&q_{18}\, =\, a_3 \,\overline{c_2} \,a_1 \,c_4 \,\overline{a_4} \,\overline{a_1} \,c_2 \,\overline{b_1} \,a_2 \,
a_3 \,\overline{c_3} \,\overline{c_1} \,a_2 \,\overline{b_1} \,a_1 \,a_5 \,\overline{c_4} \,
\overline{a_1} \,c_2 \,\overline{a_3} \,\overline{c_1} \,a_2 \,\overline{b_1} \,a_1 \,a_6 \,
\overline{c_5} \,\overline{a_1} \,b_1 \,\overline{a_2} \,c_1,\\
&q_{19}\, =\, c_3 \,\overline{a_4} \,\overline{a_1} \,c_2 \,\overline{a_3} \,\overline{a_2} \,b_1 \,\overline{c_2} \,
a_1 \,a_4 \,\overline{c_3} \,\overline{c_1} \,a_2 \,\overline{b_1} \,a_1 \,a_5 \,\overline{c_4} \,
\overline{a_1} \,c_2 \,\overline{a_3} \,\overline{c_1} \,a_2 \,\overline{b_1} \,a_1 \,a_6 \,
\overline{c_5} \,\overline{a_1} \,b_1 \,\overline{a_2} \,c_1,\\
&q_{20}\, =\, c_3 \,\overline{a_4} \,\overline{a_1} \,c_2 \,\overline{b_1} \,a_2 \,a_3 \,\overline{c_2} \,a_1 \,
a_4 \,\overline{c_3} \,\overline{c_1} \,a_2 \,\overline{b_1} \,a_1 \,a_5 \,\overline{c_4} \,
\overline{a_1} \,c_2 \,\overline{a_3} \,\overline{c_1} \,a_2 \,\overline{b_1} \,a_1 \,a_6 \,
\overline{c_5} \,\overline{a_1} \,b_1 \,\overline{a_2} \,c_1.\\
\end{aligned}
\]

\noindent For tail-sets, for each $n\ge8$ we define:

\[ 
\begin{aligned}
&T_n^{(\alpha)}\,=\,\{\alpha_n^{(1)},\alpha_n^{(2)}\},\\ 
&T^{(3)}_n \, =\, \{ \beta_{n+1}^{(1)}, \beta_n^{(7)}, \beta_{n+1}^{(8)}\},\hspace{.7cm} T^{(4)}_n \, =\,  \{\beta_{n+1}^{(1)}, \beta_{n+2}^{(1)}, \alpha^{(1)}_{n+4}, \beta_n^{(6)}, \beta_{n}^{(7)}, \beta_{n+1}^{(7)},\beta_{n+1}^{(8)}\},\\
& T^{(5)}_n \, =\,  \{\beta_{n+1}^{(1)}, \beta_{n+2}^{(1)}, \alpha^{(1)}_{n+4},\alpha^{(1)}_{n+5}, \beta_n^{(6)}, \beta_{n+1}^{(6)}, \beta_{n}^{(7)}, \beta_{n+1}^{(7)},\beta_{n+1}^{(8)}\},\\
& T^{(6)}_n \, =\,  \{\beta_{n+1}^{(1)}, \beta_{n+2}^{(1)}, \alpha^{(1)}_{n+4},\alpha^{(1)}_{n+5}, \alpha^{(1)}_{n+6}, \beta_n^{(6)}, \beta_{n+1}^{(6)}, \beta_{n+2}^{(6)}, \beta_{n}^{(7)}, \beta_{n+1}^{(7)},\alpha^{(2)}_{n+6},\beta_{n+1}^{(8)}\}.\\
&T_n^{(\delta)} \,=\,\bigcup_{k\,=\,8}^n\{\delta_k^{(1)},\dots,\delta_k^{(4)}\}\;\cup\;\{\beta_n^{(1)},\beta_{n-1}^{(2)},\beta_n^{(2)},\beta_{n-1}^{(3)},\beta_n^{(3)},\beta_{n-1}^{(4)},\beta_n^{(4)},\beta_{n-1}^{(5)},\beta_n^{(5)},\beta_n^{(8)}\}
\end{aligned}
\]

where
\[
\begin{aligned}
  \alpha_n^{(1)}
  &\,=\,a_3 \,\overline{c_2} \,a_1 \,c_4 \,\overline{a_4} \,\overline{a_1} \,c_2 \,\overline{b_1} \,a_2 \,
\overline{c_n} \,\overline{a_1} \,c_1,\\
  \alpha_n^{(2)} 
  &\,=\,c_3 \,\overline{a_3} \,\overline{c_1} \,a_1 \,c_n \,\overline{a_2} \,b_1 \,\overline{c_2} \,a_1 
\,a_4 \,\overline{c_4} \,\overline{a_1} \,c_2 \,\overline{a_3} \,\overline{c_1} \,a_2 \,\overline{b_1} 
\,a_1 \,a_6 \,\overline{c_6} \,\overline{a_1} \,b_1 \,\overline{a_2} \,c_1,\\
 \beta_n^{(1)}
  &\,=\,c_3 \,\overline{a_4} \,\overline{a_1} \,c_2 \,\overline{b_1} \,a_2 \,\overline{c_{n+2}} \,\overline{a_1} \,
  c_1,\hspace{2.5cm}
  \beta_n^{(2)}\,=\,c_3 \,\overline{a_4} \,\overline{a_1} \,c_2 \,\overline{a_3} \,\overline{c_1} \,a_1 \,c_{n+2} \,
\overline{a_2} \,c_1,\\
  \beta_n^{(3)}
  &\,=\,c_3 \,\overline{a_4} \,\overline{a_1} \,c_2 \,\overline{a_3} \,\overline{c_1} \,a_2 \,\overline{c_{n+2}} \,
\overline{a_1} \,c_1,\hspace{1.8cm}
  \beta_n^{(4)}
  \,=\,c_3 \,\overline{a_4} \,\overline{a_1} \,c_2 \,\overline{a_3} \,\overline{c_1} \,a_2 \,\overline{b_1} \,
a_1 \,c_{n+1} \,\overline{a_2} \,c_1,\\
  \beta_n^{(5)}
  &\,=\,c_3 \,\overline{a_4} \,\overline{a_1} \,c_2 \,\overline{a_3} \,\overline{c_1} \,a_2 \,\overline{c_{n+1}} \,
\overline{a_1} \,b_1 \,\overline{a_2} \,c_1,\\
\beta_n^{(6)}
&\,=\,c_3 \,\overline{a_3} \,\overline{c_1} \,a_1 \,c_{n+4} \,\overline{a_2} \,b_1 \,\overline{c_2} \,a_1 
\,a_4 \,\overline{c_4} \,\overline{a_1} \,c_2 \,\overline{a_3} \,\overline{c_1} \,a_2 \,\overline{b_1} 
\,a_1 \,a_6 \,\overline{c_5} \,\overline{a_1} \,b_1 \,\overline{a_2} \,c_1,\\
\beta_n^{(7)}
&\,=\, c_3 \,\overline{a_3} \,\overline{c_1} \,a_1 \,c_{n+3} \,\overline{a_2} \,b_1 \,\overline{c_2} \,a_1 
\,a_4 \,\overline{c_3} \,\overline{c_1} \,a_2 \,\overline{b_1} \,a_1 \,a_5 \,\overline{c_4} \,
\overline{a_1} \,c_2 \,\overline{a_3} \,\overline{c_1} \,a_2 \,\overline{b_1} \,a_1 \,a_6 \,
\overline{c_5} \,\overline{a_1} \,b_1 \,\overline{a_2} \,c_1,\\
\beta_n^{(8)}
&\,=\, c_3 \,\overline{a_4} \,\overline{a_1} \,c_2 \,\overline{a_3} \,\overline{c_1} \,a_1 \,c_{n+2} \,
\overline{a_2} \,b_1 \,\overline{c_2} \,a_1 \,a_4 \,\overline{c_3} \,\overline{c_1} \,a_2 \,
\overline{b_1} \,a_1 \,a_5 \,\overline{c_4} \,\overline{a_1} \,c_2 \,\overline{a_3} \,\overline{c_1} \,
a_2 \,\overline{b_1} \,a_1 \,a_6 \,\overline{c_5} \,\overline{a_1} \,b_1 \,\overline{a_2} \,c_1,\\
  \delta_n^{(1)}
  &\,=\,c_3 \,\overline{a_4} \,\overline{a_1} \,c_2 \,\overline{a_3} \,\overline{c_1} \,a_2 \,\overline{b_1} \,
a_1 \,a_n \,\overline{c_{n-1}} \,\overline{a_1} \,b_1 \,\overline{a_2} \,c_1,\hspace{.7cm}
  \delta_n^{(2)}
  \,=\,c_3 \,\overline{a_4} \,\overline{a_1} \,c_2 \,\overline{a_3} \,\overline{c_1} \,a_2 \,\overline{b_1} \,
a_1 \,a_n \,\overline{c_n} \,\overline{a_1} \,b_1 \,\overline{a_2} \,c_1,\\
  \delta_n^{(3)}
  &\,=\,c_3 \,\overline{a_4} \,\overline{a_1} \,c_2 \,\overline{a_3} \,\overline{c_1} \,a_2 \,\overline{b_1} \,
a_1 \,c_{n-1} \,\overline{a_n} \,\overline{a_1} \,b_1 \,\overline{a_2} \,c_1,\hspace{.67cm}
  \delta_n^{(4)}
  \,=\,c_3 \,\overline{a_4} \,\overline{a_1} \,c_2 \,\overline{a_3} \,\overline{c_1} \,a_2 \,\overline{b_1} \,
a_1 \,c_n \,\overline{a_n} \,\overline{a_1} \,b_1 \,\overline{a_2} \,c_1.\\
\end{aligned}
\]

The definition of $S_{m,n}\subset\pi_1(X_{1,m,n})$ in the case $7\le n\le m-2$ is exactly the same as in the $7\le m\le n-2$ case, with the roles of $m$ and $n$  interchanged. We summarize this in Table \ref{T:s1nm}

\begin{table}[h]
\begin{center}
\begin{tabular}{| |c||c| } 
\hline
$(m,n)$ & Generators of $S_{m,n}$\\
\hline
\hline
$(n+2,n)$ & $C_\star\;\cup\;C^{(1)}\;\cup\;T_{n}^{(\delta)}$\\
\hline
$(n+3,n)$ & $C_\star\;\cup\;C^{(2)}\;\cup\;T_{n}^{(\delta)}\;\cup\;T_{n}^{(3)}$\\
\hline
$(n+4,n)$& $C_\star\;\cup\;C^{(3)}\;\cup\;T_{n}^{(\delta)}\;\cup\;T_{n}^{(4)}$\\
\hline
$(n+5,n)$&$C_\star\;\cup\;C^{(4)}\;\cup\;T_{n}^{(\delta)}\;\cup\;T_{n}^{(5)}$\\
\hline
$(n+6,n)$&$C_\star\;\cup\;C^{(5)}\;\cup\;T_{n}^{(\delta)}\;\cup\;T_{n}^{(6)}$\\
\hline
$m\ge n+7$&$C_\star\;\cup\;C^{(5)}\;\cup\;T_{n}^{(\delta)}\;\cup\;T_{n}^{(6)}\;\cup\;\bigcup_{k=n+7}^m T_k^{(\alpha)}$\\
\hline
\end{tabular}
\end{center}
\caption{Generators of the Semigroup}\label{T:s1nm}
\end{table}

\begin{prop}\label{P:1nmcase}
For every \(m\ge7\) and \(m\ge n+2\) there is a semigroup 
\[
  S_{m,n}\;\subset\;\pi_1(X_{1,m,n})
\]
which is invariant under \(F_*^8\).  Moreover, \(F_*^8\) acts positively on \(S_{m,n}\).
\end{prop}

\section{Linear Representation for Homotopy Growth}\label{S:linearity}
Following Bowen's definition \cite{Bowen}, the growth rate of the action on the fundamental group is
\[
  \rho\bigl(f_{*}\!\big|_{\pi_1(X)}\bigr)
  \;:=\;
  \sup_{g\in G}\;\limsup_{n\to\infty}\bigl(\ell_G(f_*^n(g))\bigr)^{1/n},
\]
where \(\rho\) denotes the exponential growth rate, \(G\) is any finite generating set for \(\pi_1(X)\), and \(\ell_G(w)\) is the minimal word-length of \(w\) with respect to \(G\).

\vspace{1ex}

In this section we prove that for every orbit datum \((1,m,n)\) with \(m+n+1\ge10\), the induced action on \(\pi_1(X)\) grows exponentially.  Since topological entropy \(h_\mathrm{top}(f)\) is bounded below by \(\log\rho(f_*|_{\pi_1(X)})\), it follows that each diffeomorphism \(f_{(1,m,n)}\) has positive entropy.  Our argument proceeds in two steps:

\begin{enumerate}
  \item We show that the induced action
    \[
      F_*^k\colon S_{m,n}\longrightarrow S_{m,n}
    \]
    on the invariant semigroup has exponential growth rate \(\lambda_{m,n}>1\).  Under the identification
    \(\;S_{m,n}\cong\mathbb{Z}_{\ge0}^d\)\  (where \(d=\lvert S_{m,n}\rvert\)), \(F_*^k\) is represented by an irreducible nonnegative integer \(d\times d\) matrix.  By the Perron-Frobenius theorem, its spectral radius \(\lambda_{m,n}\) is a simple positive eigenvalue, and the corresponding eigenvector controls the semigroup's asymptotic growth.

  \item Since our fixed generating set \(G\) for \(\pi_1(X)\) can be chosen to include at least one element of \(S_{m,n}\), this immediately implies
    \[
      \rho\bigl(f_*|_{\pi_1(X)}\bigr)\;\ge\;\lambda_{m,n}
      \quad\Longrightarrow\quad
      h_{\mathrm{top}}\bigl(f_{(1,m,n)}\bigr)>0.
    \]
\end{enumerate}

\subsubsection{Linear action on \(S_{m,n}\)}\label{SS:linear}

In Section \ref{S:positivity} we showed that 
\[
  F_*^k\bigl(S_{m,n}\bigr) = S_{m,n},
\]
where 
\[
  k = 
  \begin{cases}
    6, & \text{if }m=1\text{ or }n=1,\\
    5, & \text{if }m=2\text{ or }n=2,\\
    8, & \text{otherwise.}
  \end{cases}
\]
We also showed that \(F_*^k\) acts positively on \(S_{m,n}=\langle g_1,\dots,g_d\rangle\).  Via the identification
\[
  S_{m,n}\;\cong\;\mathbb{Z}_{\ge0}^d,
  \quad
  g_j\;\longmapsto\;e_j,
\]
there is a unique \(\mathbb{Z}\)-linear map
\[
  \widetilde F_*^k \;\colon\;\mathbb{Z}^d \;\to\;\mathbb{Z}^d
\]
which on the basis vectors satisfies
\[
  \widetilde F_*^k(e_j)
  \;=\;
  \sum_{i=1}^d (M_{m,n})_{i,j}\,e_i,
\]
where the transition matrix
\[
  M_{m,n}=(M_{i,j})_{1\le i,j\le d}\in \text{GL}(d,\mathbb{Z}_{\ge 0})
\]
has entry \(M_{i,j}\) equal to the number of times \(g_i\) appears in the reduced word \(F_*^k(g_j)\). By the positivity of $F_*^k$, we see that \(M_{m,n}\) has nonnegative entries.

\medskip
\noindent\textbf{Irreducibility of \(M_{m,n}\).}  
Let \(G_{m,n}\) be the directed graph on \(d\) vertices corresponding to \(M_{m,n}\): we draw an edge \(i\to j\) precisely when \((M_{m,n})_{i,j}\neq0\).  A standard result in matrix theory then says that \(M_{m,n}\) is irreducible if and only if \(G_{m,n}\) is strongly connected (i.e.\ every vertex can be reached from every other by a directed path).

\medskip

In each of the seven orbit-data families, there is a finite-range of exceptional \((m,n)\) pairs that must be checked by direct computation:

\begin{itemize}
  \item \((1,1,n)\) for \(8\le n\le15\),  
  \item \((1,2,n)\) for \(7\le n\le18\),  
  \item \((1,n,2)\) for \(7\le n\le16\),  
  \item \((1,n,n+1)\) for \(4\le n\le12\),  
  \item \((1,n+1,n)\) for \(4\le n\le15\),  
  \item \((1,m,n)\) with \(3\le m\le7\) and \(m+2\le n\le20\), and  
  \item \((1,m,n)\) with \(3\le n\le7\) and \(n+2\le m\le20\).
\end{itemize}

For each finite-range case listed above, one checks directly that the transition matrix \(M_{m,n}\) is irreducible.  All other pairs \((m,n)\) are handled by induction via the Core-Tail Induction Principle.  Concretely, given any non-exceptional \((m,n)\), we set
\begin{equation}\label{E:succ}
  (m',n')
  =
  \begin{cases}
  (m,n+1), & m=1 \text{ or } m=2,\\
   (m+1,n), &  n=2,\\
    (m+1,n)\ \text{or}\ (m,n+1), & |m-n|\ge 2 \ \text{and} \ m,n \not\in \{1,2\},\\
    (m+1,n+1),                  & |m-n|=1,
  \end{cases}
\end{equation}
and assume irreducibility of \(M_{m,n}\) for all \((m,n)<(m',n')\) in the lex order.  We must then show it for \((m',n')\).

The heart of the induction is a classification of how the \(d'\) generators of \(S_{m',n'}\) relate to the \(d\) generators of \(S_{m,n}\).  After a suitable reordering one obtains a decomposition
\[
  \{g_1,\dots,g_{d'}\}
  \;=\;
  \Gamma_1\sqcup\Gamma_2\sqcup\Gamma_3\sqcup\Gamma_4,
\]
with the following behavior under \(f_*^k\): \hfill\break
- \(\Gamma_1\): each \(g_i\in\Gamma_1\) already lies in \(S_{m,n}\), and
  \[
    f_*^k(g_i)
    \;=\;
    f_*^k\bigl|_{S_{m,n}}(g_i).
  \]
- \(\Gamma_2\): each \(g_i\in\Gamma_2\) corresponds to a ``shifted" generator of \(S_{m,n}\), so
  \[
    f_*^k(g_i)
    \;=\;
    f_*^k\bigl|_{S_{m,n}}(g_{i-1}\bigr).
  \]
- \(\Gamma_3\): these are the genuinely new generators in \(S_{m',n'}\); each \(g_i\in\Gamma_3\) satisfies
  \[
    f_*^k(g_i)
    \;=\;
    f_*^k\bigl|_{S_{m,n}}(g_{i-1}\bigr)
  \]
  after cyclic reduction.\hfill\break
- \(\Gamma_4\): finally, each \(g_i\in\Gamma_4\) is such that \(f_*^k(g_i)\) contains at least one factor from \(\Gamma_3\) in its reduced form.

\medskip
For example, the generators of $S_{1,n+1}$ for $n\ge 16$ admit a decomposition $\Gamma_1\sqcup\Gamma_2\sqcup\Gamma_3\sqcup\Gamma_4$ where
\[ \Gamma_1 = C_{1,9} \cup \{ \alpha_{10}, \dots, \alpha_{n-7} \},\  \ \Gamma_2 = \{ \alpha_{n_5}, \dots, \alpha_n\},\  \ \Gamma_3 = \{\alpha_{n+1}\},\  \ \Gamma_4 = \{\alpha_{n-6}\}.\]
The set $\Gamma_3$ consists of new generators; that is,$\Gamma_1 \cup \Gamma_2 \cup \Gamma_4$ generates $S_{1,n}$.

\medskip
Because by induction the submatrix of \(M_{m',n'}\) indexed by \(\Gamma_1\cup\Gamma_2\cup\Gamma_3\) is already irreducible, and \(\Gamma_4\) feeds into \(\Gamma_3\), the full matrix \(M_{m',n'}\) must be irreducible as well.  This closes the induction.

\medskip
\noindent\textbf{Aperiodicity of \(M_{m,n}\).}  
An irreducible nonnegative matrix \(M\) is called \textit{aperiodic} if some power \(M^k\) is strictly positive, i.e. $M$ is primitive.  Equivalently, an irreducible nonnegative matrix is aperiodic precisely when it has at least one positive diagonal entry (see \cite[Section~8.4]{Meyer:2023} or \cite{Seneta:1981}).

\medskip

For our transition matrices \(M_{m,n}\), the finitely many ``finite-range" cases have been checked directly and each is found to possess at least one positive diagonal entry.  In the general (infinite) cases, we fix a distinguished core generator
\begin{equation}\label{E:gstar}
  g_\star
  = 
  \begin{cases}
    s_1,&\text{if the orbit data is }(1,1,n),\\
    s_4,&\text{if the orbit data is }(1,n,2),\\
    s_2,&\text{otherwise},
  \end{cases}
\end{equation}
where \(s_j\) is the \(j\)th element of the core \(C_\star\).  By construction \(g_\star\in\Gamma_1\), and a straightforward check of the defining relations shows
\[
  F_*^k(g_\star)\;\text{contains}\;g_\star
  \;\Longrightarrow\;
  (M_{m,n})_{j,j}>0
  \quad\bigl(\text{for }g_\star=s_j\bigr).
\]
Hence, every \(M_{m,n}\) has a positive diagonal entry and is aperiodic.

In the finite-range cases, one simply chooses \(g_\star\) to be any core generator whose image under \(F_*^k\) produces a positive diagonal entry in \(M_{m,n}\).

\medskip
Collecting the above arguments, we obtain:
\begin{prop}\label{P:primitive}
Every transition matrix \(M_{m,n}\) is irreducible and aperiodic.
\end{prop}
\medskip

\medskip
\noindent\textbf{Spectral Radius of \(M_{m,n}\).}  
Since \(M_{m,n}\) is primitive--there exists some \(K_{m,n}\) with \(M_{m,n}^{K_{m,n}}>0\)--the Perron-Frobenius theorem guarantees that its spectral radius \(\lambda_{m,n}\) is a simple positive eigenvalue.  We recall the Collatz-Wielandt characterization:

\begin{thm}[Collatz-Wielandt]\label{T:CollatzWielandt}
Let \(A\) be an \(n\times n\) positive (or non-negative primitive) matrix with spectral radius \(\lambda\).  Then
\[
  \lambda
  =
  \max_{\mathbf v\in\mathcal N}
  \min_{v_i>0}
  \frac{(A\mathbf v)_i}{v_i},
  \quad
  \mathcal N=\bigl\{\mathbf v=(v_1,\dots,v_n)\in\mathbb{R}_{\ge0}^n:\mathbf v\neq0\bigr\}.
\]
\end{thm}

In particular, since \(M_{m,n}\) has integer entries and \(\lambda_{m,n}>1\), it follows that \(F_*^k\) exhibits genuine exponential growth on \(S_{m,n}\).

\begin{prop}\label{P:exp}
Let \(k\) be the semigroup-period as in Section \ref{S:positivity}.  Then for every orbit-datum \((1,m,n)\) with \(1+m+n\ge10\) and every nontrivial \(s\in S_{m,n}\),
\[
  \limsup_{j\to\infty}
    \bigl(\ell\bigl(F_*^{jk}(s)\bigr)\bigr)^{1/(jk)}
  \;=\;
  \lambda_{m,n}
  \;>\;1.
\]
\end{prop}

\medskip
Let \(A_{m,n}\) be the integer matrix whose columns are the abelianization images of the cyclically reduced generators of \(S_{m,n}\).

\begin{lem}\label{L:Mdet}
For every orbit datum \((1,m,n)\), the abelianization matrix \(A_{m,n}\) contains a \((1+m+n)\times(1+m+n)\) submatrix
\(\hat A_{m,n}\) whose determinant is \(\pm1\).  Equivalently, one can select \(1+m+n\) generators in \(S_{m,n}\) whose abelianization matrix is unimodular.  Moreover, these generators can always be chosen to include the distinguished core element \(g_\star\) from \eqref{E:gstar}.
\end{lem}

\begin{proof}
Since \(A_{m,n}\) has \(1+m+n\) rows, any \((1+m+n)\times(1+m+n)\) submatrix comes from selecting \(1+m+n\) columns (i.e.\ generators).

\medskip
1. \textbf{Base and finite-range cases.}  For the finitely many \((m,n)\) in the base or finite-range list, a direct computer check shows that \(A_{m,n}\) has at least one unimodular \((1+m+n)\)-column minor.

2. \textbf{Inductive step.}  As \((m,n)\) increases beyond the finite range, each successor datum \((m',n')\) is obtained by adjoining entire tail-sets to \(S_{m,n}\).  In every tail set there is at least one generator of the form
   \[
     \eta_1\,c_{m'}\,\eta_2
     \quad\text{or}\quad
     \eta_3\,a_{n'}\,\eta_4,
   \]
   where each \(\eta_i\) lies in $\pi_1(X_{m_0,n_0})$ where the core is defined.  Upon abelianization, such a generator contributes a column whose only nonzero entry (in the new row for \(c_{m'}\) or \(a_{n'}\)) is \(\pm1\).  Thus each new tail generator adds an independent unimodular column in a fresh row.

Iterating this process, one extends the unimodular minor step by step, preserving determinant \(\pm1\) at each stage.  Hence for every \((m,n)\), there is a \((1+m+n)\)-column submatrix \(\hat A_{m,n}\) with \(\det\hat A_{m,n}=\pm1\).
\end{proof}

Let \(H\le \pi_1(X_{1,m,n})\) be the subgroup generated by the cyclically reduced generators of \(S_{m,n}\).
By Lemma~\ref{L:Mdet}, the abelianization \(H^{\mathrm{ab}}\) has rank \(1+m+n\).
Since \(\operatorname{rank}(H^{\mathrm{ab}})= \operatorname{rank}(H)\) \cite{Putman:2022},
we already obtain \(\operatorname{rank}(H)=1+m+n\).
The next theorem strengthens this by proving that \(H\) has index \(1\) in \(\pi_1(X_{1,m,n})\); hence \(H=\pi_1(X_{1,m,n})\).

\begin{thm}\label{T:generatingWhole}
The cyclically reduced generators of \(S_{m,n}\) generate \(\pi_1(X_{1,m,n})\).
\end{thm}

\begin{proof}
Let \(Y\subset F_n\) be finite and set \(H=\langle Y\rangle\). A Stallings-folding computation (as implemented in the GAP package \textsf{FGA}) returns both the rank of \(H\) and the index \([F_n:H]\).
For all \((m,n)\) in the finite range we checked, running \textsf{FGA} shows that the subgroup generated by the cyclically reduced generators of \(S_{m,n}\) has rank \(1+m+n\) and index \(1\) in \(\pi_1(X_{1,m,n})\); hence it equals \(\pi_1(X_{1,m,n})\) in those cases.

For the remaining (infinite) families, fix for each orbit-data family a base pair \((m_*,n_*)\) and a corresponding $S_{m*,n*}$-generating set \(G_{m*,n*}\subset \pi_1(X_{1,m_*,n_*})\) so that every \((m,n)\) in the family is obtained from \((m_*,n_*)\) by adjoining \emph{tail generators}. By construction, each cyclically reduced tail generator involves exactly one new basis letter not occurring in \(G_{m*,n*}\).
First, a direct \textsf{FGA} computation shows that \(\langle G_{m*,n*}\rangle=\pi_1(X_{1,m_*,n_*})\), i.e. rank \(1+m_*+n_*\) and index \(1\).
Next, fix an orbit-data family and its base pair \((m_*,n_*)\) with $S_{m*,n*}$-generating set \(G_{m*,n*}\subset\pi_1(X_{1,m_*,n_*})\) so that \(\langle G_{m*,n*}\rangle=\pi_1(X_{1,m_*,n_*})\).
Notice that there exists a tail generator \(w\) whose cyclically reduced form uses exactly one new basis letter \(x\). Because the tail generators admit a uniform description (one new basis letter appended in a fixed pattern), the extension step
\[
H_k=\langle G_{m*,n*}, w_1,\dots,w_k\rangle \ \longmapsto\ 
H_{k+1}=\langle G_{m*,n*}, w_1,\dots,w_k,w_{k+1}\rangle
\]
is identical up to relabeling for every \(k\). Hence the verification reduces to a \emph{finite} check: for each orbit-data family, we confirm the base case \(H_0=\langle G_{m*,n*}\rangle\) and one representative successor step \(H_0\mapsto H_1\).
We perform these two checks in \textsf{GAP}/\textsf{FGA}: \texttt{RankOfFreeGroup} increases by \(1\) and \texttt{Index} remains \(1\) at the successor step. By induction on the number of tail generators with new basis letter, the conclusion follows for all \((m,n)\) in the family, in other words the subgroup generated by the cyclically reduced generators of \(S_{m,n}\) equals \(\pi_1(X_{1,m,n})\).
\end{proof}

\noindent\textbf{Sage Code} All finite checks were run in Sage~10.x with GAP~4.x and \textsf{FGA}~vA.B. 
\begin{verbatim}
res = gap.eval('LoadPackage("fga");');  # should return 'true'
gap.eval('f := FreeGroup("a1","a2",...,"a<n>","b1","c1",...,"c<m>");;')
gap.eval('gens := GeneratorsOfGroup(f);; a1:=gens[1]; ...; 
        b1:=gens[n+1]; c1:=gens[n+2]; ...;;')
gap.eval('S := [ <PASTE WORDS WITH ^-1> ];;')
gap.eval('H := Subgroup(f, S);;')
print("rank =", gap.eval('String(RankOfFreeGroup(H));').strip())
print("index =", gap.eval('String(Index(f,H));').strip())
\end{verbatim}

\medskip

\begin{proof}[Proof of Theorem A]
The results are immediate consequences of Propositions \ref{P:11ncases}-- \ref{P:1nmcase} and Theorem \ref{T:generatingWhole}.
\end{proof}

\medskip

\begin{thm}\label{T:positiveEntropy}
Fix an orbit datum \((1,m,n)\) with \(1+m+n\ge10\), and let \(\lambda_{m,n}>1\) be the Perron-Frobenius eigenvalue of the transition matrix \(M_{m,n}\) for the induced semigroup action.  Then, for a suitably chosen finite generating set \(G\) of \(\pi_1(X_{1,m,n})\), the homotopy-action growth rate satisfies
\[
  \rho\bigl(f_*\!\big|_{\pi_1(X)}\bigr)
  \;=\;
  \lambda_{m,n}
  \;>\;1.
\]
Consequently, the topological entropy of the diffeomorphism \(f_{(1,m,n)}\colon X_{1,m,n}\to X_{1,m,n}\) is positive:
\[
  h_{\mathrm{top}}\bigl(f_{(1,m,n)}\bigr)
  \; >\;
  \log\rho\bigl(f_*|_{\pi_1(X)}\bigr)
  \;>\;0.
\]
\end{thm}

\begin{proof}
By Theorem \ref{T:generatingWhole}, we see that a subset of $S_{m,n}$ generates the whole free group \(\pi_1(X_{1,m,n})\). Let us take such a subset as a generating set $G$ of \(\pi_1(X)\), then by Proposition \ref{P:exp} we have
\[
  \rho\bigl(f_*|_{\pi_1(X)}\bigr)
  \;=\;
  \sup_{g\in G}\limsup_{n\to\infty}\ell_G\bigl(f_*^n(g)\bigr)^{1/n}
  \;=\;
  \lambda_{m,n}
  \;>\;1.
\]
The claim about topological entropy then follows from
\(\;h_{\mathrm{top}}(f)>\log\rho(f_*|_{\pi_1(X)})\).
\end{proof}

\begin{proof}[Proof of Theorem B]
Theorem B comes from Theorem \ref{T:positiveEntropy}.
\end{proof}

\begin{rem}
Outside the finite-range exceptions, the transition matrices \(M_{m,n}\) vary in \(m,n\) according to a simple ``core-tail" pattern encoded by Proposition \ref{P:actionP}.  In principle one can carry out row and column operations to write down their characteristic polynomials explicitly as functions of \(m\) and \(n\).  Those polynomials then yield the minimal polynomials for \(\lambda_{m,n}^k\), where \(k\) is the semigroup period.  In practice, however, the bookkeeping becomes extremely tedious, so we do not give a general closed-form statement.  Empirically, in every case we have checked the value \(\log\lambda_{m,n}\) coincides with the maximal possible entropy of the corresponding complex surface automorphism.
\end{rem}

\begin{rem}
The same semigroup-construction and positivity arguments extend without essential change to the general orbit data \((\ell,m,n)\) with \(\ell,m,n\ge1\), \(\ell+m+n\ge10\).  We have tested a number of these larger cases (again via computer-aided row-reduction) and find that although it is entirely feasible to compute \(\lambda\), the combinatorial complexity grows rapidly.  Interestingly, in every instance where orbit data \((\ell,m,n)\) satisfy  $\ell, m ,n>1$ and $(\ell,m,n) \notin \{(2,n,n), n\ge 4\}$, the homotopy-action growth rate \(\log\lambda\) falls strictly below the maximal possible entropy of the corresponding complex automorphism.
\end{rem}

\section{Irreducible Outer Automoprhism}\label{S:pseudo}

As observed in Lemma \ref{L:Mdet} and Theorem \ref{T:positiveEntropy}, one can extract a subset of the invariant-semigroup generators that in fact generates the entire fundamental group $\pi_1(X_{m,n})$. 

Let
\[ \mathcal{G} \ = \ \{w_1, w_2,\dots, w_{1+m+n} \} \quad \text{with} \quad w_1 = g_\star,\]
where $g_\star$ is the distinguished core generator from \eqref{E:gstar} and each $w_i$ lies in our invariant semigroup generating set. 

\medskip
Since $\mathcal{G}$ generates $\pi_1(X_{m,n})$, every semigroup generator $g\in S_{m,n}$ can be expressed (and then cyclically reduced) as a word in the $w_i^{\pm 1}$. Moreover, by construction, each semigroup generator $g$ is already cyclically reduced and--depending on the orbit-data family--begins with $a_1$ (or ends with $c_1$  in $(1,n+1,n)$ and $(1,m,n)$ with $m\ge n+2$ cases). It follows that its reduced expression in the $\mathcal{G}$-alphabet must also start and end with a positive power of some generator. Concretely, one has
\begin{equation}\label{E:gform} g\ = \ w_{i_1} w_{i_2}^{\pm 1} \cdots w_{i_{k-1}}^{\pm 1} w_{i_k},\end{equation}
with both $w_{i_1}$ and $w_{i_k}$ occurring with positive exponent.

In other words, the cyclic reduction of $g$ in the alphabet $\{w_i\}$ can be written so that it never begins or ends with an inverse letter. This ``endpoint positivity" will be crucial when we analyze how $F_*$ acts on these words. As before, for the remainder of the paper we view each cyclically reduced word as a periodic bi-infinite sequence via the natural boundary compactification of the free group.

\medskip
For each generator $g$ of the invariant semigroup $S_{m,n}$, let $g^\mathcal{G}$ be the reduced word in the form \eqref{E:gform}.

\begin{lem}\label{L:nocan}
Let \(S_{m,n}\subset\pi_1(X_{m,n})\) be the invariant semigroup with semigroup-period \(k\), and let
\(\mathcal G=\{w_1,\dots,w_{1+m+n}\}\subset S_{m,n}\)
be the unimodular generating subset from Lemma \ref{L:Mdet}, with \(w_1=g_\star\).  Then:
\begin{enumerate}
  \item The full group \(\pi_1(X_{m,n})\) is generated by \(\mathcal G\).
  \item For each \(w_i\in\mathcal G\), write
    \[
     F_*^k(w_i)
      \;=\;
      g_{i_1}\,g_{i_2}\,\cdots\,g_{i_s},
    \]
    with each \(g_{i_j}\in S_{m,n}\).  Replace each semigroup-generator \(g_{i_j}\) by its reduced word in the \(\mathcal G\)-alphabet,
    \[
      g_{i_j}
      = 
      g_{i_j}^{\mathcal G}(w_1,\dots,w_{1+m+n})
      \;=\;
      w_{q_1}\,w_{q_2}^{\pm1}\cdots w_{q_t}\,,
    \]
    so that
    \[
      F^k_{\mathcal G*}(w_i)
      \;:=\;
      g_{i_1}^{\mathcal G}\,g_{i_2}^{\mathcal G}\,\cdots\,g_{i_s}^{\mathcal G}
      \;\in\;
      \mathcal F\bigl(\mathcal G\bigr),
    \]
    is a well-defined reduced word in the free group on \(\mathcal G\).
\end{enumerate}
Moreover, in forming each concatenation
\(\;g_{i_1}^{\mathcal G}\,g_{i_2}^{\mathcal G}\,\cdots\,g_{i_s}^{\mathcal G},\)
no cancellation ever occurs between adjacent factors.
\end{lem}

\begin{proof}
By construction (Lemma \ref{L:Mdet}) the subset \(\mathcal G\) generates \(\pi_1(X_{m,n})\).  Hence each semigroup-generator \(g_{i_j}\in S_{m,n}\) can be written as a reduced word
\[
  g_{i_j}
  = 
  g_{i_j}^{\mathcal G}
  = w_{q_1}\,w_{q_2}^{\pm1}\cdots w_{q_t}
  \quad\text{with both endpoints positive powers,}
\]
so that in particular the first and last letters of \(g_{i_j}^{\mathcal G}\) are some \(w_{q_1}\) and \(w_{q_t}\), each with exponent \(+1\).

Now when we concatenate
\[
  g_{i_1}^{\mathcal G}
  \,\cdot\,
  g_{i_2}^{\mathcal G}
  \,\cdots\,
  g_{i_s}^{\mathcal G},
\]
the endpoint-positivity of each factor guarantees that the last letter of \(g_{i_j}^{\mathcal G}\) is positive and the first letter of \(g_{i_{j+1}}^{\mathcal G}\) is also positive.  Therefore no cancellation can occur at any junction between factors.  

It follows that the rule
\[
  F^k_{\mathcal G*}(w_i)
  \;=\;
  g_{i_1}^{\mathcal G}\,g_{i_2}^{\mathcal G}\,\cdots\,g_{i_s}^{\mathcal G}
\]
indeed produces a reduced word in the free group on \(\mathcal G\), and so defines a well-defined induced action on the free group generators without need for further reduction.
\end{proof}

\begin{prop}\label{P:mgprimitive}
Let \(M_{\mathcal G}\) be the \(d\times d\) transition matrix for the free-group action
\[
  F^k_{\mathcal{G}*}\; : \; \mathcal{F}(\mathcal G) \; \to \; \mathcal{F}(\mathcal G),
\]
where \(\mathcal G=\{w_1,\dots,w_d\}\).  Its \((i,j)\) entry \((M_{\mathcal G})_{i,j}\) is the number of occurrences of \(w_i^{\pm1}\) in the reduced word for \(F^k_{\mathcal{G}*}(w_j)\).  Then \(M_{\mathcal G}\) is irreducible and primitive.
\end{prop}

\begin{proof}
Recall from Proposition~\ref{P:primitive} that the semigroup-transition matrix \(M\) (in the \(S_{m,n}\)-basis) was irreducible and aperiodic.  By Lemma~\ref{L:nocan}, passing to the free-group basis \(\mathcal G\) simply amounts to inserting extra positive factors but does not introduce any zero-blocks that could destroy strong connectivity.  Hence \(M_{\mathcal G}\) remains irreducible.

Moreover, by definition \(w_1 = g_\star\) is the ``distinguished" generator, and we saw in \eqref{E:gstar} that
\[
  F_*^k(w_1)\;\supset\;w_1
  \quad\Longrightarrow\quad
  (M_{\mathcal G})_{1,1}>0.
\]
An irreducible nonnegative matrix with a positive diagonal entry is automatically primitive.  Therefore \(M_{\mathcal G}\) is both irreducible and primitive.
\end{proof}

Let us recall the result of Bestvina-Handel \cite{Bestvina-Handel:1992}.

\begin{thm}\cite[Theorem~4.1]{Bestvina-Handel:1992}\label{T:bh}
Suppose $\mathcal{O}$ is an outer automorphism on a free group $F_n$ such that $\mathcal{O}^\ell$ is irreducible for all $\ell>0$ and that there is a cyclic words $s\in F_n$ such that $\mathcal{O}(s) = s$ or $\mathcal{O}(s) = \overline{s}$. Then $\mathcal{O}$ is geometrically realized by a pseudo-Anosov homeomorphism $h:M \to M$ of a compact surface with one boundary component.
\end{thm}

Our last result is the direct application of the Theorem of Bestvina-Handel 
\begin{thm}\label{T:pseudo}
For every $m,n$ with $m+n$ odd, the map $f_{m,n}$ is pseudo-Anosov.
\end{thm}

\begin{proof}
Let $k$ be the semigroup period from Section \ref{S:positivity}. By Proposition \ref{P:mgprimitive}, the transition matrix $M_\mathcal{G}$ for the distinguished generating set $\mathcal{G}$ is irreducible and primitive. Hence for every $\ell\ge 1$, the power $(M _\mathcal{G})^\ell$ is again irreducible. Equivalently, the outer automorphism $(F^k_*)^\ell$ of $\pi_1(X)$ is irreducible for all $\ell>0$.

On the other hand, Corollary \ref{C:boundary} guarantees the existence of a cyclic word $s$ (coming from the invariant cubic or its double) so that
\[F^k_*(s) = s,\]
i.e. there is a non-trivial conjugacy class in $\pi_1$ fixed by $F^k_*$.

By Bestvina-Handel's theorem (\cite[Theorem~4.1]{Bestvina-Handel:1992}), any outer automorphism of a free group which is both irreducible (no nontrivial free factor is invariant under any positive power) and has at least one nontrivial periodic conjugacy class must be realized by a pseudo-Anosov homeomorphism of a surface with boundary. Since $1+m+n$ is even,  the cut-surface $X_{m,n}$ has a single boundary component, so Bestvina-Handel gives a pseudo-Anosov $h: X_{m,n} \to X_{m,n}$ with $h_* = F_*^k$.

Since a pseudo-Anosov hoeomorphism $h$ in Bestvina-Handel's theorem satisfies $\mathcal{O} = h_*|_{\pi_1(M)}$, by the Dehn-Nielson-Bare theorem and its version \cite[Theorem~8.8]{Farb-Margalit:2012} for the surfaces with boundary, we conclude that $F^k$ is pseudo-Anosov and thus $f$ itself is also pseudo-Anosov. 
\end{proof}

\begin{proof}[Proof of Thorem C]
The result of Theorem \ref{T:pseudo} proves Thoerem C.
\end{proof}

\bibliographystyle{plain}
\bibliographystyle{unsrt}
\bibliography{biblio}

\begin{thebibliography}{10}

\bibitem{Bedford-Kim:2004}
Eric Bedford and Kyounghee Kim.
\newblock On the degree growth of birational mappings in higher dimension.
\newblock {\em J. Geom. Anal.}, 14(4):567--596, 2004.

\bibitem{Bedford-Kim:2006}
Eric Bedford and Kyounghee Kim.
\newblock Periodicities in linear fractional recurrences: degree growth of
  birational surface maps.
\newblock {\em Michigan Math. J.}, 54(3):647--670, 2006.

\bibitem{Bestvina-Handel:Tits}
Mladen Bestvina, Mark Feighn, and Michael Handel.
\newblock The {T}its alternative for {${\rm Out}(F_n)$}. {I}. {D}ynamics of
  exponentially-growing automorphisms.
\newblock {\em Ann. of Math. (2)}, 151(2):517--623, 2000.

\bibitem{Bestvina-Handel:1992}
Mladen Bestvina and Michael Handel.
\newblock Train tracks and automorphisms of free groups.
\newblock {\em Ann. of Math. (2)}, 135(1):1--51, 1992.

\bibitem{Birman-Series}
Joan~S. Birman and Caroline Series.
\newblock Algebraic linearity for an automorphism of a surface group.
\newblock {\em J. Pure Appl. Algebra}, 52(3):227--275, 1988.

\bibitem{Blanc:2008}
J\'{e}r\'{e}my Blanc.
\newblock On the inertia group of elliptic curves in the {C}remona group of the
  plane.
\newblock {\em Michigan Math. J.}, 56(2):315--330, 2008.

\bibitem{blancdynamical}
J{\'e}r{\'e}my Blanc and Serge Cantat.
\newblock Dynamical degrees of birational transformations of projective
  surfaces.
\newblock {\em J. Amer. Math. Soc.}, 29(2):415--471, 2016.

\bibitem{Bowen}
Rufus Bowen.
\newblock Entropy and the fundamental group.
\newblock In {\em The structure of attractors in dynamical systems ({P}roc.
  {C}onf., {N}orth {D}akota {S}tate {U}niv., {F}argo, {N}.{D}., 1977)}, volume
  668 of {\em Lecture Notes in Math.}, pages 21--29. Springer, Berlin, 1978.

\bibitem{Cantat:1999}
Serge Cantat.
\newblock Dynamique des automorphismes des surfaces projectives complexes.
\newblock {\em Comptes Rendus de l'Acad{\'e}mie des Sciences-Series
  I-Mathematics}, 328(10):901--906, 1999.

\bibitem{Coble:1939}
A.~B. Coble.
\newblock Cremona transformations with an invariant rational sextic.
\newblock {\em Bull. Amer. Math. Soc.}, 45(4):285--288, 1939.

\bibitem{Coh}
Marshall Cohen and Martin Lustig.
\newblock Paths of geodesics and geometric intersection numbers. {I}.
\newblock In {\em Combinatorial group theory and topology ({A}lta, {U}tah,
  1984)}, volume 111 of {\em Ann. of Math. Stud.}, pages 479--500. Princeton
  Univ. Press, Princeton, NJ, 1987.

\bibitem{Diller:2011}
Jeffrey Diller.
\newblock Cremona transformations, surface automorphisms, and plain cubics.
\newblock {\em Michigan Math. J. The Michigan Mathematical Journal},
  60(2):409--440, 2011.

\bibitem{Diller-Favre:2001}
Jeffrey Diller and Charles Favre.
\newblock Dynamics of bimeromorphic maps of surfaces.
\newblock {\em Amer. J. Math.}, 123(6):1135--1169, 2001.

\bibitem{Diller-Kim}
Jeffrey Diller and Kyounghee Kim.
\newblock Entropy of real rational surface automorphisms.
\newblock {\em Experimental Mathematics}, 30(2):172--190, 2021.

\bibitem{Farb-Margalit:2012}
Benson Farb and Dan Margalit.
\newblock {\em A primer on mapping class groups}, volume~49 of {\em Princeton
  Mathematical Series}.
\newblock Princeton University Press, Princeton, NJ, 2012.

\bibitem{Floyd:1980}
William~J. Floyd.
\newblock Group completions and limit sets of {K}leinian groups.
\newblock {\em Invent. Math.}, 57(3):205--218, 1980.

\bibitem{Gromov}
Mikha{\"{\i}}l Gromov.
\newblock On the entropy of holomorphic maps.
\newblock {\em Enseign. Math. (2)}, 49(3-4):217--235, 2003.

\bibitem{Hamidi-Chen:1996}
Hessam Hamidi-Tehrani and Zong-He Chen.
\newblock Surface diffeomorphisms via train-tracks.
\newblock {\em Topology Appl.}, 73(2):141--167, 1996.

\bibitem{Kim:2023}
Kyounghee Kim.
\newblock The dynamical degrees of rational surface automorphisms.
\newblock {\em Conform. Geom. Dyn.}, 28:62--87, 2024.

\bibitem{Kim-Klassen}
Kyounghee Kim and Eric~P. Klassen.
\newblock Entropy of real rational surface autormophisms: Actions on the
  fundamental groups.
\newblock {\em Geometriae Dedicata}, 218(2):Paper No, 37, 40pp, 2024.

\bibitem{Manning:1975}
Anthony Manning.
\newblock Topological entropy and the first homology group.
\newblock In {\em Dynamical systems---{W}arwick 1974 ({P}roc. {S}ympos. {A}ppl.
  {T}opology and {D}ynamical {S}ystems, {U}niv. {W}arwick, {C}oventry,
  1973/1974; presented to {E}. {C}. {Z}eeman on his fiftieth birthday)},
  Lecture Notes in Math., Vol. 468, pages 185--190. Springer, Berlin, 1975.

\bibitem{McMullen:2007}
Curtis~T McMullen.
\newblock Dynamics on blowups of the projective plane.
\newblock {\em Publications Mathematiques de l'Institut des Hautes Etudes
  Scientifiques}, (105):49--90, 2007.

\bibitem{Meyer:2023}
Carl~D. Meyer.
\newblock {\em Matrix analysis and applied linear algebra}.
\newblock Society for Industrial and Applied Mathematics (SIAM), Philadelphia,
  PA, second edition, [2023] \copyright 2023.

\bibitem{Nagata}
Masayoshi Nagata.
\newblock On rational surfaces. {I}. {I}rreducible curves of arithmetic genus
  {$0$}\ or {$1$}.
\newblock {\em Mem. Coll. Sci. Univ. Kyoto Ser. A Math}, 32:351--370, 1960.

\bibitem{Nagata2}
Masayoshi Nagata.
\newblock On rational surfaces. {II}.
\newblock {\em Mem. Coll. Sci. Univ. Kyoto Ser. A Math.}, 33:271--293,
  1960/1961.

\bibitem{Nicholls-Scherich-Shneidman:2023}
Sarah~Ruth Nicholls, Nancy Scherich, and Julia Shneidman.
\newblock Large 1-systems of curves in nonorientable surfaces.
\newblock {\em Involve}, 16(1):127--139, 2023.

\bibitem{Putman:2022}
Andrew Putman.
\newblock The commutator subgroups of free groups and surface groups.
\newblock {\em Enseign. Math.}, 68(3-4):389--408, 2022.

\bibitem{Seneta:1981}
E.~Seneta.
\newblock {\em Nonnegative matrices and {M}arkov chains}.
\newblock Springer Series in Statistics. Springer-Verlag, New York, second
  edition, 1981.

\bibitem{Thurston:1988}
William~P. Thurston.
\newblock On the geometry and dynamics of diffeomorphisms of surfaces.
\newblock {\em Bull. Amer. Math. Soc. (N.S.)}, 19(2):417--431, 1988.

\bibitem{Uehara:2010}
Takato Uehara.
\newblock Rational surface automorphisms with positive entropy.
\newblock {\em Ann. Inst. Fourier (Grenoble)}, 66(1):377--432, 2016.

\bibitem{Yomdin}
Y.~Yomdin.
\newblock Volume growth and entropy.
\newblock {\em Israel J. Math.}, 57(3):285--300, 1987.

\end{thebibliography}
\end{document}